\documentclass{article}

\usepackage{amsmath,amsfonts,amssymb,amsthm}

\usepackage{pifont}
\usepackage{braket}
\usepackage{comment}
\usepackage{stmaryrd}
\usepackage{adjustbox}
\usepackage{bbm}
\usepackage{enumerate}
\usepackage[bookmarks=false,pdfstartview={FitH}]{hyperref}
\usepackage{MnSymbol}
\usepackage{tablefootnote}
\usepackage{lscape}
\usepackage{mathtools}
\usepackage{adjustbox}

\newtheorem{theorem}{Theorem}[section]
\newtheorem{corollary}[theorem]{Corollary}
\newtheorem{lemma}[theorem]{Lemma}
\newtheorem{proposition}[theorem]{Proposition}
\newtheorem{example}[theorem]{Example}
\newtheorem{Definition}[theorem]{Definition}
\newtheorem{remark}[theorem]{Remark}

\newcommand{\thistheoremname}{}
\newtheorem{genericthm}[theorem]{\thistheoremname}

\newcommand{\iu}{\mathrm{i}\mkern1mu}
\newcommand{\ju}{\mathrm{j}\mkern1mu}
\newcommand{\ku}{\mathrm{k}\mkern1mu}

\DeclarePairedDelimiter{\floor}{\lfloor}{\rfloor}

\newcommand{\N}{\mathbb{N}}
\newcommand{\Z}{\mathbb{Z}}
\newcommand{\R}{\mathbb{R}}
\newcommand{\C}{\mathbb{C}}
\renewcommand{\H}{\mathbb{H}}
\newcommand{\K}{\mathbb{K}}

\newcommand{\fg}{\mathfrak{g}}
\newcommand{\fk}{\mathfrak{k}}
\newcommand{\fp}{\mathfrak{p}}
\newcommand{\fz}{\mathfrak{z}}

\newcommand{\SU}{\operatorname{SU}}
\newcommand{\SO}{\operatorname{SO}}
\newcommand{\GL}{\operatorname{GL}}
\newcommand{\SL}{\operatorname{SL}}
\newcommand{\Sp}{\operatorname{Sp}}
\newcommand{\U}{\operatorname{U}}

\newcommand{\su}{\mathfrak{su}}
\newcommand{\so}{\mathfrak{so}}

\newcommand{\mf}[1]{\mathfrak{#1}}

\newcommand{\mb}[1]{\mathbb{#1}}
\newcommand{\mbf}[1]{\mathbf{#1}}

\newcommand{\hatad}{\widehat{\operatorname{ad}}}

\newcommand{\down}{\shortdownarrow}
\newcommand{\myparallel}{{\mkern3mu\vphantom{\perp}\vrule depth 0pt\mkern2mu\vrule depth 0pt\mkern3mu}}
\renewcommand{\parallel}{\myparallel}

\newcommand{\todo}[1]{
\textcolor{blue}{TODO: #1}
}

\newcommand{\degen}{\mathsf{deg}}

\newcommand{\Int}{\operatorname{Int}}
\renewcommand{\ker}{\operatorname{ker}}
\newcommand{\img}{\operatorname{im}}

\newcommand{\id}{\operatorname{id}}

\newcommand{\Tr}{\operatorname{Tr}}

\newcommand{\Ad}{\operatorname{Ad}}
\newcommand{\ad}{\operatorname{ad}}
\newcommand{\tr}{\operatorname{tr}}
\newcommand{\Lie}{\operatorname{Lie}}
\newcommand{\Iso}{\operatorname{Iso}}

\renewcommand{\epsilon}{\varepsilon}

\title{Analytic, Differentiable and Measurable Diagonalizations in Symmetric Lie Algebras}                                     
\author{Emanuel Malvetti$^{a,b}$, Gunther Dirr$^c$, Frederik vom Ende$^d$, \\Thomas Schulte-Herbr\"uggen$^{a,b}$}                

\begin{document}


\maketitle

{\centering
\footnotesize{$^a$School of Natural Sciences,
Technische Universit\"at M\"unchen,
85737 Garching,
Germany}\\
\footnotesize{$^b$Munich Centre for Quantum Science and Technology (MCQST) \& Munich Quantum Valley (MQV),
80799 M{\"u}nchen,
Germany}\\
\footnotesize{$^c$Department of Mathematics,
University of W{\"u}rzburg,
97074 W{\"u}rzburg,
Germany}\\
\footnotesize{$^d$Dahlem Center for Complex Quantum Systems,
Freie Universit{\"a}t Berlin,
14195 Berlin,
Germany}}


\begin{abstract}
We generalize several important results from the perturbation theory of linear operators to the setting of semisimple orthogonal symmetric Lie algebras.
These Lie algebras provide a unifying framework for various notions of matrix diagonalization, such as the eigenvalue decomposition of real symmetric or complex Hermitian matrices, and the real or complex singular value decomposition.
Concretely, given a path of structured matrices with a certain smoothness, we study what kind of smoothness one can obtain for the corresponding diagonalization of the matrices.
\end{abstract}

\section{Introduction}

\subsection{Motivation}

Matrix diagonalizations such as eigenvalue decompositions and singular value decompositions are ubiquitous in pure and applied mathematics. While computing such a decomposition is an elementary task, the result is not unique since the order of the eigenvalues or singular values, as well as the diagonalizing matrices (e.g., the applied unitaries), are not unique. Hence, if we are given a path of Hermitian matrices $\rho(t)$, we might ask if it is possible to diagonalize all $\rho(t)$ in a consistent way. More precisely, if $\rho(t)$ is continuous, measurable, real analytic or $k$ times differentiable, can we choose eigenvalue functions $\lambda_i(t)$ with the same properties? Similarly, what properties can we guarantee for the function $U(t)$ of diagonalizing unitaries? Many of these questions have been answered for the eigenvalue decompositions of real symmetric and complex Hermitian matrices, as well as for singular value decompositions. However the treatment is not uniform, and there exist several other diagonalizations which have not been studied in the same detail. Here we consider symmetric Lie algebras which provide a unifying framework for many notions of diagonalization, see Table~\ref{tab:diags} for some examples. We will answer the questions posed above and several more in this general setting. 

We start by recalling some known results in this direction, which we will then generalize to symmetric Lie algebras. For the symmetric or Hermitian eigenvalue decomposition, many results can be found in~\cite{Baumgaertel85,Kato80,Rellich69}. For instance, if a path of Hermitian matrices is continuous or continuously differentiable, then one can choose the eigenvalues to be continuous or continuously differentiable respectively.
Furthermore~\cite{Kato80} shows that a real analytic path of Hermitian matrices can be diagonalized in a real analytic way. The real singular value decomposition (SVD) is considered in \cite{BunseGerstner91}; there it is shown that a real analytic path has a real analytic SVD, and a smooth path of full-rank matrices with distinct singular values has a smooth SVD.
In~\cite{Quintana14} it is shown that a measurable function of positive definite matrices can be measurably diagonalized.
There are further results which may generalize to symmetric Lie algebras, but which are not treated here.
For example, many results for Hermitian matrices have been extended to the infinite dimensional setting, again see~\cite{Kato80}.
In~\cite{Rainer11} improved results on differentiability are shown for normal and Hermitian matrices, see in particular Table~1 therein.
In~\cite{Mailybaev05} the behaviour of eigenvalues of matrices depending on several variables is studied.

Now let us briefly introduce the notion of a symmetric Lie algebra, which will provide a unifying framework for many diagonalizations. This connection is explored in~\cite{Kleinsteuber} where algorithms for computing such diagonalizations are proposed. A symmetric Lie algebra is a (real, finite dimensional) Lie algebra $\mf g$ together with an involutive Lie algebra automorphism $s$. This yields a vector space decomposition $\mf g=\mf k\oplus\mf p$ into $+1$ and $-1$ eigenspaces of $s$ which we call Cartan-like decomposition since it generalizes the usual Cartan decomposition. Given a connected Lie group $G$ with Lie algebra $\mf g$, and a Lie subgroup $K\subseteq G$ with Lie algebra $\mf k$, we say that the pair $(G,K)$ is associated to the symmetric Lie algebra. 
One can show that the adjoint action of $K$ on $\mf g$ leaves $\mf p$ invariant, an in fact the orbits of $K$ in $\mf p$ do not depend on the choice of $K$. For this reason we can assume without loss of generality that $K$ is connected.
We will only consider symmetric Lie algebras which are semisimple and orthogonal, which implies that the group $\Ad_K$ is compact.
If $\mf a\subseteq\mf p$ is a maximal Abelian subspace, then every point $x\in\mf p$ can be mapped to $\mf a$ by some $U\in K$, that is $\Ad_U(x)\in\mf a$. This generalizes the idea of diagonalization. However the resulting element $\Ad_U(x)\in\mf a$ is not unique, since the elements of $K$ which leave $\mf a$ invariant can act non-trivially on $\mf a$. The resulting group of transformations of $\mf a$ is called the Weyl group, denoted $W,$ and it is a finite group generated by reflections. A convenient fact about Weyl groups is that they admit a (closed) Weyl chamber $\mf w\subseteq\mf a$, such that each orbit $\Ad_K(x)$ intersects $\mf w$ in exactly one point. If this point lies in the interior of $\mf w$, then $x$ is called regular. Note that even if we fix $\Ad_U(x)\in\mf a$, the element $U\in K$ need still not be unique. 

We can now formulate more precisely the questions that we will answer in this paper. Given a path $\rho:I\to\mf p$ with certain nice properties, can we choose a corresponding path $\lambda:I\to\mf a$ with similarly nice properties? How do we deal with the non-uniqueness of $\lambda$ caused by the Weyl group?  What about a corresponding path $U:I\to K$?\medskip

In order to warm up to our setting, let us revisit the eigenvalue decomposition of Hermitian matrices with the above Lie algebraic setting in mind. This example will be very useful in understanding the results presented in this article as it is a well-explored special case.
Therefore we will use it as a running example throughout the paper.

\begin{example}[Hermitian EVD] \label{ex:hermitian-evd}
The semisimple Lie algebra $\mf{sl}(n,\C)$ admits the Cartan decomposition $\mf{sl}(n,\C)=\mf{su}(n)\oplus\mf{herm}_0(n,\C)$ by means of the automorphism $s(X)=-X^*$. This gives it the structure of a semisimple, orthogonal, symmetric Lie algebra, and a possible pair associated to it is given by $(\SL(n,\C), \SU(n))$. Keeping the idea of diagonalization in mind, a convenient choice of a maximal Abelian subspace of $\mf{herm}_0(n,\C)$ (i.e.~the traceless Hermitian $n\times n$ matrices) is the subset of all diagonal matrices. These will automatically be real and traceless; we denote this set by $\mf d_0(n,\R)$. The corresponding Weyl group---which captures the non-uniqueness of the diagonalized element from $\mf d_0(n,\R)$---is isomorphic to the symmetric group $S_n$ acting on $n$ elements. The action on $\mf d_0(n,\R)$ is given by permutation of the diagonal elements of the matrix. Then a natural choice of Weyl chamber is the subset of $\mf d_0(n,\R)$ with the diagonal elements in non-increasing order.
\end{example}

Now let us explore the most elementary example, the orthogonal diagonalization of (traceless) real symmetric $2\times 2$ matrices, in a bit more detail. In particular this turns out to be equivalent to the polar decomposition of $\C$. 

\begin{example}[Polar decomposition] \label{ex:polar-dec}
Consider the semisimple Lie algebra $\mf{sl}(2,\R)$.
Similarly to Example~\ref{ex:hermitian-evd}, the automorphism $s(X)=-X^\top$ yields the Cartan decomposition $\mf{sl}(2,\R)=\mf{so}(2,\R)\oplus\mf{sym}_0(2,\R)$ into the orthogonal Lie algebra and the space of symmetric traceless matrices, and this yields the structure of a semisimple, orthogonal, symmetric Lie algebra. A choice of associated pair is given by $(\SL(2,\R),\SO(2,\R))$. 
Again we choose the diagonal matrices as our maximal Abelian subspace. Consider the identifications
\begin{align*}
\iota_1&:\mf{sym}_0(2,\R)\to\C, \quad \begin{bmatrix}a&b\\b&-a\end{bmatrix}\mapsto a+ib,\\
\iota_2&:\SO(2,\R)\to\U(1), \quad \begin{bmatrix}\cos(\phi)&-\sin(\phi)\\\sin(\phi)&\cos(\phi)\end{bmatrix}\mapsto e^{i2\phi},
\end{align*}
where the first is an $\R$-linear isomorphism and the second is a double cover. Note that $\iota_2$ induces an isomorphism on the quotient $\SO(2,\R)/\{\pm \id\}\to U(1)$.

Either way the chosen maximal Abelian subspace of $\mf{sym}_0(2,\R)$ corresponds (w.r.t.~$\iota_1$) to the real numbers, with the non-negative numbers as an obvious choice of a Weyl chamber. One readily verifies $\iota_1(OAO^\top)=\iota_2(O)\iota_1(A)$, which
shows that the eigenvalue decomposition of real symmetric traceless $2$ by $2$ matrices is equivalent to the polar decomposition of complex numbers.
\end{example}

Interestingly, already in this simple setting many counterexamples can be found.
The nature of these examples is that they violate regularity, that is, problems may occur as soon as the diagonalization does not live \textit{only} in the interior of a Weyl chamber.

\begin{example}[Differentiability of eigenvalues] \label{ex:C2-eigenvalues}
In~\cite[Example~p.~2]{Kriegl03} it is shown that there exists a path $\rho:\R\to\mf{sym}_0(2,\R)\cong\C$ which is $C^\infty$, but the eigenvalues cannot be chosen as $C^2$ functions.
This can only happen because the eigenvalues coincide at some point; such degeneracies corresponds precisely to the boundary of the Weyl chamber of non-increasingly sorted eigenvalues.
However, by~\cite[Thm.~(C)~p.~1]{Kriegl03} the eigenvalues can still be chosen twice differentiable.
\end{example}

\begin{example}[Continuity of diagonalization] \label{ex:no-cont-diag}
The following is Example~5.3 in~\cite{Kato80}, originally due to Rellich. Consider the path $\rho:\R\to\mf{sym}_0(2,\R)\cong\C$ given by
\begin{align*}
\rho(x)=e^{-1/x^2}\begin{bmatrix}
\cos(2/x)&\sin(2/x)\\\sin(2/x)&-\cos(2/x)
\end{bmatrix}
,\quad\rho(0)=0
\end{align*}
This path is $C^\infty$ on $\R$, and so are the eigenvalues $\lambda_{\pm}=\pm e^{-1/x^2}$. However, there does not exist a continuous path of orthogonal matrices diagonalizing $\rho$.
\end{example}

Finally, semisimple, orthogonal, symmetric Lie algebras are closely related to the classification of simple Lie algebras over $\C$ and $\R$. Indeed, this connection allows for a classification of different diagonalizations in a certain irreducible case, and we will show that all diagonalizations considered here are in some sense composed of these irreducible diagonalizations. 

\medskip
\subsection{Outline}


In Section~\ref{sec:continuous} we consider functions $\rho:X\to\mf p$ which are continuous and show that by diagonalizing them in a given Weyl chamber, the result is also continuous. Indeed the same argument extends to stronger forms of continuity, like uniform, H\"older, Lipschitz and absolute continuity, cf.~Proposition~\ref{prop:continuous-diagonalizations}.

In Section~\ref{sec:differentiable} we consider paths $\rho:I\to\mf p$ which are differentiable. In Proposition~\ref{prop:deriv-of-projected-path} we show that if $\rho$ is differentiable at a point $t\in I$, then $\lambda:I\to\mf a$ can be chosen to be differentiable at $t$, and we can give an explicit formula for this derivative. Furthermore, if $\rho$ is (continuously) differentiable on $I$, then $\lambda$ can be chosen to be (continuously) differentiable on $I$, see Theorem~\ref{thm:differentiable-diagonalization}. To prove this, we study (continuously) differentiable paths in orbifolds in Appendix~\ref{app:orbifolds}. 
Moreover we show that if the path $\rho$ only contains regular (i.e. non-degenerate) elements and is $C^k$, then one can indeed find $C^k$ paths $U$ and $\lambda$ diagonalizing $\rho$, see Proposition~\ref{prop:Ck-regular}.

In Section~\ref{sec:analytic} we consider paths $\rho:I\to\mf p$ which are real analytic. In this case one can find $U$ and $\lambda$ real analytic, and moreover $\lambda$ is determined uniquely up to a global Weyl group action. This is in stark contrast to the differentiable case. In the analytic case we can also give a useful differential equation defining $U$. This is done in Theorem~\ref{thm:analytic-diag}.

In Section~\ref{sec:measurable} we consider paths $\rho:\Omega\to\mf p$ which are measurable. Here $\Omega$ can by any measurable space; then we can find $U$ and $\lambda$ measurable, see Theorem~\ref{thm:meas-diag}. In fact we can generalize this result to a finite family of commuting $\rho_i:\Omega\to\mf p$ which we can then simultaneously measurably diagonalize, see Theorem~\ref{thm:meas-diag-comm}. For absolutely continuous paths, this allows us to simultaneously measurably diagonalize the path and a certain projection of the derivative, see Proposition~\ref{prop:abs-cont-diag}.

In Section~\ref{sec:classification} we show how the classification of simple Lie algebras over $\C$ and $\R$ translates to a classification of diagonalizations, and we explain in what sense all semisimple, orthogonal, symmetric Lie algebras are composed of these irreducible ones appearing in the classification. See Theorem~\ref{thm:orbit-equivalence}.

To make the exposition self-contained and to fix terminology give a rigorous treatment of symmetric Lie algebras in Appendix~\ref{app:symmetric}.

\section{Main Results}

Throughout this section we consider a semisimple, orthogonal, symmetric Lie algebra $(\mf g,s)$ with Cartan-like decomposition $\mf g=\mf k\oplus\mf p$ and with an associated pair $(G,K)$ as defined in the introduction and where $K$ is connected. We fix a choice of maximal Abelian subspace $\mf a\subseteq\mf p$ and a (closed) Weyl chamber $\mf w\subseteq\mf a$. The corresponding Weyl group is denoted by $W$. A key geometric fact is that since $(\mf g,s)$ is orthogonal, there exists an inner product on $\mf g$, and hence also on $\mf p$ and $\mf a$, which is invariant under the action of $K$ and $W$ respectively. In particular $\Ad_K$ is a compact Lie group and it acts isometrically on $\mf p$. For precise definitions we refer to Appendix~\ref{app:symmetric}.

\medskip
\subsection{Continuous Functions} \label{sec:continuous}

We start with a natural way to make the diagonalization unique. 
Indeed, a basic fact about Weyl group actions is that they admit a Weyl chamber which intersects every $W$-orbit, and every $K$-orbit\footnote{We always consider the adjoint action of $K$ on $\mf p$, and so we will often shorten $\Ad_K(x)$ to $Kx$.}, in exactly one point, see Lemma~\ref{lemma:triple-bijection}. 
We denote by $\pi:\mf p\to\mf p/ K$ and $\pi_{\mf a}:\mf a\to\mf a/ W$ the respective quotient maps. They are continuous and open.
Consider the following diagram\footnote{Here $\hookrightarrow$ denotes an injection and $\twoheadrightarrow$ denotes a surjection. By $\iota$ we denote the inclusion, which we will often suppress from the notation. In particular we often write $\pi$ instead of $\pi_{\mf a}$.}.
%
\begin{align} \label{cd:global}
\centering
\adjincludegraphics[valign=c,width=0.33\textwidth]{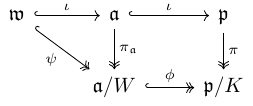}
\end{align}
The maps $\psi(x)=Wx$ and $\phi(Wx)=Kx$ are the unique maps which make the diagram commute. Furthermore one can show that $\psi$ and $\phi$ are in fact isometries, where the quotients $\mf a/ W$ and $\mf p/ K$ are endowed with their quotient metric. This is shown in Lemma~\ref{lemma:triple-isometry}. This crucially uses that all $K$-orbits in $\mf p$ intersect $\mf a$ orthogonally. These facts already suffice to prove some interesting results:

\begin{proposition}
\label{prop:continuous-diagonalizations}
For a given function $\rho$ with values in $\mf p$ we denote by $\lambda^\down=\psi^{-1}\circ\phi^{-1}\circ\pi\circ\rho$ the corresponding function with values in $\mf w$. Then it holds that $\pi\circ\lambda^\down=\pi\circ\rho$ and
\begin{enumerate}[(i)]
\item \label{it:diag-cont} if $\rho:X\to \mf p$ is continuous, then so is $\lambda^\down:X\to\mf w$;
\item \label{it:diag-unif-cont} if $\rho:Y\to \mf p$ is uniformly continuous, then so is $\lambda^\down:Y\to\mf w$;
\item \label{it:diag-Hoelder} if $\rho:Y\to \mf p$ is $\alpha$-H\"older continuous, then so is $\lambda^\down:Y\to\mf w$, with the same constant $\alpha$;
\item \label{it:diag-Lipschitz} if $\rho:Y\to \mf p$ is $L$-Lipschitz continuous, then so is $\lambda^\down:Y\to\mf w$, with the same constant $L$;
\item \label{it:diag-AC} if $\rho:I\to \mf p$ is absolutely continuous, then so is $\lambda^\down:I\to\mf w$.
\end{enumerate}
Here $X$ denotes any topological space, $Y$ any metric space, and $I$ an open interval.
\end{proposition}

\begin{proof}
From the commutativity of Diagram~\eqref{cd:global} it follows that $\pi\circ\lambda^\down=\phi\circ\psi\circ\lambda^\down=\pi\circ\rho$.
The remaining statements follow immediately from the fact that $\pi$ is non-expansive (by definition of the quotient metric, see Lemma~\ref{lemma:quotient-metric}) and the fact that $\phi\circ\psi$ is an isometry (Lemma~\ref{lemma:triple-isometry}).
\end{proof}

\begin{remark}
In the setting of Example~\ref{ex:hermitian-evd}, this result generalizes the idea of~\cite[p.~109]{Kato80} of choosing the eigenvalues continuously by ordering them in non-increasing order.
\end{remark}

Now one might wonder about the existence of a continuous function in $K$ diagonalizing $\rho$, however Example~\ref{ex:no-cont-diag} shows that, even under stronger assumptions, continuity of the diagonalizing group elements cannot be guaranteed. 

\medskip
\subsection{Differentiable Paths} \label{sec:differentiable}

In this section we are interested in differentiable paths $\rho:I\to\mf p$. We already know from Proposition~\ref{prop:continuous-diagonalizations}~\eqref{it:diag-AC} that if $\rho$ is absolutely continuous, then $\lambda$ can also be chosen absolutely continuous, and hence almost everywhere differentiable, simply by choosing $\lambda=\lambda^\down$ to take values in the Weyl chamber $\mf w$. However it is clear that this cannot work to give us $\lambda$ everywhere differentiable, as can be seen by choosing $\rho:I\to\mf a$ differentiable and crossing several distinct Weyl chambers. Forcing $\lambda$ to take values in $\mf w$ would introduce ``kinks'' in the path when $\rho$ passes from one Weyl chamber to a different one.
In this section we show that if $\rho:I\to\mf p$ is (continuously) differentiable, then one can also choose $\lambda:I\to\mf a$ (continuously) differentiable, see Theorem~\ref{thm:differentiable-diagonalization}. Then Example~\ref{ex:C2-eigenvalues} shows that the analogous result for $C^2$ paths does not hold, and Example~\ref{ex:no-cont-diag} shows that there might not even exist a continuous choice of $U:I\to K$ diagonalizing $\rho$. Moreover we show that problems with the differentiability of $\lambda$ only occur at non-regular points. Indeed, Proposition~\ref{prop:Ck-regular} shows that if $\rho$ is $C^k$ and takes regular values, then we can find $C^k$ paths $U$ and $\lambda$ diagonalizing $\rho$.

\medskip
We start with an important geometric fact about the $K$-orbits in $\mf p$. For this we define the commutant of $x$ in $\mf p$ by $\mf p_x=\{y\in\mf p:[x,y]=0\}$. Note that if $x\in\mf a$ then $\mf a\subseteq\mf p_x$ with equality if and only if $x$ is regular.
It turns out that every $K$-orbit in $\mf p$ intersects the maximal Abelian subspace $\mf a$ orthogonally, see Lemma~\ref{lemma:orbit-tangent-centralizer}. More precisely, for $x\in\mf a$, the tangent space $T_x\mf p$ splits into an orthogonal vector space sum of the tangent space to the orbit and $\mf p_x$:
\begin{align*}
T_x\mf p
=
T_x(Kx)\oplus\mf p_x
=
\ad_{\mf k}(x)\oplus\mf p_x,
\end{align*}
where we make liberal use of the identification $T_x\mf p\cong\mf p$. 
We denote the orthogonal projection onto $\mf p_x$ by $\Pi_x : \mf p\to\mf p_x$. Its kernel is then exactly $\ad_{\mf k}(x)$. Similarly we denote by $\Pi_x^\perp=1-\Pi_x$ the orthogonal projection onto $\ad_{\mf k}(x)$ and with kernel $\mf p_x$.

\medskip
To gain some intuition let us consider a path $\rho:I\to\mf p$ which admits a differentiable diagonalization, meaning that there exist differentiable $\lambda:I\to\mf a$ and $U:I\to K$ such that $\rho=\Ad_U(\lambda)$. 

\begin{lemma} \label{lemma:differentiably-diagonalizable}
Let $\lambda:I\to\mf a$ and $U:I\to K$ be differentiable and let $\rho=\Ad_U(\lambda)$. Then\footnote{We use a simplified notation is this lemma and its proof. For instance we use $U'U^{-1}$ as a shorthand for $r_U^\star(U')$, the pull back along the right multiplication $r_U$ by $U$. If $K$ is a matrix Lie group then both of these expressions are well defined and equal.}
\begin{align*}
\rho' = \Ad_U(\lambda') - \ad_{\rho}(U'U^{-1}).
\end{align*}
In particular it must hold that 
\begin{align}
\ad_\rho(U'U^{-1}) &= -\Pi_\rho^\perp\rho' \label{eq:diff-U} \\ 
\lambda' &= \Ad_U^{-1}(\Pi_\rho\rho') = \Pi_\lambda(\Ad_U^{-1}(\rho')). \label{eq:diff-lambda}
\end{align}
\end{lemma}

\begin{proof}
The first statement follows from a simple computation. Recall that $(U^{-1})'=-U^{-1}U'U^{-1}$. Then
\begin{align*}
\rho'
=(U\lambda U^{-1})'
=U\lambda' U^{-1}+U'\lambda U^{-1}-U\lambda U^{-1}U'U^{-1} 
=U\lambda' U^{-1}+[U'U^{-1},\rho].
\end{align*}
Conveniently, the two terms on the right-hand side respect the orthogonal splitting of $\mf p$ into kernel and image of $\ad_\rho$, since $[\rho,\Ad_U(\lambda')]=\Ad_U([\lambda,\lambda'])=0$, which proves the second statement. For the last equality we used Lemma~\ref{lemma:equivariance}~\eqref{it:equiv-proj}.
\end{proof}

Lemma~\ref{lemma:differentiably-diagonalizable} already tells us much about the structure of differentiable paths $\rho:I\to\mf p$, however it has some problems. It might seem that given $\rho$, we can find $U$ by solving a differential equation obtained from~\eqref{eq:diff-U}, and then we can determine $\lambda$ by solving~\eqref{eq:diff-lambda}. Unfortunately, even if $\rho$ is $C^\infty$, there might not exist a diagonalizing function $U$ which is continuous, see Example~\ref{ex:no-cont-diag}. Another problem is that in general the right-hand side of~\eqref{eq:diff-lambda} need not lie in $\mf a$. Nevertheless, in Proposition~\ref{prop:deriv-of-projected-path} we will show that if $\rho$ is differentiable at a point, then $\lambda$ can also be chosen differentiable at that point, and~\eqref{eq:diff-lambda} will return in a slightly modified form. Eq.~\eqref{eq:diff-U} will return in Section~\ref{sec:analytic} where the much stronger condition of $\rho$ being real analytic will guarantee that the solution $U$ exists (and is itself real analytic). Similarly in Proposition~\ref{prop:Ck-regular} we will use~\eqref{eq:diff-U} to show that for regular $C^k$ paths we can find a $C^k$ diagonalization.

\medskip
Before we can prove the main results of this section we need to introduce some technical tools. For a point $x\in\mf p$ we denote by $K_x$ the stabilizer (also called isotropy subgroup) of $x$ in $K$. Similarly for $y\in\mf a$ we write $W_y$ for the stabilizer of $y$ in $W$. With this we can define a number of quotient spaces. The details of the following facts can be found in Appendix~\ref{app:sym-quotients}. 

For $x\in\mf a$, there exists a homeomorphism $\phi_x:\mf a/ W_x \cong \mf p_x/ K_x$, given by $W_x z\mapsto K_x z$, which is induced by the inclusion of $\mf a$ in $\mf p_x$.\footnote{In particular, setting $x=0$ one gets the homeomorphism $\phi:\mf a/ W\to \mf p/ K$ which we used in Section~\ref{sec:continuous}.} Furthermore it holds that if $y\in\Ad_K(x)$ then there is a well-defined homeomorphism $\phi_{x,y}:\mf p_x/ K_x\to \mf p_y/ K_y$ induced by any $U\in K$ with $\Ad_U(x)=y$.
Summarizing one can say that the diagram
\begin{align} \label{cd:local}
\centering
\adjincludegraphics[valign=c,width=0.42\textwidth]{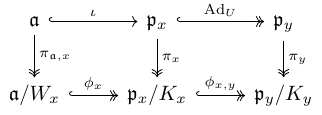}
\end{align}
commutes. 

Although the quotients encountered here have singularities and hence are not manifolds, they can still be given the structure of an orbifold. In fact the orbifolds that we deal with will have a single linear chart. The relevant facts about such orbifolds are proven in Appendix~\ref{app:orbifolds}. 
In order to find a differentiable path $\lambda:I\to\mf a$, we need to be able to make sense of differentiable paths in such orbifolds. 
We start by defining the tangent bundle $T(\mf a/W):=(T\mf a)/W$, where the action of $W$ on $T\mf a$ is given by $w\cdot(x,v)=(w\cdot x,w\cdot v)$. We denote the corresponding quotient map by
\begin{align}
D\pi_{\mf a}: T\mf a \to T(\mf a/ W)
\end{align}

If $x\in\mf a$, then the tangent space in $\mf a/ W$ at the point $\pi_{\mf a}(x)$ is denoted by $T_{\pi_{\mf a}(x)}(\mf a/ W)$ and it turns out to be homeomorphic to $(T_x\mf a)/ W_x$. Since one can canonically identify $T_x\mf a$ and $\mf a$, the commutative Diagram~\eqref{cd:local} shows that we have the homeomorphisms
\begin{align*}
T_{\pi_{\mf a}(x)}(\mf a/ W) \cong \mf a/ W_x \cong \mf p_x/ K_x.
\end{align*}
Hence we can define the differential of the quotient map $\pi_{\mf a}$ at a point $x$ as the map
\begin{align*}
D\pi_{\mf a}(x): T_x\mf a \to T_{\pi_{\mf a}(x)}(\mf a/ W), \quad v\mapsto\pi_{\mf a,x}(v).
\end{align*}

Let us briefly recall what it means for a path $\xi:I\to\mf a/W$ to be differentiable in the orbifold sense, as defined in Definition~\ref{def:orbifold-derivatives}. We say that $\xi$ is differentiable at $t\in I$ if there exists a function $\lambda:I\to\mf a$ satisfying $\pi_{\mf a}\circ\lambda=\xi$, called a lift of $\xi$, which is differentiable at $t$. The derivative of $\xi$ at $t$ is given by $D\xi(t):=D\pi_{\mf a}(\lambda(t),\lambda'(t))$ and it is well-defined. If $\xi$ is differentiable at every $t\in I$, then we say that $\xi$ is differentiable and if additionally $D\xi:I\to T(\mf a/W)$ is continuous, then $\xi$ is continuously differentiable or $C^1$. In the following proofs we will show that if $\rho$ is (continuously) differentiable, then so is $\xi:=\phi^{-1}\circ\pi\circ\rho$. Then we use Proposition~\ref{prop:orbifold-C1-lift} to show that there exists a (continuously) differentiable lift $\lambda:I\to\mf a$ of $\xi$.

The following lemmas will be quite useful. For the proof of the first lemma we use the concept of a ``slice'' for the action of $K$ on $\mf p$. A slice at a point $x$ is an embedded submanifold containing $x$ and intersecting the orbit through $x$ in a complementary way, see Definition~\ref{def:slice}. Such slices exist in very general settings, but in our case we can even choose a slice in $\mf p_x$ intersecting the orbit orthogonally. The main idea is then to ``project'' the path $\rho$ onto the slice while keeping each point in its original orbit. 

\begin{lemma} \label{lemma:project-to-slice} 
Let $I$ be an open interval and let $\rho:I\to\mf p$ be differentiable at $t_0\in I$. 
Then there exists $U\in K$ such that $x:=\Ad_U^{-1}(\rho(t_0))\in\mf a$ and such that $v:=\Ad_U^{-1}(\Pi_{\rho(t_0)}(\rho'(t_0)))\in\mf a$.
Moreover there is a subinterval $I'\subseteq I$ containing $t_0$ and a function $\tilde\rho:I'\to\mf p_x$ satisfying $\tilde\rho(t_0)=x$ and $\tilde\rho'(t_0)=v$ and $\pi\circ\tilde\rho=\pi\circ\rho$ on $I'$.
\end{lemma}

\begin{proof}
By definition of the projection, $\rho(t_0)$ and $\Pi_{\rho(t_0)}(\rho'(t_0))$ commute, and hence by Lemma~\ref{lemma:max-abelian-conj} there is some $U\in K$ such that $x=\Ad_U^{-1}(\rho(t_0))\in\mf a$ and $v=\Ad_U^{-1}(\Pi_{\rho(t_0)}(\rho'(t_0)))\in\mf a$. By the chain rule $\Ad_U^{-1}(\rho)$ is differentiable at $t_0$ and by linearity of $\Ad_U$ it holds that $(\Ad_U^{-1}\circ\rho)'(t_0)=\Ad_U^{-1}(\rho'(t_0))$.

By Lemma~\ref{lemma:slice} there exists a slice $S_x$ at $x$ for the action of $K$ on $\mf p$, which is contained in $\mf p_x$. Let $\mf k_x:=\mf k\cap \ker\ad_x$ and let $\mf k_x^\perp$ be the orthogonal complement of $\mf k_x$ in $\mf k$. 
Let $O$ be an open neighborhood of the origin in $\mf k_x^\perp$ and consider the map $\sigma:O\times S_x \to \mf p : (k,y)\mapsto \Ad_{e^k}(y)$.
Since
\begin{align} \label{eq:Dsigma}
D\sigma(0,x):\mf k^\perp_x\oplus\mf p_x\to\mf p, \quad (l,z)\mapsto[l,x]+z,
\end{align}
Lemma~\ref{lemma:inner-prod-centralizer} shows that $D\sigma(0,x)$ bijective and by the inverse function theorem, and potentially shrinking $O$ and $S_x$, we may assume that $\sigma$ is a diffeomorphism onto its image, denoted $V$.
Hence $x\in V$ and $\sigma$ can be seen as a chart for $V$. On $V$ we define the smooth map $\kappa=\sigma\circ\operatorname{pr}_2\circ\,\sigma^{-1}$. Then $\kappa(x)=x$ and~\eqref{eq:Dsigma} shows that $D\kappa(x)=\Pi_x$.
By continuity of $\rho$ at $t_0$, there is an open interval $I'\subseteq I$ containing $t_0$ such that the image of $\Ad_U^{-1}(\rho)$ on $I'$ lies in $V$. 
Set $\tilde\rho=\kappa\circ\Ad_U^{-1}\circ\rho|_{I'}$, then $\pi\circ\tilde\rho=\pi\circ\rho$ on $I'$, and $\tilde\rho(t_0) = x$, and $\tilde\rho'(t_0)=D\kappa(x)(\Ad_U^{-1}(\rho'(t_0))=\Ad_U^{-1}(\Pi_{\rho(t_0)}(\rho'(t_0)))=v$ by Lemma~\ref{lemma:equivariance}~\eqref{it:equiv-proj}.
\end{proof}

For the second lemma we use Diagram~\eqref{cd:local} as well as the fact that the stabilizer subgroup $W_x$ still has the properties of a Weyl group and hence admits a (closed) Weyl chamber, which we denote $\tilde{\mf w}$.

\begin{lemma} \label{lemma:project-to-chamber}
Let $x\in\mf a$ be given and let $\tilde{\mf w}\subseteq\mf a$ be a Weyl chamber for the action of $W_x$ on $\mf a$. Then
\begin{enumerate}[(i)]
\item\label{it:cont-proj} there is a continuous map $\omega:\mf p_x\to\tilde{\mf w}$ satisfying $\pi_x\circ\omega=\pi_x$,
\item\label{it:sequence} for any sequence $y_n$ in $\mf p_x$ converging to some $y\in\tilde{\mf w}$ there is a subsequence $y_n'$ and a sequence $U_n'\in K_x$ such that $\Ad_{U_n'}^{-1}(y_n')\in\tilde{\mf w}$ converge to $y$ and there is some $U\in K_x\cap K_y$ such that $\Ad_{U_n'}\to\Ad_{U}$.
\end{enumerate}
\end{lemma}

\begin{proof}
Recall from Corollary~\ref{coro:homeo-a-p} that we have a homeomorphism $\phi_x:\mf a/ W_x\to\mf p_x/ K_x$ induced by the inclusion $\mf a\hookrightarrow\mf p_x$. By Lemma~\ref{lemma:Weyl-stabilizer} the action of $W_x$ on $\mf a$ admits a (closed) Weyl chamber $\tilde{\mf w}$. The proof of Lemma~\ref{lemma:triple-isometry} also yields a homeomorphism $\psi_x:\tilde{\mf w}\to\mf a/W_x$ induced by the inclusion $\tilde{\mf w}\hookrightarrow\mf a$. Combining this we define
\begin{align*}
\omega=\psi_x^{-1}\circ\phi_x^{-1}\circ\pi_x.
\end{align*}
Since $\psi_x=\pi_{\mf a,x}$ on $\tilde{\mf w}$ and $\phi_x\circ\,\pi_{\mf a,x}=\pi_x$ on $\mf a$, as can be read off from the corresponding commutative diagram, it holds that $\pi_x\circ\omega=\pi_x$. This proves~\eqref{it:cont-proj}.

This shows that for every element $y_n\in\mf p_x$ there is some $U_n\in K_x$ with $\Ad_{U_n}^{-1}(y_n)\in\tilde{\mf w}$. The same point in $\tilde{\mf w}$ can be obtained using the continuous map $\omega$ applied to $y_n$. Since $y\in\tilde{\mf w}$ it holds that
\begin{align*}
\Ad_{U_n}^{-1}(y_n)=\omega(y_n)\to\omega(y)=y.
\end{align*}
The existence of a subsequence $U_n'$ with the desired properties follows from the compactness of $\Ad_K$. This proves~\eqref{it:sequence}.
\end{proof}

\begin{corollary} \label{coro:lambda-def}
Let $x\in\mf a$ be given and let $\tilde\rho:I'\to\mf p_x$ be differentiable at $t_0\in I'$ satisfying $x=\tilde\rho(t_0)$ and $v:=\tilde\rho'(t_0)\in\mf a$. Then
\begin{enumerate}[(i)]
\item \label{it:coro-cont-proj} there exists $\lambda:I'\to\mf a$ differentiable at $t_0$ with $\lambda(t_0)=x$ and $\lambda'(t_0)=v$ and $\pi\circ\lambda=\pi\circ\tilde\rho$, and
\item \label{it:coro-sequence} for any sequence $t_n\to t_0$ in $I'$ there is a subsequence $t_n'$ and elements $U_n'\in K_x$ and $U\in K_x\cap K_v$ such that $\Ad_{U_n'}^{-1}(\tilde\rho(t_n'))=\lambda(t_n')$ and such that $\Ad_{U_n'}\to\Ad_U$.
\end{enumerate} 
\end{corollary}

\begin{proof}
Let $\tilde{\mf w}$ be a Weyl chamber for $W_x$ containing $v$ and let $\omega:\mf p_x\to\tilde{\mf w}$ denote the map from Lemma~\ref{lemma:project-to-chamber}~\eqref{it:cont-proj}. Define the path
\begin{align*} 
\lambda: I'\to\mf a,\quad t\mapsto
\begin{cases}
\omega\Big(\frac{\tilde\rho(t)-x}{t-t_0}\Big)(t-t_0)+x &\text{ if } t\neq t_0 \\
x &\text{ if } t=t_0.
\end{cases}
\end{align*}
Note that for $t>t_0$, $\lambda$ lies in $\tilde{\mf w}$, and for $t<t_0$, $\lambda$ lies in $-\tilde{\mf w}$.
Then by continuity of $\omega$ it holds that
\begin{align*} 
\frac{\lambda(t)-x}{t-t_0}=\omega\Big(\frac{\tilde\rho(t)-x}{t-t_0}\Big) \to \omega(v)=v,
\end{align*}
as $t\to t_0$. By Lemma~\ref{lemma:project-to-chamber}~\eqref{it:cont-proj} there exists for every $t\in I'\setminus\{t_0\}$ some element $U_t\in K_x$ such that 
\begin{align} \label{eq:diff-quotient}
\frac{\lambda(t)-x}{t-t_0}=\Ad_{U_t}^{-1}\Big(\frac{\tilde\rho(t)-x}{t-t_0}\Big),
\end{align}
and hence $\lambda(t)=\Ad_{U_t}^{-1}(\tilde\rho(t))$ which shows that $\pi\circ\lambda=\pi\circ\tilde\rho$. Hence $\lambda$ satisfies all the desired properties and this proves~\eqref{it:coro-cont-proj}.
Now let any sequence $t_n\to t_0$ in $I'$ be given and set $U_n=U_{t_n}$. Then $U_n\in K_x$ and $\Ad_{U_n}^{-1}(\tilde\rho(t_n))=\lambda(t_n)$. By Eq.~\eqref{eq:diff-quotient} and the compactness of $\Ad_K$, there is a subsequence of $\Ad_{U_n}$ converging to some $\Ad_U$ with $U\in K_x\cap K_v$. This proves~\eqref{it:coro-sequence}.
\end{proof}

Now we are ready to prove the first main result of this section, which shows that if $\rho$ is differentiable at some point, then one can also choose $\lambda$ to be differentiable at that point. Moreover the derivative of $\lambda$ is then unique up to some Weyl group action.

\begin{proposition}
\label{prop:deriv-of-projected-path}
Let $I$ be an open interval and let $\rho:I\to\mf p$ be differentiable at some $t_0\in I$. Then there is $\lambda:I\to\mf a$ which is differentiable at $t_0$ and satisfies $\pi\circ\rho = \pi\circ\lambda$. Moreover there is some $U\in K$ such that
\begin{align*}
\lambda(t_0)=\Ad_U^{-1}(\rho(t_0)) \text{ and } \lambda'(t_0)=\Ad_U^{-1}(\Pi_{\rho(t_0)}(\rho'(t_0))),
\end{align*}
and any other path $\mu:I\to\mf a$ which is differentiable at $t_0$ and satisfies $\pi\circ\rho = \pi\circ\mu$ satisfies that $\mu(t_0)=w\cdot\lambda(t_0)$ and $\mu'(t_0)=w\cdot\lambda'(t_0)$ for some $w\in W$.
\end{proposition}

\begin{remark} \label{rmk:orbifold-deriv}
If we write $\xi=\phi^{-1}\circ\pi\circ\rho$, then this proposition shows that $\xi$ is differentiable at $t_0$ in the orbifold sense, as defined in Definition~\ref{def:orbifold-derivatives}. The derivative of $\xi$ is then 
\begin{align*}
D\xi(t_0)=D\pi_{\mf a}(\Ad_U^{-1}(\rho(t_0)),\Ad_U^{-1}(\Pi_{\rho(t_0)}(\rho'(t_0)))),
\end{align*}
for any $U\in K$ such that $\Ad_U^{-1}(\rho(t_0))\in\mf a$ and $\Ad_U^{-1}(\Pi_{\rho(t_0)}(\rho'(t_0)))\in\mf a$.
\end{remark}

\begin{proof} 
By Lemma~\ref{lemma:project-to-slice} we find some $U\in K$ such that $x:=\Ad_U^{-1}(\rho(t_0))\in\mf a$ and $v:=\Ad_U^{-1}(\Pi_{\rho(t_0)}(\rho'(t_0)))\in\mf a$, as well as some open interval $I'\subseteq I$ containing $t_0$ and a path $\tilde\rho:I'\to\mf p_x$ with $\tilde\rho(t_0)=x$ and $\tilde\rho'(t_0)=v$ satisfying $\pi\circ\tilde\rho=\pi\circ\rho$ on $I'$. 
Then by Corollary~\ref{coro:lambda-def}~\eqref{it:coro-cont-proj} we obtain $\lambda:I'\to\mf a$ satisfying the desired properties.
The uniqueness of $(\lambda(t_0),\lambda'(t_0))$ up to Weyl group action follows immediately from Lemma~\ref{lemma:deriv-well-def}~\eqref{it:deriv-well-def}.
\end{proof}

\begin{example} 
Let us illustrate this result in the setting of Example~\ref{ex:hermitian-evd}.
Let $\rho:I\to\mf{herm}_0(n,\C)$ be a path of traceless Hermitian matrices which is differentiable at some $t_0\in I$. Let $\rho(t_0)=\sum_{j=1}^m\mu_j P_j$ be the eigendecomposition of $\rho(t_0)$. Then it holds that $\Pi_{\rho(t_0)}(\rho'(t_0))=\sum_{j=1}^m P_j\rho'(t_0)P_j$. Using a unitary change of basis, we can assume that both $\rho(t_0)$ and $\Pi_{\rho(t_0)}(\rho'(t_0))$ are diagonal. By Proposition~\ref{prop:deriv-of-projected-path} there exist eigenvalue functions $\lambda_i:I\to\R$ which are differentiable at $t_0$ and satisfy $\lambda_i(t_0)=\rho_{i,i}(t_0)$ and $\lambda'_i(t_0)=\rho'_{i,i}(t_0)$. The formula for $\lambda'_i(t_0)$ coincides with that of~\cite[Ch.~II, Thm.~5.4]{Kato80} and~\cite[Ch.~I.\S5, Thm.~1]{Rellich69}.
\end{example}

\begin{lemma} \label{lemma:deriv-of-projected-path-cont}
Let $\rho:I\to\mf p$ be continuously differentiable. Then $\xi:I\to\mf a/W$ given by $\xi=\phi^{-1}\circ\pi\circ\rho$ is continuously differentiable in the sense of orbifolds.
\end{lemma}

\begin{proof}
By Proposition~\ref{prop:deriv-of-projected-path} we know that $\xi$ is differentiable on $I$ in the sense of orbifolds, and its derivative is denoted by $D\xi:I\to T(\mf a/W)$. Let $t_0\in I$ be arbitrary. We want to show that $D\xi$ is continuous at $t_0$. By Lemma~\ref{lemma:project-to-slice} we obtain a path $\tilde\rho$ on an open interval $I'\subseteq I$ containing $t_0$ which satisfies $\pi\circ\tilde\rho=\pi\circ\rho$. From the definition of $\tilde\rho$ it is clear that it is $C^1$. Hence in the following we will work with $\tilde\rho$ instead of $\rho$. Note that by Remark~\ref{rmk:orbifold-deriv} we know that $D\xi(t_0)=D\pi_{\mf a}(x,v)$ where $x=\tilde\rho(t_0)$ and $v=\tilde\rho'(t_0)$.

First we show that $\Pi_{\tilde\rho(t)}(\tilde\rho'(t))$ is continuous at $t_0$. Let $\lambda:I'\to\mf a$ be the function given by Corollary~\ref{coro:lambda-def}~\eqref{it:coro-cont-proj}. Let $t_n$ be any sequence in $I'$ converging to $t_0$. Then by Corollary~\ref{coro:lambda-def}~\eqref{it:coro-sequence} there is a subsequence $t_n'$ and a sequence of elements $U_n'\in K_x$ such that $\Ad_{U_n'}^{-1}(\tilde\rho(t_n'))=\lambda(t_n')$ as well as some $U\in K_x\cap K_v$ such that $\Ad_{U'_n}\to\Ad_U$. To simplify notation we write $x_n = \Ad_{U_n'}^{-1}(\tilde\rho(t_n'))$ and $v_n = \Ad_{U_n'}^{-1}(\tilde\rho'(t_n'))$. Since $\lambda$ is continuous at $t_0$ it holds that $x_n\to x$, and since $\tilde\rho'(t_n')\to v$ it holds that $v_n\to v$. If $\Pi_{\mf a}:\mf p \to\mf a$ denotes the orthogonal projection onto $\mf a$, then for any $z\in\mf a$ it holds that $\Pi_z^\perp\circ\Pi_{\mf a} = 0$ since $\mf a\subseteq\mf p_z$. Since $v\in\mf a$ and since $\Pi_{x_n}$ is an orthogonal projection, it holds that
\begin{align*}
\Pi_{x_n}v_n=v_n-\Pi^\perp_{x_n}v_n 
=
v_n-\Pi^\perp_{x_n}\Pi^\perp_{\mf a}v_n\to v.
\end{align*}
Hence by Lemma~\ref{lemma:equivariance}~\eqref{it:equiv-proj} it holds that
\begin{align*}
\Pi_{\tilde\rho(t_n')}(\tilde\rho'(t_n'))=\Ad_{U_n'}(\Pi_{x_n}(v_n))\to v,
\end{align*}
as desired. Since for every given sequence $t_n$ we have found a subsequence $t_n'$, this shows that $\Pi_{\tilde\rho(t)}(\tilde\rho'(t))$ is continuous at $t_0$.

Now we show that $D\xi$ is also continuous at $t_0$. This is done using a similar method. Again let $t_n\to t_0$ be given. 
Consider the sequence $\Pi_{\tilde\rho(t_n)}(\tilde\rho'(t_n))$ which converges to $v$ as shown above and note that by Lemma~\ref{coro:containment-of-slices} the sequence lies in $\mf p_x$.
Then by Lemma~\ref{lemma:project-to-chamber}~\eqref{it:sequence} applied to this sequence, there exists a subsequence $t_n'$ and elements $U_n'\in K_x$ satisfying 
$v_n:=\Ad_{U_n'}^{-1}(\Pi_{\tilde\rho(t_n')}(\tilde\rho'(t_n')))\in\tilde{\mf w}$ and $v_n\to v$,
as well as some $U\in K_x\cap K_v$ such that $\Ad_{U'_n}\to\Ad_U$. For each $n$, by Corollary~\ref{coro:simul-diag}, we find some $V_n\in K_x\cap K_{v_n}$ such that $x_n:=\Ad_{V_n}^{-1}\Ad_{U_n'}^{-1}(\tilde\rho(t_n'))\in\mf a$. By Remark~\ref{rmk:orbifold-deriv} it holds that $D\xi(t_n')=D\pi_{\mf a}(x_n,v_n)$. Moreover it holds that $x_n\to x$. Hence 
\begin{align*}
D\xi(t_n')=D\pi_{\mf a}(x_n,v_n) \to D\pi_{\mf a}(x,v)=D\xi(t_0)
\end{align*}
by continuity of the quotient map $D\pi_{\mf a}$. This concludes the proof.
\end{proof}

So far we have shown that if $\rho:I\to\mf p$ is (continuously) differentiable, then $\xi:I\to\mf a/W$ is (continuously) differentiable in the sense of orbifolds. At this point it is not at all clear that a corresponding (continuously) differentiable path $\lambda:I\to\mf a$ must also exist. That this is the case is shown in detail in Appendix~\ref{app:orbifolds} in the more general setting of orbifolds.

\begin{theorem}[Differentiable Diagonalization] \label{thm:differentiable-diagonalization}
Let $\rho:I\to\mf p$ be (continuously) differentiable, then there exists a (continuously) differentiable path $\lambda:I\to\mf a$ satisfying $\pi\circ\rho=\pi\circ\lambda$. Moreover, for every $t\in I$, there is some $U\in K$ such that $\Ad_U^{-1}(\rho(t))\in\mf a$ and $\Ad_U^{-1}(\Pi_{\rho(t)}(\rho'(t)))\in\mf a$, and for any such $U$ it holds that
\begin{align*}
(\lambda(t),\lambda'(t)) = w\cdot\Ad_U^{-1}(\rho(t),\Pi_{\rho(t)}(\rho'(t)))
\end{align*}
for some Weyl group element $w\in W$.
\end{theorem}

\begin{proof}
The differentiable case follows from Proposition~\ref{prop:deriv-of-projected-path} combined with Proposition~\ref{prop:orbifold-C1-lift}~\eqref{it:glue-diff2}. The continuously differentiable case follows from Lemma~\ref{lemma:deriv-of-projected-path-cont} combined with Proposition~\ref{prop:orbifold-C1-lift}~\eqref{it:glue-C1-2}.
\end{proof}

\begin{remark}
Considering the Cartan decomposition $\mf{sl}(n,\C) = \mf{su}(n) \oplus \mf{herm}_0(n)$, Theorem~\ref{thm:differentiable-diagonalization} generalizes a well-known result by Rellich (see~\cite[Ch.~I.\S5, Thm.~1]{Rellich69}, as well as~\cite[Ch.~I.\S5, Theorem, pp.~44-45]{Rellich69}, or, for a simpler proof,~\cite[Ch.~II, Thm.~6.8]{Kato80}) showing that for a $C^1$ path of Hermitian matrices, the eigenvalues can be chosen as $C^1$ functions. 
\end{remark}

\begin{remark}
Let us mention some counterexamples to different generalizations of this result.
\cite[Example~5.9]{Kato80} shows that for a $C^1$ path of diagonalizable (but not symmetric) matrices, the eigenvalues need not be $C^1$ (but they are differentiable).
Due to Example~\ref{ex:C2-eigenvalues}, even if $\rho$ is $C^\infty$ we cannot guarantee that $\lambda$ can be chosen $C^2$. 
Note also that the diagonalizing unitary may have to be discontinuous, see Example~\ref{ex:no-cont-diag}. 
On the other hand, slight improvements of Theorem~\ref{thm:differentiable-diagonalization} might be possible by generalizing results from~\cite{Rainer11}.
\end{remark}

The following result shows that as long as we don't run into singular points, i.e., points with non-trivial stabilizer in the Weyl group, a $C^k$ path can always be diagonalized in a $C^k$ fashion. In Section~\ref{sec:analytic} we will show that real analytic paths always have a real analytic diagonalization, even without the exclusion of singular points. First we need to define an inverse of $\ad_x$ by restricting the domain and codomain. Indeed we get a well-defined inverse $\ad_x^{-1}:\mf p_x^\perp\to\mf k_x^\perp$, since $\mf p_x^\perp=\mf p\cap\img\ad_x$. Recall that $\Pi^\perp_x$ is the orthogonal projection onto $\mf p_x^\perp$.

\begin{proposition} \label{prop:Ck-regular}
Let $0\in I$ be an open interval and let $\rho:I\to\mf p$ be regular and $C^k$ for $0\leq k\leq \infty$. Then there exists a $C^k$ path $U:I\to K$ such that $\lambda(t):=\Ad_U^{-1}(t)(\rho(t))\in\mf a$ for all $t\in I$. Moreover for $k\geq1$ we can choose $U$ to satisfy
\begin{align} \label{eq:Ck-lift}
U'(t)=h(t)U(t),\quad 
h(t)=-\ad_{\rho(t)}^{-1}(\Pi^\perp_{\rho(t)}(\rho'(t))), \quad
\Ad_{U(0)}^{-1}(\rho(0))\in\mf a.
\end{align}
Furthermore, any continuous path $\mu:I\to\mf a$ satisfying $\pi\circ\mu=\pi\circ\rho$ satisfies $\mu=w\cdot\lambda$ for some fixed Weyl group element $w\in W$.
\end{proposition}

\begin{proof}
First we consider the continuous case $k=0$. Let $\mf p_0$ denote the set of all regular points of $\mf p$. 
By Lemma~\ref{lemma:stratification-of-p} (proven later) this is a trivial smooth fiber bundle over the open Weyl chamber $\mf w_0$ with fiber $K/Z_K(\mf a)$, where $Z_K(\mf a)=K_{\mf a}$ is the centralizer (stabilizer) of $\mf a$ in $K$. Hence we can project $\rho$ to give continuous paths in $\mf w_0$ and $K/Z_K(\mf a)$. It remains to continuously lift the path in $K/Z_K(\mf a)$ to $K$.
For any $t\in I$ one can find a local continuous lift in a neighborhood of $t$ by working in any local trivialization of the bundle $\pi_K:K\to K/Z_K(\mf a)$. Then such local lifts can be glued together to a global continuous lift, analogously to the proof of Lemma~\ref{lemma:local-global-C1}.

Now consider $k\geq1$. By the above, there is some continuous $V:I\to K$ such that $\Ad_{V(t)}^{-1}(\rho(t))\in\mf a$. Define $h$ as in~\eqref{eq:Ck-lift}. Then by Lemma~\ref{lemma:equivariance}~\eqref{it:equiv-adinv} and~\eqref{it:equiv-projperp} it holds that
\begin{align} \label{eq:h-transformed}
h=-(\Ad_V\circ\ad^{-1}_{\Ad_V^{-1}(\rho)}\circ\Pi^\perp_{\mf a}\circ\Ad^{-1}_V)(\rho').
\end{align}
This shows that $h$ is continuous. Define $U$ as in~\eqref{eq:Ck-lift}. In order to verify that $\Ad_U^{-1}(\rho)$ lies in $\mf a$, we compute
\begin{align*}
(\Ad_U^{-1}(\rho))'
&=
\Ad_U^{-1}(\rho') + [\Ad_U^{-1}(\rho),U^{-1}U']
\\&=
\Ad_U^{-1}(\rho') + \Ad_U^{-1}([\rho,h])
=
\Pi_{\Ad_U^{-1}(\rho)}(\Ad_U^{-1}(\rho')),
\end{align*}
and hence the part of the derivative tangent to the fibers is always zero. Together with the assumption $\Ad_{U(0)}^{-1}(\rho(0))\in\mf a$, this shows that $\Ad_U^{-1}(\rho)$ remains in $\mf a$ at all times.
It remains to show that we indeed have the desired level of differentiability.
Note that if $h\in C^{j-1}$, this implies that $U\in C^{j}$, and by replacing $V$ by $U$ in~\eqref{eq:h-transformed}, we see that $U\in C^{j}$ implies that $h\in C^{\min(j,k-1)}$, which by induction implies that $U\in C^k$.
Finally, the uniqueness claim is clear, since a continuous lift of $\pi\circ\rho$ in $\mf a$ must lie in a single open Weyl chamber and is uniquely defined within this Weyl chamber.
\end{proof}

\begin{remark}
The homogeneous space $K/Z_K(\mf a)$ is reductive with $\mf k=\mf z_{\mf k}(\mf a)\oplus \mf z_{\mf k}(\mf a)^\perp$, since $Z_K(\mf a)$ is compact. This induces a connection of the principal bundle $\pi:K\to K/Z_K(\mf a)$, called the canonical connection, cf.~\cite[Ch.~II, Thm~11.1]{KobNom96}. In Proposition~\ref{prop:Ck-regular} we implicitly used this connection to lift differentiable paths from $K/Z_K(\mf a)$ to $K$.
\end{remark}

\begin{remark}
Proposition~\ref{prop:Ck-regular} generalizes 
\cite[Thm.~2]{BunseGerstner91} which shows that a $C^1$ path of real $m$ by $n$ matrices of full-rank and with distinct singular values has a $C^1$ singular value decomposition.
\end{remark}

\medskip
\subsection{Real Analytic Paths} \label{sec:analytic}

In this section we show that for a real analytic path $\rho:I\to\mf p$ there exists a real analytic path $U:I\to K$ such that $\Ad_U^{-1}(\rho)$ lies in $\mf a$. Clearly the path $\Ad_U^{-1}(\rho)$ is real analytic, and in fact it is the unique real analytic path in $\mf a$ which is a lift of $\pi\circ\rho$, up to a global Weyl group action. This is the content of Theorem~\ref{thm:analytic-diag}. This result stands in stark contrast to the previous section since even a $C^\infty$ path $\rho$ cannot guarantee the existence of a continuous diagonalizing $U$, see Example~\ref{ex:no-cont-diag}, or a $C^2$ choice of eigenvalues, see Example~\ref{ex:C2-eigenvalues}.

We will start with the well-known matrix case and then lift the diagonalization to the symmetric Lie algebra via the adjoint representation.
Consider a finite dimensional real inner product space $V$ and its complexification $V^\C$.
For an operator $A\in\mf{gl}(V)$ we call $A_c\in\mf{gl}(V^\C)\cong\mf{gl}(V)^\C$ its complexification.
To avoid confusion, we use $\hatad$ to denote the adjoint maps on $\mf{gl}(V)$ and $\mf{gl}(V^\C)$. 
We endow $\mf{gl}(V)$ and $\mf{gl}(V^\C)$ with the Hilbert-Schmidt inner product $\braket{A,B}=\tr(A^*B)$.

\begin{lemma} \label{lemma:matrix-inheritance}
Let $V$ be a real inner product space and $A\in\mf{gl}(V)$. If $A$ is semisimple\footnote{Recall that a linear operator is semisimple if each invariant subspace has an invariant complement. Over $\C$ this is equivalent to being diagonalizable.}, then so are $A_c$, and $\hatad_A$, and $\hatad_{A_c}$. Similarly, if $A$ is normal, then so are $A_c$, and $\hatad_A$, and $\hatad_{A_c}$. 
\end{lemma}

\begin{proof}
Let $A$ be semisimple. Recall that an operator is semisimple if and only if its minimal polynomial is square free. This shows that $A_c$ is semisimple if and only if $A$ is. Then it is easy to see that diagonalizability of $A_c$ implies diagonalizability of $\hatad_{A_c}$. But since $\hatad_{A_c}$ is the complexification of $\hatad_A$, the above shows that $\hatad_A$ is semisimple.

If $A$ is normal, then clearly $A_c$ is too. An elementary computation shows that 
\begin{equation*}
\tr\big((\hatad_A(B))^*C\big) = \tr\big(B^*\,\hatad_{A^*}(C)\big)      
\end{equation*}
which implies that $(\hatad_A)^* = \hatad_{A^*}$ for $A \in \mf{gl}(V)$ and thus
\begin{equation*}
\Big[(\hatad_A)^*,\hatad_A\Big]=\hatad_{[A^*,A]}=0
\end{equation*}
and so $\hatad_A$ is also normal. The proof for $A_c$ is identical.
\end{proof}

\begin{lemma} \label{lemma:matrix-decomps}
Let $V$ be a finite dimensional real inner product space and let $V^\C$ be its complexification. Let $A\in\mf{gl}(V)$ be semisimple and let $A_c\in\mf{gl}(V^\C)\cong\mf{gl}(V)^\C$ denote the complexification. Then it holds that
$$
\mf{gl}(V)=\img\hatad_A\oplus\ker\hatad_A, \quad \mf{gl}(V^\C)=\img\hatad_{A_c}\oplus\ker\hatad_{A_c}
$$
and
$$
\img\hatad_A=\mf{gl}(V)\cap\img\hatad_{A_c}, \quad \ker\hatad_A=\mf{gl}(V)\cap\ker\hatad_{A_c},
$$
Moreover, if $A$ is normal, then the decompositions are orthogonal.
\end{lemma}

\begin{proof}
The decomposition into kernel and image holds for every semisimple operator.
It is clear that if $X\in\mf{gl}(V)$ then $[A_c,X_c]=0$ if and only if $[A,X]=0$.
Let $Z\in\mf{gl}(V^\C)$ and assume that $\hatad_{A_c}(Z)\in\mf{gl}(V)$. Then $[A,Z+\overline Z]/2=[A_c,Z]$.
If $A$ is normal, then clearly the decomposition $\mf{gl}(V)=\img\hatad_A\oplus\ker\hatad_A$ is orthogonal. Using Lemma~\ref{lemma:matrix-inheritance} the same is true for the complexification.
\end{proof}

Hence, for semisimple $A\in\mf{gl}(V)$, we can define the projection $\widehat\Pi_A$ onto $\ker\hatad_A$ and along $\img\hatad_A$, and its complement $\widehat\Pi_A' = 1-\widehat\Pi_A$. 
If $A$ is normal, then the projection is orthogonal and we write $\widehat\Pi_A^\perp = 1-\widehat\Pi_A$.
We use the analogous notation for the complexification $A_c$.
Lemma~\ref{lemma:matrix-decomps} also shows that for $A$ semisimple, the maps $\hatad_A|_{\img\hatad_A}$ and $\hatad_{A_c}|_{\img\hatad_{A_c}}$ are bijective to their image. We will write the inverse maps as $\hatad_A^{-1}$ and $\hatad_{A_c}^{-1}$ and leave the restriction implicit.

More explicitly, for semisimple $A_c$ with eigenvalues $\lambda_i$ and eigenprojections $P_i$ for $i=1,\ldots,N$, we can express the inverse of $\hatad_{A_c}$ on the image of $\widehat\Pi_{A_c}'$ by  
\begin{equation} \label{eq:adinv}
(\hatad_{A_c}^{-1} \circ \widehat \Pi_{A_c}') (B)
=
\sum_{k=1}^N \sum_{l=1 \atop l\neq k}^N \frac{P_kBP_l}{\lambda_k-\lambda_l}.
\end{equation}

\begin{corollary} \label{coro:matrix-invariance}
For $A,B\in\mf{gl}(V)$ and $C\in\img\hatad_A$ it holds that $\Pi^\perp_{A_c}(B_c)=\Pi^\perp_{A}(B)$ and $\hatad_{A_c}^{-1}(C_c)=\hatad_A^{-1}(C)$.
\end{corollary}

\begin{proof}
This follows from Lemma~\ref{lemma:matrix-decomps}.
\end{proof}

Now we consider a real analytic path of operators $A:I\to\mf{gl}(V)$. We can always find a simply connected open set $\mb G$ containing $I$ and such that there is an analytic continuation $A_c$ of $A$ on $\mb G$. Outside of a finite set of exceptional points\footnote{An exceptional point in $I$ is a point at which two eigenvalues meet, without being permanently degenerate. Any compact set in $\C$ contains only finitely many exceptional points.} in $\mb G$, the number of eigenvalues $\lambda_i$ and the dimensions of the corresponding eigenprojectors $P_i$ are constant. In fact, by~\cite[Thm.~1.8]{Kato80}, the eigenvalues $\lambda_i$, eigenprojectors $P_i$, and the eigennilpotents $D_i$ are branches of analytic functions with only algebraic singularities at some of the exeptional points. 
If $A(t)$ is semisimple for all $t\in I$, then $A_c$ is also semisimple on $\mb G$, since its eigen-nilpotents must vanish identically. 

\begin{lemma} \label{lemma:analytic-diag-V} 
Let $V$ be a finite dimensional complex Hilbert space and $A:I\to\mf{gl}(V)$ be a real analytic curve of normal operators. 
Moreover, let $t_0 \in I$ be an exceptional point and let $A_c(z)$ denote the analytic extension of $A(t)$ to an open disk $\mb D_r$ of radius $r$ about $t_0$ such that no other exceptional points are contained in $\mathbb{D}_r$. Then on the punctured disc $\dot{\mb D}_r$ the following identity holds
\begin{equation} \label{eq:holomorphic-vf}
-\left(\hatad_{A_c(z)}^{-1} \circ \widehat\Pi_{A_c(z)}'\right) (A'_c(z))
=\frac{1}{2}\sum_{k=1}^N [P_k'(z),P_k(z)]\,,
\end{equation}
where $P_k(z)$ are the corresponding eigenprojections of $A_c(z)$. In particular,~\eqref{eq:holomorphic-vf} shows that the expression can be continued analytically to $z=t_0$.
\end{lemma}

\begin{proof}
The proof is a long but straightforward computation involving resolvents given in Appendix~\ref{app:computations}. 
\end{proof}

The motivation for~\eqref{eq:holomorphic-vf} comes from Lemma~\ref{lemma:differentiably-diagonalizable} and Proposition~\ref{prop:Ck-regular}. The right-hand side is the formula derived in~\cite[p.~105]{Kato80}, whereas the left-hand side is written in terms of Lie algebraic quantities. A priori it is not even clear that the left-hand side should be continuous.

\begin{corollary} \label{coro:real-analytic-diag-V}
Let $V$ be a real inner product space, and let $I$ be an open interval containing $0$. Let $A:I\to\mf{gl}(V)$ be a real analytic curve of normal operators, and let $U:I\to\GL(V)$ be the solution of the ordinary differential equation
\begin{align} \label{eq:real-analytic-U}
U'(t)
= -\left(\hatad_{A(t)}^{-1} \circ \widehat\Pi_{A(t)}^\perp\right) (A'(t)) \cdot U(t), \quad U(0)=\id\in\GL(V).
\end{align}
Then it holds for all $t,s\in I$ that
\begin{align*}
[\Ad_{U(t)}^{-1}(A(t)),\Ad_{U(s)}^{-1}(A(s))]=0.
\end{align*}
\end{corollary}

\begin{proof}
We start by complexifying. 
Let $\mb G\subseteq\C$ be a simply connected open set containing $I$ such that there is an analytic continuation $A_c$ of $A$ on $\mb G$.
Let $U_c:\mb G\to\GL(V^\C)$ be the solution of the ordinary differential equation\footnote{The solution exists and is unique and holomorphic. See for instance~\cite[p.~100]{Kato80}.}
\begin{align*} 
U_c'(z)
= -\left(\hatad_{A_c(z)}^{-1} \circ \widehat\Pi_{A_c(z)}'\right) (A'_c(z)) \cdot U_c(z), \quad U_c(0)=\id\in\GL(V^\C).
\end{align*}
By Lemma~\ref{lemma:analytic-diag-V} the solution $U_c$ satisfies
$$
U'_c(z)=\frac{1}{2}\sum_{k = 1}^N [P'_k(z),P_k(z)]\cdot U_c(z)
$$
where the $P_k$ are the eigenprojections of $A_c$. By~\cite[Ch.~II~\S4.5]{Kato80} this implies that $P_k(z)=\Ad_{U_c(z)}(P_k(0))$. Hence all $\Ad_{U_c(z)}^{-1}(A(z))$ for $z\in\mb G$ commute.
Now we define by restriction $U:=U_c|_I$. By Corollary~\ref{coro:matrix-invariance} it is clear that $U$ satisfies~\eqref{eq:real-analytic-U} and of course $[\Ad_{U(t)}^{-1}(A(t)),\Ad_{U(s)}^{-1}(A(s))]=0$ still holds.
\end{proof}

The main idea is to go from a semisimple, orthogonal, symmetric Lie algebra to concrete matrices via the adjoint representation and use the previous result. The following elementary properties of Lie algebra homomorphisms will be useful for this transition step. We use an arbitrary homomorphism $\phi$ here instead of $\ad$ to avoid confusion with the other uses of $\ad$.

\begin{lemma} \label{lemma:homomorphism-projection}
Let $\phi:\mf g\to\mf h$ be a Lie algebra homomorphism and let $\ad$ and $\widehat\ad$ denote the respective adjoint maps. Then for $x,y\in\mf g$ it holds that
\begin{align}\label{eq:lie-alg-hom}
\phi(\ad_x(y))=\hatad_{\phi(x)}(\phi(y)),
\end{align}
and hence $\phi$ maps $\img\ad_x$ to $\img\widehat\ad_{\phi(x)}$ and $\ker\ad_x$ to $\ker\widehat\ad_{\phi(x)}$. Now assume that $x$ and $\phi(x)$ are semisimple\footnote{An element of a Lie algebra is semisimple if its $\ad$-representation is a semisimple operator.}. Let $\Pi'_x:\mf g\to\mf g$ denote the projection onto $\img\ad_x$ along $\ker\ad_x$ and analogously for $\widehat\Pi'_{\ad(x)}:\mf h\to\mf h$. Then
\begin{align}\label{eq:lie-alg-proj}
\phi(\Pi'_x(y)) = \widehat\Pi'_{\phi(x)}(\phi(y)).
\end{align} 
\end{lemma}

\begin{proof}
Eq.~\eqref{eq:lie-alg-hom} holds by definition, and it immediately implies that $\phi$ maps $\img\ad_x$ to $\img\widehat\ad_{\phi(x)}$ and $\ker\ad_x$ to $\ker\widehat\ad_{\phi(x)}$. Now assume that $x$ and $\phi(x)$ are semisimple. If $y\in\img\ad_x$ and $z\in\ker\ad_x$ we see that $\phi(\Pi'_x(y+z))=\phi(y)$ and $\widehat\Pi'_{\phi(x)}(\phi(y+z))=\phi(y)$ by the previous observation.
\end{proof}

\begin{theorem}[Real Analytic Diagonalization] \label{thm:analytic-diag}
Let $I$ be an open interval containing $0$ and let $\rho:I\to\mf p$ be a real analytic path. Then there exists a real analytic path $U:I\to K$ such that $\lambda(t):=\Ad_{U(t)}^{-1}(\rho(t))\in\mf a$ for all $t\in I$. 
Moreover, such $U$ can be obtained as the solution to
$$
U'(t)=k(t)U(t),\quad \Ad_{U(0)}^{-1}(\rho(0))\in\mf a,\quad k(t) = -\ad_{\rho(t)}^{-1}(\Pi^\perp_{\rho(t)}(\rho'(t))).
$$
Furthermore, any real analytic path $\mu:I\to\mf a$ satisfying $\pi\circ\mu=\pi\circ\rho$ satisfies $\mu=w\cdot\lambda$ for some fixed Weyl group element $w\in W$.
\end{theorem}

\begin{proof}
Let $\widehat\rho(t) = \ad_{\rho(t)}$. This is a real analytic path in $\mf {gl}(\mf g)$.
By the proof of Lemma~\ref{lemma:semisimple-normal-adx} $\widehat\rho$ is normal on $I$.
By Corollary~\ref{coro:real-analytic-diag-V} the function $\widehat U:I\to\GL(\mf g)$ solving the differential equation
\begin{align*}
\widehat U'(t)
= -\left(\hatad_{\widehat\rho(t)}^{-1} \circ \widehat\Pi^\perp_{\widehat\rho(t)}\right)  (\widehat\rho'(t))
\cdot\widehat U(t), \quad\widehat U(0)=\id\in\GL(\mf g)
\end{align*}
diagonalizes $\widehat\rho$ in the sense that
\begin{align} \label{eq:diagonalize-ad-rep}
[\Ad_{\widehat U(t)}^{-1}(\widehat\rho(t)),\Ad_{\widehat U(s)}^{-1}(\widehat\rho(s))]=0
\end{align}
for all $t,s\in I$. 
Now we need to translate this back to $\mf g$. First we define the function $k:I\to\mf k$ by
\begin{align} \label{eq:def-omega}
k(t) = -\ad_{\rho(t)}^{-1}(\Pi_{\rho(t)}^\perp(\rho'(t))),
\end{align}
and we claim that
\begin{align} \label{eq:ad-omega}
\ad_{k(t)} = -\widehat\ad_{\widehat\rho(t)}^{-1}(\widehat\Pi^\perp_{\widehat\rho(t)}(\widehat\rho'(t))).
\end{align}
Indeed, this follows from Lemma~\ref{lemma:homomorphism-projection} with $\ad:\mf g\to\mf{gl}(\mf g)$ as $\phi$. To apply the lemma we need to verify that both $\rho(t)$ and $\widehat\rho(t)$ are semisimple as Lie algebra elements, which follows in both cases from the above observation that $\widehat\rho(t)$ is a normal operator.
So we may compute:
\begin{align*}
\widehat\ad_{\widehat\rho(t)}(\ad_{k(t)})
\overset{\eqref{eq:lie-alg-hom}}{=}
\ad_{\ad_{\rho(t)}(k(t))}
\overset{\eqref{eq:def-omega}}{=}
-\ad_{\Pi^\perp_{\rho(t)}(\rho'(t))}
\overset{\eqref{eq:lie-alg-proj}}{=}
-\widehat\Pi^\perp_{\widehat\rho(t)}(\widehat\rho'(t)).
\end{align*}
Now we define $U:I\to K$ by
$$
U'(t)=k(t)U(t),\quad U(0)=\id,
$$
and we claim that
$$
\widehat U(t) = \Ad_{U(t)}.
$$
Indeed, both $\widehat U(t)$ and $\Ad_{U(t)}$ lie in $\GL(\mf g)$, they satisfy $\widehat U(0)=\Ad_{U(0)}$, and since
\begin{align*}
(\Ad_{U(t)})' = \ad_{k(t)} \Ad_{U(t)},
\end{align*}
they satisfy the same differential equation by~\eqref{eq:ad-omega}.\footnote{Note that in the case where $K=\Int_{\mf g}(\mf k)$, it even holds that $U=\widehat U$.} This implies that $\Ad_{\widehat U}^{-1}(\ad_{\rho})=\ad_{\Ad_U^{-1}(\rho)}$, and hence
\begin{align*}
0&\overset{\eqref{eq:diagonalize-ad-rep}}{=}
[\Ad_{\widehat U(t)}^{-1}(\widehat\rho(t)),\Ad_{\widehat U(s)}^{-1}(\widehat\rho(s))]
\\&=[\ad_{\Ad_{U(t)}^{-1}(\rho(t))},\ad_{\Ad_{U(s)}^{-1}(\rho(s))}]
\\&=\ad_{[\Ad_{U(t)}^{-1}(\rho(t)),\Ad_{U(s)}^{-1}(\rho(s))]},
\end{align*}
and by semisimplicity of $\mf g$ this means that $[\Ad_{U(t)}^{-1}(\rho(t)),\Ad_{U(s)}^{-1}(\rho(s))]=0$. 
Since $\Ad:G\to\GL(\mf g)$ has discrete kernel, $U$ is real analytic as continuous lift of $\widehat U$. 
Since all $\Ad_{U(t)}^{-1}(\rho(t))$ commute, by Lemma~\ref{lemma:max-abelian-conj} there exists $V\in K$ such that $\Ad_{U(t)V}^{-1}(\rho(t))\in\mf a$ for all $t\in I$. Then $\tilde U(t)=U(t)V$ also satisfies $\tilde U'(t)=k(t)\tilde U(t)$.
Finally, uniqueness of the diagonalized path up to global Weyl group action follows from Lemma~\ref{lemma:analytic-unique}.
\end{proof}

\begin{remark}
Theorem~\ref{thm:analytic-diag} generalizes the well-known special case for Hermitian matrices, see~\cite[Thm.~6.1]{Kato80}, and also~\cite[Thm.~1]{BunseGerstner91} which is the special case for the real singular value decomposition.
\end{remark}


\medskip
\subsection{Measurable Functions} \label{sec:measurable}

The main result of this section is Theorem~\ref{thm:meas-diag}, which shows that any measurable\footnote{In this section, all topological spaces will be endowed with their Borel $\sigma$-algebra, that is, the smallest $\sigma$-algebra containing all open sets, except for Prop.~\ref{prop:abs-cont-diag} where we use the Lebesgue measure.} function $\rho:\Omega\to\mf p$, where $\Omega$ is any measurable space, can be diagonalized as $\rho=\Ad_U\circ\lambda$ with measurable functions $U:\Omega\to K$ and $\lambda:\Omega\to\mf a$. This generalizes the analogous result~\cite[Thm.~2.1]{Quintana14} for the unitary diagonalization of positive definite matrices. We then further generalize this result in Theorem~\ref{thm:meas-diag-comm} to show that a finite family of commuting measurable functions can be simultaneously measurably diagonalized.

In order to show that a single measurable function $\rho:\Omega\to\mf p$ can be measurably diagonalized, we will describe a stratification of $\mf p$ into embedded submanifolds with a simple structure. This stratification originates from an intuitive partition of the Weyl chamber $\mf w$.

\begin{remark} \label{rmk:weyl-faces}
Let $\mf w \subseteq \mf a$ be a closed Weyl chamber.
Then $\mf w$ is a polyhedral cone, that is, it is defined by a finite set of linear homogeneous inequalities on $\mf a$.
Indeed one can choose these inequalities such that each corresponds to a reflection in $W$ whose hyperplane defines a facet of $\mf w$.
Let $\mf w_s$, with $s\in S$ some index set, denote the finitely many open faces of $\mf w$, that is, the relative interiors\footnote{The relative interior of a subset $S$ of a vector space $V$ is the topological interior of $S$ seen as a subset of the affine hull of $S$, that is, the smallest affine subspace containing $S$. See~\cite[pp.~2-3]{Zalinescu02}.} of the closed faces. Then the $\mf w_s$ form a partition of $\mf w$.
\end{remark}

\begin{lemma} \label{lemma:same-stabilizers}
Let $\mf w \subseteq \mf a$ be a closed Weyl chamber and $x,y\in\mf w$. Then $x$ and $y$ belong the the same open face of $\mf w$ if and only if $K_x=K_y$, or equivalently, $W_x=W_y$.
\end{lemma}

\begin{proof}
First we argue that if $x,y\in\mf w$ have the same stabilizer in $W$, then they belong to the same open face of $\mf w$. Indeed, if they have the same stabilizer, then by Remark~\ref{rmk:weyl-faces}, they satisfy the same equalities in the inequality description of $\mf w$ and hence they belong the the same open face.
Next we show that if $x,y\in\mf w_s$ for some open face $\mf w_s$, then they have the same stabilizer in $K$.
By Kleiner's Lemma, see~\cite[Lemma~3.70]{Alexandrino15}, if $\rho:[0,1]\to\mf p$ is a geodesic segment realizing the distance between the orbits $\Ad_K(\rho(0))$ and $\Ad_K(\rho(1))$, then all points $\rho(t)$ for $t\in(0,1)$ have the same stabilizer in $K$. 
By Corollary~\ref{coro:geodesic-segment-in-w}, every line segment in $\mf w$ is of this type. 
Since the $\mf w_s$ are convex and relatively open, this shows that all points belonging to the same $\mf w_s$ have the same stabilizer in $K$.
Finally it is clear that if $x,y$ have the same stabilizers in $K$, then the same is true in $W$.
This concludes the proof.
\end{proof}

\begin{corollary} \label{coro:same-commutants}
Let $\mf w\subseteq\mf a$ be a closed Weyl chamber and let $\mf w_s$ be an open face of $\mf w$. If $x,y\in\mf w_s$, then $\mf p_x=\mf p_y$. 
\end{corollary}

\begin{proof}
For $z\in\mf a$, Corollary~\ref{coro:centralizer-stabilizer} with $A=\{z\}$ shows that $\mf p_z=\Ad_{K_z}(\mf a)$. 
For $x,y\in\mf w_s$, by Lemma~\ref{lemma:same-stabilizers} it holds that $K_x=K_y$ and hence $\mf p_x=\mf p_y$.
\end{proof}

\begin{lemma} \label{lemma:stratification-of-p}
Let $\mf w\subseteq\mf a$ be a closed Weyl chamber and denote $\mf p_s = \Ad_K(\mf w_s)$ and let $K_s$ denote the stabilizer in $K$ of the points in $\mf w_s$. Then the map $K/ K_s\times \mf w_s \to \mf p_s,\,(UK_s,x)\mapsto \Ad_U(x)$ is a $K$-equivariant diffeomorphism and $\mf p_s$ is an embedded submanifold.
\end{lemma}

\begin{proof}
First note that the $K_s$ are well-defined due to Lemma~\ref{lemma:same-stabilizers}.
Clearly the map $\phi:K/ K_s\times \mf w_s \to \mf p_s,\,(UK_s,x)\mapsto \Ad_U(x)$ is well-defined, smooth and $K$-equivariant and surjective. 
To see that it is injective, consider two points $(U_1 K_s, x)$ and $(U_2 K_s, y)$ mapped to the same point. 
Then $x=y$ since each orbit intersects $\mf w$ in exactly one point by Corollary~\ref{coro:intersect-weyl-chamber}, and $U_1^{-1} U_2 \in K_x = K_s$. 
Hence $U_1 K_s = U_2 K_s$. 
We show that the differential $D\phi(UK_s,x)$ is injective.
By equivariance, it suffices to consider $U=\id$. 
Then for $v\in T_x\mf w_s$ we obtain $D\phi(\id K_s,x)(0,v)=v\in\mf a$ and for $w\in \mf k/ \mf k_s\cong\mf k_x^\perp$ we obtain $D\phi(\id K_s,x)(w,0)=[w,x]\in\mf a^\perp$ by Lemma~\ref{lemma:inner-prod-centralizer}. 
Hence $D\phi(\id K_s,x)(w,v)$ only if $v=0$ and $w=0$.
Hence $\phi$ is an immersion.
To see that $\mf p_s$ is embedded, consider $x\in \mf w_s$. 
Consider a sequence of points $x_i\in\mf w_s$ and $U_i\in K$ such that $\Ad_{U_i}(x_i) \to x$. 
Then since the quotient map $\pi_{\mf a}$ is open, $x_i\to x$ and since the action is proper, a subsequence of $U_i$ converges to some $U\in K_s$. 
This shows that $\phi$ is an embedding.
\end{proof}

\begin{lemma}
\label{lemma:homogeneous-mfd-measurable-lift}
Let $G$ be a Lie group, $H$ a closed subgroup of $G$, and $\Omega$ a measurable space. If $\gamma:\Omega\to G/ H$ is measurable, then there exists a measurable lift $\tilde\gamma:\Omega\to G$.
\end{lemma}

\begin{proof}
Since the quotient map $\pi:G\to G/ H, \,g\mapsto gH$ is a smooth submersion (cf.~\cite[Thm.~21.17]{Lee13}), there is an open neighborhood $U$ of $e$ in $G$ and there are charts $\sigma:\R^n\to U$ and $\tau:\R^{n-k}\to \pi(U)\subseteq G/ H$ such that $\tau^{-1}\circ\pi\circ\sigma:\R^n\to\R^{n-k}$ is simply the projection onto the first $n-k$ coordinates. The sets of the form $\pi(gU)$ form an open cover of $G/ H$ and hence there exists a countable subcover whose open sets are $W_i :=\pi(g_i U)$ with $g_i\in G$ for $i\in\N$. Then define the sets $A_1=W_1$ and $A_i=W_i\setminus \bigcup_{m=1}^{i-1} A_m$ for $i\geq 2$, which form a countable partition of $G/ H$ consisting of measurable sets. Let $\Omega_i=\{\gamma\in A_i\}$ be the preimages, which form a countable measurable partition of $\Omega$. Then it suffices to find measurable lifts $\tilde\gamma_i:\Omega_i\to G/ H$ of each restriction $\gamma_i:=\gamma|_{\Omega_i}$. By definition, $g_i^{-1}\gamma_i$ takes image in $g_i^{-1} A_i\subseteq \pi(U)$. Using the chart $\tau$ this path can be seen as a measurable path in $\R^{n-k}$, which can be lifted to $\R^n$ using the inclusion $\iota:\R^{n-k}\to\R^n:x\to(x,0,\ldots,0)$. That is we define $\tilde\gamma_i=g_i\circ\sigma\circ\iota\circ\tau^{-1}\circ g_i^{-1}\circ\gamma_i$ and this concludes the proof.
\end{proof}

Putting everything together we can now prove the first main result of this section.

\begin{theorem}[Measurable Diagonalization]
\label{thm:meas-diag}
Let $\Omega$ be a measurable space and let $\rho:\Omega\to\mf p$ be measurable. Then there exist measurable functions $U:\Omega\to K$ and $\lambda:\Omega\to \mf a$ such that $\rho(\omega)=\Ad_{U(\omega)}(\lambda(\omega))$ for all $\omega\in \Omega$.
\end{theorem}

\begin{proof}
Let the $\mf w_s$ and $\mf p_s$ be as in Lemma~\ref{lemma:stratification-of-p}. 
Then the $\mf p_s$ yield a finite partition of $\mf p$ into measurable subsets. Then the sets $\Omega_s:=\{\rho\in\mf p_s\}\subseteq\Omega$ yield a finite partition of $\Omega$ into measurable subsets and it suffices to find measurable functions $U_s:\Omega_s\to K$ and $\lambda_s:\Omega_s\to\mf a$ satisfying $\Ad_{U_s(\omega)}(\lambda_s(\omega))=\rho(\omega)$ for all $\omega\in\Omega_s$. 
Let $\rho_s = \rho|_{\Omega_s}$, then by Lemma~\ref{lemma:stratification-of-p} one can consider $\rho_s$ as a measurable map $\Omega_s\to K/ K_s\times\mf w_s$. Then we define the measurable maps $\lambda_s=\Omega_s\to\mf a$ by $\lambda_s=\mathrm{pr}_2\circ\rho_s$ and $\tilde U_s:\Omega_s\to K/ K_s$ by $\tilde U_s=\mathrm{pr}_1\circ\rho_s$, and using Lemma~\ref{lemma:homogeneous-mfd-measurable-lift} we obtain a corresponding measurable map $U_s:\Omega_s\to K$.
\end{proof}

\begin{remark}
This generalizes~\cite[Thm.~2.1]{Quintana14} which shows that a measurable function of positive definite matrices can be unitarily diagonalized in a measurable way. 
\end{remark}

We can further strengthen this result by showing that finitely many commuting measurable functions $\rho_i:\Omega\to\mf p$ for $i=1,\ldots,n$ can be simultaneously measurably diagonalized. The proof will be based on induction on $i$. The idea will be to diagonalize $\rho_{i+1}$ using group elements which stabilize all the previously diagonalized paths.
To do this we need to work with symmetric Lie subalgebras of $\mf g$, which will in general not be semisimple, but still reductive.

First we will need two simple lemmas about Lie subgroups and restrictions of Lie group homomorphisms:

\begin{lemma} \label{lemma:lie-subgroup}
Let $G$ be a Lie group with Lie subgroups $H$ and $K$ satisfying the inclusion $K\subseteq H$. Then $K$ is a Lie subgroup of $H$.
\end{lemma}

\begin{proof}
The identity $G\to G$ descends to the inclusion $K\hookrightarrow H$, and by Lemma~\ref{lemma:restriction-Lie-hom} the latter is smooth, so its image is a Lie subgroup in $H$ by~\cite[Thm.~7.17]{Lee13}.
\end{proof}

\begin{lemma}
\label{lemma:restriction-homomorphism}
Let $X$ be a real, finite dimensional vector space and let $Y\subseteq X$ be a subspace. Let $G\subseteq \GL(X)$ be a Lie subgroup containing only elements which leave $Y$ invariant. Let $H\subseteq \GL(Y)$ be any Lie subgroup such that for $g\in G$ the restriction $g|_Y$ lies in $H$. Then the restriction $G\to H$ which maps $g\to g|_Y$ is a Lie group homomorphism.
\end{lemma}

\begin{proof}
Let $\GL(X,Y)\subseteq \GL(X)$ be the subgroup of elements which leave $Y$ invariant. 
If $P_Y$ is any idempotent linear map on $X$ with image $Y$, then $\GL(X,Y)=\{g\in\GL(X) : g P_Y = P_Yg P_Y\}$; hence it is a closed subgroup and thus an embedded Lie subgroup, see~\cite[Thm.~7.21]{Lee13}. 
Clearly $G$ is a subgroup of $\GL(X,Y)$, and by Lemma~\ref{lemma:lie-subgroup} and the above, it is a Lie subgroup.
Let $r:\GL(X,Y)\to\GL(Y)$ be the restriction map $g\mapsto g|_Y$, which is a Lie group homomorphism.
By Lemma~\ref{lemma:restriction-Lie-hom} it descends to a Lie group homomorphism $G\to H$, which concludes the proof.
\end{proof}

Now we can give the promised induction argument which will be the key ingredient for the following theorem.

\begin{lemma}
\label{lemma:measurable-diag-induction}
Let $\Omega$ be a measurable space and let $A\subseteq\mf p$ be any subset. If $\rho:\Omega\to\mf p_A$ is measurable, then there exists a measurable function $U:\Omega\to K_A$ such that $Ad_{U(\omega)}^{-1}(\rho(\omega))\in\mf a$ for all $\omega\in\Omega$.
\end{lemma}

\begin{proof}
By Lemma~\ref{lemma:reductive-lie-sub-alg} and Lemma~\ref{lemma:semisimple-subsystem} the commutant $\mf g_A$ is reductive and can be written as $\mf g_A=\mf h\oplus\mf z$ where $\mf h=[\mf g_A,\mf g_A]$ is the semisimple part and $\mf z$ is the center of $\mf g_A$.
Moreover $(\mf h, s|_{\mf h})$ is a semisimple, symmetric Lie subalgebra of $\mf g$, and by Lemma~\ref{lemma:orth-semisimple-converse} it is orthogonal. Its Cartan-like decomposition is $\mf h=\mf l\oplus\mf q$, where $\mf l=\mf h\cap\mf k$ and $\mf q=\mf h\cap\mf p$.
It holds that $\mf p_A = \mf q \oplus (\mf z\cap\mf p)$ since $\mf h$ and $\mf z$ are invariant under $s$. Let $\tilde\rho:\Omega\to\mf q$ be the component of $\rho$ in $\mf q$.
Let $\mf b$ denote a maximal Abelian subspace of $\mf q$ and note that without loss of generality $\mf b\subseteq\mf a$.
The first step is to diagonalize $\tilde\rho$ using the previous theorem. Let $(H=\Int(\mf h),L=\Int_{\mf l}(\mf h))$ be the canonical pair associated with $(\mf h, s|_{\mf h})$ as in Lemma~\ref{lemma:existence-of-pair}. By Theorem~\ref{thm:meas-diag} there exists a measurable path $\tilde U:\Omega\to L$ such that $\Ad_{\tilde U(\omega)}^{-1}(\tilde\rho(\omega))\in\mf b$ for all $\omega\in\Omega$.
The next step is to lift the path $\tilde U$ to $K_A$.
Since $\mf l\subseteq\mf k_A$, Lemma~\ref{lemma:lie-subgroup} shows that $L=\Int_{\mf l}(\mf h)$ is a Lie subgroup of $\Int_{\mf k_A}(\mf h)$, and so we can consider the path $\tilde U$ to take values in $\Int_{\mf k_A}(\mf h)$. 
Consider the adjoint representation of $(K_A)_0$ on $\mf h$, denoted by $\Ad|_{\mf h}:(K_A)_0\to \Int_{\mf k_A}(\mf h)$, which is surjective. We need to show that this is a Lie group homomorphism. 
Indeed, we can write this as a composition $(K_A)_0\to\Int_{\mf k_A}(\mf g_A)\to\Int_{\mf k_A}(\mf h)$. 
By Lemma~\ref{lemma:pairs-factor-canonical} the first map is a Lie group homomorphism. Since the adjoint representation of $\mf k_A$ preserves $\mf h$ and by Lemma~\ref{lemma:restriction-homomorphism} the second one is also a Lie group homomorphism.
Hence by the Lie group Isomorphism Theorem~\cite[Theorem~21.27]{Lee13}, we can consider the path $\tilde U$ to take values in $(K_A)_0/\ker(\Ad|_{\mf h})$, and finally by Lemma~\ref{lemma:homogeneous-mfd-measurable-lift} we obtain a measurable path $U:\Omega\to(K_A)_0$ satisfying $\Ad_{U(\omega)}^{-1}(\tilde\rho(\omega))\in\mf b$ for all $\omega\in\Omega$.
Finally we show that $U$ is the desired path. 
But this follows from the fact that $K_A$ leaves $\mf h$ and $\mf z$ invariant, and hence also $\mf q$ and $\mf z\cap\mf p$, and from the fact that $\mf z\cap\mf p\subseteq\mf a$.
\end{proof}

Now let us describe a partition of $\mf a$ which extends the decomposition of $\mf w$ of Remark~\ref{rmk:weyl-faces}. In fact we may simply take all relatively open faces of all Weyl chambers, removing duplicates of course. This yields a partition of $\mf a$ into finitely many subsets $\mf a_s$ with $s\in S$ for some new index set $S$. 
We generalize this partition of $\mf a$ to the $n$-fold Cartesian product $\mf a^n$. 
Consider some tuple $\mbf s=(s_i)_{i=1}^n\in S^n$ of indices in $S$.
We write $\mf a_{\mbf s}=\mf a_{s_1}\times\ldots\times\mf a_{s_n}$ and note that there are finitely many such sets and they are disjoint and cover $\mf a^n$.

\begin{corollary} \label{coro:stratification-a}
Let $\mbf s\in S^n$ and $x,y\in\mf a_{\mbf s}$. Then $x$ and $y$ have the same stabilizer in $K$ and the same commutant in $\mf p$, i.e., $K_{\{x_1,\ldots,x_n\}}=K_{\{y_1,\ldots,y_n\}}$ and $\mf p_{\{x_1,\ldots,x_n\}} = \mf p_{\{y_1,\ldots,y_n\}}$.
\end{corollary}

\begin{proof}
This follows from Lemma~\ref{lemma:same-stabilizers} and Lemma~\ref{coro:same-commutants}.
\end{proof}

The corollary shows that we may define the simultaneous stabilizer $K_{\mbf s}=K_{\{x_1,\ldots,x_n\}}$ and the simultaneous commutant $\mf p_{\mbf s}=\mf p_{\{x_1,\ldots,x_n\}}$. With this we can prove the second main result of this section.

\begin{theorem}[Simultaneous Measurable Diagonalization] \label{thm:meas-diag-comm}
Let $\Omega$ be a measurable space and let $\rho_i:\Omega\to\mf p$ be measurable for $i=1,\ldots,n$. Assume that $[\rho_i(\omega),\rho_j(\omega)]=0$ for all $\omega\in\Omega$ and for all $i,j\in\{1,\ldots,n\}$. Then there exists a measurable function $U:\Omega\to K$ such that $\Ad_{U(\omega)}^{-1}(\rho_i(\omega))\in\mf a$ for all $i=1,\ldots,n$ and for all $\omega\in\Omega$. 
\end{theorem}

\begin{proof}
We proceed by induction on $i$ by showing that if there exists a measurable $U:\Omega\to K$ such that $\Ad^{-1}_{U(\omega)}(\rho_j(\omega))\in\mf a$ for all $j\leq i$ and for all $\omega\in\Omega$, then there exists a measurable $\tilde U:\Omega\to K$ such that $\Ad^{-1}_{\tilde U(\omega)}(\rho_j(\omega))\in\mf a$ for all $j\leq i+1$ and for all $\omega\in\Omega$. 
The base case $i=1$ is exactly Theorem~\ref{thm:meas-diag}. 
Assume now that $1\leq i<n$ and let $U$ be such that $\Ad_U^{-1}\circ\,\rho_j\in\mf a$ for all $\omega\in\Omega$ and $j\leq i$. Now consider any subset $\mf a_{\mbf s}$ with $\mbf s\in S^i$ of the partition of $\mf a^i$ defined above. 
Then the set $\Omega_{\mbf s}=\{\Ad^{-1}_U(\rho_1,\ldots,\rho_i)\in \mf a_{\mbf s}\}$ is measurable and it suffices to show that we can diagonalize $\rho_{i+1}|_{\Omega_{\mbf s}}$. By Corollary~\ref{coro:stratification-a}, for all $\omega\in\Omega_{\mbf s}$, the set $\{\Ad^{-1}_{U(\omega)}(\rho_j(\omega)):j=1,\ldots,i\}$ will have the same stabilizer $K_{\mbf s}$ in $K$ and the same commutant $\mf p_{\mbf s}$ in $\mf p$. 
Hence for all $\omega\in\Omega_{\mbf s}$ it holds that $\Ad^{-1}_{U(\omega)}(\rho_{i+1}(\omega))\in \mf p_{\mbf s}$ and by Lemma~\ref{lemma:measurable-diag-induction} there exists a measurable path $\tilde U:\Omega_{\mbf s}\to K_{\mbf s}$ which diagonalizes $\Ad^{-1}_{U(\omega)}(\rho_{i+1}(\omega))$ on $\Omega_{\mbf s}$. 
This proves the induction step and concludes the proof.
\end{proof}

For our final result we specialize to the case where our measurable space is an interval $I$ with the Lebesgue measure and where $\rho:I\to\mf p$ is absolutely continuous. 

\begin{proposition} \label{prop:abs-cont-diag}
Let $I\subseteq\R$ be an interval. Let $\rho:I\to\mf p$ be absolutely continuous. Then there exists $U:I\to K$ measurable such that $\lambda(t)=\Ad^{-1}_{U(t)}(\rho(t))$ and $\mu(t)=\Ad^{-1}_{U(t)}(\Pi_{\rho(t)}(\rho'(t)))$ lie in $\mf a$ and $\mu(t)=\lambda'(t)$ for almost every $t\in I$. In fact we can ensure that $\lambda=\lambda^\down$.
\end{proposition}

\begin{proof}
First we show that $\Pi_{\rho(t)}(\rho'(t))$ is measurable. 
By Theorem~\ref{thm:meas-diag} there is a measurable $U:I\to K$ such that $\lambda(t):=\Ad_{U(t)}^{-1}(\rho(t))\in\mf a$. 
By Lemma~\ref{lemma:equivariance}~\eqref{it:equiv-proj} it suffices to show that $\Pi_{\lambda(t)}$ is measurable, and in fact we may show this on each $I_s=\{\lambda\in\mf w_s\}$ with the partition from Remark~\ref{rmk:weyl-faces}. Indeed, on these sets, $\Pi_{\lambda(t)}$ is a constant linear projection, so it is clearly measurable.
Since $\rho$ and $\Pi_{\rho(t)}(\rho'(t))$ are measurable and commute almost everywhere by construction, there is, by Theorem~\ref{thm:meas-diag-comm}, some measurable $\tilde U:I\to K$ such that $\tilde\lambda(t)=\Ad^{-1}_{\tilde U(t)}(\rho(t))$ and $\tilde\mu=\Ad^{-1}_{\tilde U(t)}(\Pi_{\rho(t)}(\rho'(t)))$ are measurable and lie in $\mf a$ almost everywhere. 
Now consider the path $\lambda^\down$ as defined in Proposition~\ref{prop:continuous-diagonalizations}, which is absolutely continuous by item~\eqref{it:diag-AC} of the same proposition.
Then the path $(\lambda^\down,(\lambda^\down)')$ in $T\mf a$ is measurable almost everywhere. By Proposition~\ref{prop:deriv-of-projected-path} there exists for almost every $t_0\in I$ some $\tilde w\in W$ such that $(\lambda^\down(t_0),(\lambda^\down)'(t_0))=\tilde w\cdot(\tilde\lambda(t_0),\tilde\mu(t_0))$. By Lemma~\ref{lemma:orb-meas} there is some measurable $w:I\to W$ such that 
$(\lambda^\down,(\lambda^\down)')=w\cdot(\tilde\lambda,\tilde\mu)$
almost everywhere.
Let $V:I\to K$ be a measurable lift of $w$, and define $U:I\to K$ as $\tilde UV^{-1}$. Then $U$ satisfies the desired properties and this concludes the proof.
\end{proof}

\section{Classification of Diagonalizations}
\label{sec:classification}

As illustrated in Examples~\ref{ex:hermitian-evd} and~\ref{ex:polar-dec}, the semisimple, orthogonal, symmetric Lie algebras correspond to various notions of diagonalization. In this section we first recall the classification of irreducible orthogonal symmetric Lie algebras and prove that every semisimple, orthogonal, symmetric Lie algebra is orbit equivalent to a direct sum of irreducible orthogonal symmetric Lie algebras (and a trivial part). This is the content of Theorem~\ref{thm:orbit-equivalence}. Then in Table~\ref{tab:diags} we give a list of diagonalizations corresponding to the irreducible orthogonal symmetric Lie algebras, as described in~\cite{Kleinsteuber}.

An orthogonal symmetric Lie algebra $\mf g=\mf k\oplus\mf p$ is called \emph{irreducible} if it is semisimple, strongly effective and irreducible, as defined in Appendix~\ref{app:symmetric}. Then~\cite[Ch.~VIII, Thms.~5.3,~5.4]{Helgason78} shows that there are exactly four types of irreducible orthogonal symmetric Lie algebras $(\mf g, s)$:
\begin{enumerate}[I.]
\item \label{it:type-I} $\mf g$ is a compact, simple Lie algebra over $\R$ and $s$ is any involutive automorphism of $\mf g$;
\item \label{it:type-II} $\mf g$ is a compact Lie algebra, and it is the Lie algebra direct sum $\mf g=\mf g_1\oplus\mf g_2$ of simple ideals, $s$ interchanges $\mf g_1$ and $\mf g_2$;
\item \label{it:type-III} $\mf g$ is a non-compact, simple Lie algebra over $\R$, $\mf g_\C$ is a simple Lie algebra over $\C$, and $\mf k$ is compactly embedded in $\mf g$;
\item \label{it:type-IV} $\mf g$ is a complex simple Lie algebra considered as a real Lie algebra and $s$ is the conjugation with respect to a maximal compactly embedded subalgebra.
\end{enumerate}
Moreover there is a duality between types~\ref{it:type-I} and~\ref{it:type-III} and between types~\ref{it:type-II} and~\ref{it:type-IV}, given in Lemma~\ref{lemma:cpt-non-cpt-duality}. This shows that the problem of classifying all irreducible orthogonal symmetric Lie algebras is equivalent to the classification of all simple Lie algebras over $\R$ and $\C$. We have summarized these well-known results in Tables~\ref{tab:complex-simple}~and~\ref{tab:real-simple}, where we omitted the exceptional Lie algebras for simplicity. For explanations of the notation see Remarks~\ref{rmk:simple-lie-algebra-reps} and~\ref{rmk:notation}.

\begin{table}[h]
\centering
\begin{adjustbox}{center}
\begin{tabular}{|c|c|c|c|c|c|}
\hline
\hline
Label & $\mf g$ & $\mf k$ & $s(X)$ & $\mf p$ & $\mf a$ \\
\hline
\hline
A & $\mathfrak{sl}(n+1,\C)$ & $\su(n+1)$ & $-X^*$ & $\mf{herm}_0(n+1,\C)$ & $\mf d_0(n+1,\R)$ \\
\hline
B & $\so(2n+1,\C)$ & $\so(2n+1)$ & $-X^*$ & $\iu\,\mf{asym}(2n+1,\R)$ &  $\mf d(n,\iu\R)\otimes J_1\oplus 0_1$ \\
\hline
C & $\mathfrak{sp}(n,\C)$ & $\mathfrak{sp}(n)$
& $-X^*$ & 
$\big(\begin{smallmatrix}X&\bar Y\\Y&-\bar X\end{smallmatrix}\big),\,\begin{smallmatrix}X\in\mf{herm}(n,\C) \\ Y\in\mf{sym}(n,\C)\end{smallmatrix}$
& $X\in\mf d(n,\R),Y=0$ \\
\hline
D & $\so(2n,\C)$ & $\so(2n)$ & $-X^*$ & $\iu\,\mf{asym}(2n,\R)$ &  $\mf d(n,\iu\R)\otimes J_1$ \\
\hline
\hline
\end{tabular}
\end{adjustbox}
\caption{\textbf{Irreducible orthogonal symmetric Lie algebras (Types~\ref{it:type-II} and~\ref{it:type-IV}).} We list the simple Lie algebras $\mf g$ over $\C$ and a maximal compactly embedded subalgebra $\mf k$. Then $s$ is the corresponding Cartan involution and $\mf p=\iu \mf k$ is the $-1$ eigenspace. Moreover $\mf a$ is a maximal Abelian subspace of $\mf p$, and its complexification is a Cartan subalgebra of $\mf g$. These are the irreducible orthogonal symmetric Lie algebras of type~\ref{it:type-IV}. The corresponding compact irreducible orthogonal symmetric Lie algebras of type~\ref{it:type-II} are then $\mf k \oplus \mf k$ and $s$ simply interchanges the terms. See also~\cite[Ch.~III~\S8, Ch.~X]{Helgason78}.}
\label{tab:complex-simple}
\end{table}

\bgroup
\def\arraystretch{1.5}
\begin{table}[h]
\centering
\begin{adjustbox}{center}
\begin{tabular}{|c|c|c|c|c|c|}
\hline
\hline
Label &$\mf g$ & $\mf k$ & $s(X)$ & $\mf p$ & $\mf a$ \\
\hline
\hline
AI & $\mathfrak{sl}(n,\R)$ & $\so(n)$ & $-X^\top$ & $\mf{sym}_0(n,\R)$ & $\mf d_0(n,\R)$ \\
\hline
AII & $\su^*(2n)$\tablefootnote{The name comes from the fact that the dual symmetric Lie algebra is $\su(2n)$. Note that $\su^*(2n)$ is isomorphic to $\mathfrak{sl}(n,\H)$ via the standard embedding $\jmath$. Similarly the corresponding $\mf p$ part equals $\jmath(\mf{herm}_0(n,\H))$.} & $\mathfrak{sp}(n)$ & $-X^*$ 
& $\big(\begin{smallmatrix}X&-\bar Y\\Y&\bar X\end{smallmatrix}\big),\,\begin{smallmatrix}X\in\mf{herm}_0(n,\C) \\ Y\in\mf{asym}(n,\C)\end{smallmatrix}$
& $X\in\mf d_0(n,\R),Y=0$ \\
\hline
AIII & $\su(p,q)$ & $\mf s(\mathfrak{u}(p)\oplus\mathfrak{u}(q))$ & $-X^*$ & $\big(\begin{smallmatrix}0&Y\\Y^*&0\end{smallmatrix}\big),\, Y\in\C^{p,q}$ & $Y\in\mf d(p,q,\R)$ \\
\hline
BDI & $\so(p,q)$ & $\so(p)\oplus\so(q)$ & $-X^\top$ & $\big(\begin{smallmatrix}0&Y\\Y^\top&0\end{smallmatrix}\big),\, Y\in\R^{p,q}$ & $Y\in\mf d(p,q,\R)$ \\
\hline
CI & $\mathfrak{sp}(n,\R)$ & $\imath(\mathfrak{u}(n))$ & $-X^\top$ & 
$\big(\begin{smallmatrix}X&Y\\Y&-X\end{smallmatrix}\big),\,X,Y\in\mf{sym}(n,\R)$
& $X\in\mf d(n,\R),Y=0$ \\ 
\hline
CII & $\mathfrak{sp}(p,q)$ & $\mathfrak{sp}(p)\oplus\mathfrak{sp}(q)$ & $-X^*$ & $\jmath(\big(\begin{smallmatrix}0&Y\\Y^*&0\end{smallmatrix}\big)),\, Y\in\H^{p,q}$ & $Y\in\mf d(p,q,\R)$ \\
\hline
DIII & $\so^*(2n)$\tablefootnote{Again the name stems from the dual symmetric Lie algebra $\so(2n)$.} & $\imath(\mathfrak{u}(n))$ & $-X^*$ & 
$\big(\begin{smallmatrix}X&Y\\Y&-X\end{smallmatrix}\big),\,X,Y\in\iu\mf{asym}(n,\R)$
 & $X\in\mf d(\floor{\tfrac{n}{2}},\iu\R)\otimes J_1(\oplus 0_1),Y=0$ \\
\hline
\hline
\end{tabular}
\end{adjustbox}
\caption{\textbf{Irreducible orthogonal symmetric Lie algebras (Types~\ref{it:type-I} and~\ref{it:type-III}).} We list the simple Lie algebras $\mf g$ over $\R$ with a Cartan involution $s$, the corresponding Cartan decomposition $\mf k\oplus\mf p$, and a Cartan subalgebra $\mf a$. These are the irreducible orthogonal symmetric Lie algebras of type~\ref{it:type-III}. The corresponding compact irreducible orthogonal symmetric Lie algebras of type~\ref{it:type-I} are easily obtained via duality.
See also~\cite[Ch.~X \S2.3]{Helgason78}.}
\label{tab:real-simple}
\end{table}
\egroup 

\begin{remark} \label{rmk:simple-lie-algebra-reps}
The Lie algebras in Tables~\ref{tab:complex-simple}~and~\ref{tab:real-simple} can be represented in different but equivalent ways. We use the definitions given in~\cite[p.~446]{Helgason78}. Hence they are all real or complex matrix Lie algebras. 
\end{remark}

\begin{remark}
\label{rmk:notation}
Let $\K=\R$, $\C$, or $\H$. For $x\in\K$, $\overline x$ denotes the (complex or quaternionic) conjugate. For a matrix $X\in\K^{m,n}$, $\overline X$ denotes the elementwise conjugate, $X^\top$ denotes the transposed matrix, and $X^*=\overline X^\top$ denotes the Hermitian conjugate.
Then $\mf{sym}(n,\K)=\{X\in \K^{n,n} : X=X^\top\}$ denotes the set of all symmetric matrices. Similarly $\mf{asym}(n,\K)=\{X\in \K^{n,n} : X=-X^\top\}$ denotes the set of all skew-symmetric matrices. Moreover $\mf{herm}(n,\K)=\{X\in \K^{n,n} : X=X^*\}$ denotes the set of all Hermitian matrices. If we additionally assume that the matrices are traceless, we write $\mf{sym}_0(n,\K)$ and $\mf{herm}_0(n,\K)$. Finally, diagonal matrices are denoted by $\mf d(m,n,\K)$, or $\mf d(n,\K)$, and with subscript $0$ if the diagonal elements add up to $0$.
Furthermore we define some useful matrices:
$$
I_{n,m}=\begin{bmatrix}
I_n&0\\0&-I_m
\end{bmatrix},\quad
J_n=\begin{bmatrix}
0&I_n\\-I_n&0
\end{bmatrix},
$$
where $I_n$ denotes the identity matrix of size $n$. Similarly we will write $0_n$ for the zero matrix of size $n$ if the size is not clear from context. We will also use the standard embeddings
\begin{align}
\imath:\C\to\R^{2,2}, \quad x+\iu y
\mapsto \begin{bmatrix}x&-y\\y&x\end{bmatrix}
\end{align}
and
\begin{align}
\jmath:\H\to\C^{2,2}, \quad a+\iu b+\ju c+\ku d = \alpha+\ju\beta
\mapsto \begin{bmatrix}\alpha&-\overline\beta\\\beta&\overline\alpha\end{bmatrix}
\end{align}
where $\alpha=a+\iu b$ and $\beta=c-\iu d$, which can analogously be defined to act on matrices.
\end{remark}

Since the irreducible orthogonal symmetric Lie algebras can be fully classified, it is natural to ask under which conditions an orthogonal, symmetric Lie algebra can be decomposed in some sense into such irreducible pieces.
We have seen that an effective, orthogonal, symmetric Lie algebra can be decomposed into a Euclidean, a compact, and a non-compact part, cf. Lemma~\ref{lemma:decomp-of-eos-lie-alg}. Similarly, a semisimple, strongly effective, orthogonal, symmetric Lie algebra can be decomposed into a Euclidean part and a direct sum of irreducible orthogonal symmetric Lie algebras, see~\cite[Ch.~VIII, Prop.~5.2]{Helgason78}. The case we are mostly interested in, semisimple, orthogonal, symmetric Lie algebras, lies between these two cases. In the following we will show that a semisimple, orthogonal, symmetric Lie algebra is still orbit equivalent to a direct sum of irreducible orthogonal symmetric Lie algebras (and a trivial part). This is the content of Theorem~\ref{thm:orbit-equivalence}. 

First we make the concept of orbit equivalence precise.

\begin{Definition}
Let $X$ be a set and let $G,H$ be groups acting on $X$. We say that the actions are \emph{orbit equivalent} if they have the same set of orbits, that is $X/G=X/H$.
\end{Definition}

By Corollary~\ref{coro:wx-kx} it holds for semisimple, orthogonal, symmetric Lie algebras that the orbits in $\mf p$ do not depend on the choice of associated pair. Thus we will often simply choose to work with the canonical associated pair. This also allows us to define orbit equivalence for symmetric Lie algebras with the same $\mf p$:

\begin{Definition}
Let $\mf g_i=\mf k_i\oplus\mf p$ for $i=1,2$ be symmetric Lie algebras. They are \emph{orbit equivalent} if the representations of $K_i=\Int_{\mf k_i}(\mf g_i)$ on $\mf p$ are orbit equivalent.
\end{Definition}

Now we show how $\mf p$ splits into irreducible representations and that the maximal Abelian subspace $\mf a$ as well as the tangent space to the orbit $\ad_{\mf k}(x)$ for $x$ regular respect this decomposition.

\begin{lemma} \label{lemma:orthogonal-invariant-commute}
Let $(\mf g, s)$ be a semisimple, symmetric Lie algebra and let $V,W\subseteq\mf p$ be $\ad_{\mf k}$-invariant and orthogonal with respect to the Killing form of $\mf g$. Then they commute.
\end{lemma}

\begin{proof}
Let $v\in V$ and $w\in W$. Then $B([v,w],[v,w]) = B(v,[w,[v,w]]) = 0$ and hence $[v,w]=0$.
\end{proof}

\begin{lemma}
\label{lemma:mfa-decomposition}
Let $(\mf g, s)$ be a semisimple, orthogonal, symmetric Lie algebra with Cartan-like decomposition $\mf g=\mf k\oplus\mf p$. Then there is an orthogonal decomposition $\mf p=\bigoplus_{i=1}^n\mf p_i$ into irreducible components for the action of $\ad_{\mf k}$ such that $\mf p_i\subseteq\mf p_\pm$ for each $i=1,\ldots,n$. If $\mf a_i=\mf p_i \cap\mf a$, then it holds that $\mf a=\bigoplus_{i=1}^n\mf a_i$.
\end{lemma}

\begin{proof}
We use the inner product from Lemma~\ref{lemma:nice-inner-prod}, in which all $\ad_k$ for $k\in\mf k$ are skew-symmetric. 
First using Corollary~\ref{coro:decomp-of-semisimple}, $\mf p$ splits into $\mf p_-\oplus\mf p_+$, the compact and non-compact part, which are invariant under the action of $\mf k$.
Hence there exists an orthogonal decomposition $\mf p=\bigoplus_{i=1}^n\mf p_i$ into irreducible components for the action of $\mf k$ such that each $\mf p_i$ is contained in $\mf p_-$ or $\mf p_+$.
Then the $\mf p_i$ are also orthogonal with respect to the Killing form $B$ of $\mf g$ since for $x_i\in\mf p_i$ and $x_j\in\mf p_j$ it holds that $B(x_i,x_j)=\pm\braket{x_i,x_j}$.
Hence, by Lemma~\ref{lemma:orthogonal-invariant-commute}, the $\mf p_i$ commute.
Now let $x\in\mf a$, and $x=\sum_{i=1}^n x_i$ for $x_i\in\mf p_i$. We show that all $x_i\in\mf a_i\subseteq\mf a$. 
Let $K=\Int_{\mf k}(\mf g)$ as in Lemma~\ref{lemma:existence-of-pair}.
Since all $x_i$ commute with each other, there is $U\in K$ such that $U\cdot x_i\in\mf a$ and hence $U\cdot x=\sum_i U\cdot x_i\in\mf a$. 
By left multiplying $U$ with an appropriate element from the normalizer $N_K(\mf a)$ we may assume without loss of generality that $U\cdot x=x$ and still $U\cdot x_i\in \mf a$.
Hence by the invariance of the $\mf p_i$ under $K$ it holds that $U\cdot x_i\in\mf p_i$ and so $x_i=U\cdot x_i$ and hence $x_i\in\mf a_i$, as desired.
\end{proof}

\begin{corollary} \label{coro:orbit-tangent-splitting}
We use the same notation as in Lemma~\ref{lemma:mfa-decomposition}. Let $x=\sum_{i=1}^n x_i\in\mf a$ be regular. Then it holds that $\ad_{\mf k}(x) = \bigoplus_{i=1}^n \ad_{\mf k}(x_i)$.
\end{corollary}

\begin{proof}
Regular elements exist by Lemma~\ref{lemma:regular-element}.
By Lemma~\ref{lemma:inner-prod-centralizer} it holds that $\mf a^\perp=\ad_{\mf k}(x)$, 
and by Lemma~\ref{lemma:mfa-decomposition} we have the orthogonal decomposition $\mf p=\bigoplus_{i=1}^n\mf p_i$.
Then for $y\in\mf p_i$ it holds that $\braket{y,\ad_{\mf k}(x)}=\braket{y,\ad_{\mf k}(x_i)}$.
Hence $y$ is orthogonal to $\ad_{\mf k}(x_i)$ if and only if $y\in\mf a_i$, that is, we have the orthogonal decomposition $\mf p_i=\mf a_i\oplus\ad_{\mf k}(x_i)$.
This shows that $\ad_{\mf k}(x)=\mf a^\perp=\bigoplus_{i=1}^n\ad_{\mf k}(x_i)$, as desired.
\end{proof}


This local result on the splitting of the tangent space of an orbit can be generalized to the entire orbit, showing that the semisimple, orthogonal, symmetric Lie algebra is orbit equivalent to a direct sum of reductive ones with irreducible isotropy representations.

\begin{lemma} \label{lemma:reductive-orbit-equivalence}
We use the same notation as in Lemma~\ref{lemma:mfa-decomposition}. 
For each $i$, $\mf g_i=\mf k\oplus\mf p_i\subseteq\mf g$ is a reductive, orthogonal, symmetric Lie subalgebra,
and so is $(\mf g', s') = \bigoplus_{i=1}^n(\mf g_i,s|_{\mf g_i})$.
Since $\mf p=\bigoplus_{i=1}^n\mf p_i$ the isotropy representations of $(\mf g, s)$ and of $(\mf g', s')$ act on the same space, and in fact they are orbit equivalent.
\end{lemma}

\begin{proof}
It is clear that $(\mf g_i,s|_{\mf g_i})$ is a symmetric Lie subalgebra and by Lemma~\ref{lemma:mfa-decomposition} it holds that $\mf p_i\subseteq\mf p_\pm$. Hence it is reductive and orthogonal by Lemma~\ref{lemma:reductive-lie-sub-alg}. 
By Lemma~\ref{lemma:algebra-sum-orth}, also $(\mf g', s')$ is a reductive, orthogonal, symmetric Lie algebra.
Let $K=\Int_{\mf k}(\mf g)$ and $K'=\Int_{\mf k'}(\mf g')$ be the respective compact Lie groups acting on $\mf p$.
Note that on each $\mf p_i$, $K$ and $K'$ generate the same orbits since $K$ and $K'$ are connected and $\ad_{\mf k}|_{\mf p_i}=\ad_{\mf k'}|_{\mf p_i}$. 
Hence, in $\mf p$, each $K$-orbit lies in some $K'$-orbit.
Let $x\in\mf a$ be regular in $(\mf g, s)$, then Corollary~\ref{coro:orbit-tangent-splitting} shows that the tangent space at $x$ of the $K$ and $K'$ orbits through $x$ is the same for both isotropy representations. Since the orbits $Kx$ and $K'x$ though $x$ satisfy  $Kx\subseteq K'x$, and since they are connected and by the previous argument have the same dimension, they must coincide.
This shows that the orbits of regular points of $(\mf g, s)$ coincide. It remains to show the same for singular orbits.
Since the regular points are Zariski open in $\mf p$, they are dense in the standard topology.
Now let $y,z\in\mf p$ be non-regular for $(\mf g, s)$  with distinct $K$-orbits. Let $N_y$ and $N_z$ be disjoint tubular neighborhoods in $\mf p$ for the action of $K$, see~\cite[Thm.~3.57]{Alexandrino15}.
If there is $U\in K'$ such that $U\cdot y=z$ then $U$ also maps some regular points in $N_y$ to $N_z$ which gives a contradiction. This concludes the proof.
\end{proof}

\begin{lemma} \label{lemma:reductive-orthogonal-splitting}
Let $(\mf g,s)$ be a reductive, orthogonal, symmetric Lie algebra with $\ad:\mf k\to\mf{gl(\mf p)}$ irreducible. 
Then $\mf g$ is a direct sum of symmetric Lie subalgebras $\mf g', \mf g'' \subseteq \mf g$ where $(\mf g',s|_{\mf g'})$ is an irreducible orthogonal symmetric Lie algebra and $(\mf g'',s|_{\mf g''})$ has trivial isotropy representation. 
\end{lemma}

\begin{proof}
Since $\mf g$ is reductive, by Lemma~\ref{lemma:reductive-decomp} it can be written as a direct sum of Lie subalgebras $[\mf g,\mf g]$ and $\mf z$ where the former is semisimple and the latter is the center of $\mf g$. Clearly they are symmetric Lie subalgebras of $\mf g$, and by Lemma~\ref{lemma:orth-semisimple-converse}, or by Lemma~\ref{lemma:algebra-sum-orth}, $[\mf g,\mf g]$ is orthogonal.
Since $\ad:\mf k\to\mf{gl(\mf p)}$ is irreducible there are two possibilities:
Either $[\mf g,\mf g]\cap\mf p$ is zero, in which case we can set $\mf g'=0$ and $\mf g''=\mf g$. 
Otherwise, $\mf z\cap\mf p$ is zero. 
Then since $[\mf g,\mf g]$ is semisimple it is the direct sum of simple ideals $\mf g_i$  for $i=1,\ldots,m$. 
For each $i$ there is some $j$ such that $s(\mf g_i)=\mf g_j$. 
Consider the semisimple, orthogonal, symmetric Lie subalgebras $\mf h_i = \mf g_i\oplus s(\mf g_i)$ (without repetitions). 
Denote their Cartan-like decomposition by $\mf h_i = \mf l_i\oplus\mf q_i$. 
Then, by irreducibility of $\mf g$, all but one $\mf q_i$ is zero. 
Without loss of generality we say that $\mf q_1$ is non-zero. 
Then the adjoint action of $\mf l_1$ on $\mf q_1$ is irreducible and effective.
Hence $\mf h_1$ is an irreducible orthogonal symmetric Lie algebra so we set $\mf g' = \mf h_1$, and the remaining $\mf h_i=\mf l_i$ and $\mf z$ together yield $\mf g''$.
\end{proof}

Putting it all together, we obtain the main theorem of this section. This result is similar to~\cite[Thm.~4]{Dadok85}, which considers more general polar actions.

\begin{theorem}[Classification of Diagonalizations] \label{thm:orbit-equivalence}
Let $(\mf g, s)$ be a semisimple, orthogonal, symmetric Lie algebra. Then there exist irreducible orthogonal symmetric Lie subalgebras $(\mf g_i,s|_{\mf g_i})$ for $i=1,\ldots,n$ and a symmetric Lie subalgebra $(\mf g',s_{\mf g})$ with trivial isotropy representation such that $\mf p=\bigoplus_{i=1}^n \mf p_i \oplus \mf p'$ and such that the isotropy representation of $(\mf g,s)$ is orbit equivalent to the isotropy representation of $\bigoplus_{i=1}^n (\mf g_i,s_i) \oplus (\mf g',s')$.
\end{theorem}

\begin{proof}
First let $\mf p=\bigoplus_{i=1}^n\mf p_i$ be an orthogonal decomposition into irreducible subrepresentations for the action of $\ad_{\mf k}$ on $\mf p$. 
Then by Lemma~\ref{lemma:reductive-orbit-equivalence}, each $\mf g_i=\mf k\oplus\mf p_i\subseteq\mf g$ is a reductive, orthogonal, symmetric Lie subalgebra of $(\mf g, s)$ and their direct sum $(\mf g', s') = \bigoplus_{i=1}^n(\mf g_i,s|_{\mf g_i})$ has isotropy representation orbit equivalent to that of $(\mf g,s)$. By Lemma~\ref{lemma:reductive-orthogonal-splitting} each $\mf g_i$ is a direct sum of two symmetric Lie subalgebras $\mf g_i'$, and $\mf g_i''$, where the first is an irreducible orthogonal symmetric Lie algebra and the second has a trivial isotropy representation. Hence $(\mf g', s') = \bigoplus_{i=1}^n (\mf g_i',s|_{\mf g_i'})\oplus (\mf g_i'',s|_{\mf g_i'})$. This is the desired decomposition.
\end{proof}

Now that we understand in which sense an arbitrary semisimple, orthogonal, symmetric Lie algebra can be decomposed into irreducible parts, we have a look at these irreducible cases and their relation to diagonalizations. In Table~\ref{tab:diags} we give a list of diagonalizations corresponding to the irreducible orthogonal symmetric Lie algebras. Recall Examples~\ref{ex:hermitian-evd} and~\ref{ex:polar-dec} for some detailed examples.
Theorem~\ref{thm:orbit-equivalence} shows that the diagonalizations shown in Table~\ref{tab:diags} are essentially the only possibilities (omitting  diagonalizations stemming from exceptional Lie algebras). This means that if we are given a semisimple, orthogonal, symmetric Lie algebra, we may think of it as a direct sum of irreducible parts. In particular we have a decomposition $\mf p=\bigoplus_{i=1}^n\mf p_i$, and similarly $\mf a=\bigoplus_{i=1}^n\mf a_i$. If $\rho:I\to\mf p$ is a path, we can compute the diagonalization in each $\mf p_i$ individually, that is we consider the paths $\rho_i(t)\in\mf p_i$ and compute their diagonalizations $\lambda_i(t)\in\mf a_i$, which can be done in practice using algorithms for the various diagonalizations in Table~\ref{tab:diags}. Note however that it is not straightforward to find a diagonalizing $U\in K$ form the individual $U_i$.

\begin{landscape}
\begin{table}[ht]
\centering
\begin{adjustbox}{center}
\begin{tabular}{|c|c|c|c|c|}
\hline
\hline
Name & Group & Matrices & Diagonal form & Type \\
\hline
\hline
real EVD & $\SO(n)$ & $\mf{sym}_0(n,\R)$ & $\mf d_0(n,\R)$ & AI \\
\hline
complex EVD & $\SU(n)$ & $\mf{herm}_0(n,\C)$ & $\mf d_0(n,\R)$ & A \\
\hline
quaternionic EVD & $\Sp(n)$ & $\jmath(\mf{herm}_0(n,\H))$ & $\jmath(\mf d_0(n,\R))$ & AII \\
\hline
skew-symmetric EVD & $\SO(2n+1)$ & $\mf{asym}(2n+1,\R)$ & $\mf d(n,\R)\otimes J_1\oplus 0_1$ & B\\
\hline
skew-symmetric EVD & $\SO(2n)$ & $\mf{asym}(2n,\R)$ & $\mf d(n,\R)\otimes J_1$ & D\\
\hline
\hline
real SVD & $\operatorname{S}(\operatorname{O}(p)\times\operatorname{O}(q))$ & $\R^{p,q}$ & $\mf d(p,q,\R)$ & BDI \\
\hline
complex SVD & $\operatorname{S}(\U(p)\times\U(q))$ & $\C^{p,q}$ & $\mf d(p,q,\R)$ & AIII \\
\hline
quaternionic SVD & $\Sp(p)\times\Sp(q)$ & $\jmath(\H^{p,q})$ & $\jmath(\mf d(p,q,\R))$ & CII \\
\hline
\hline
real symmetric Hamiltonian EVD & $\imath(\U(n))$ & $\mf{sym}(2n,\R)\cap\mf{ham}(n,\R)$ & $I_{1,1}\otimes\mf d(n,\R)$ & CI\\ 
\hline
Hermitian Hamiltonian EVD & $\Sp(n)$ & $\mf{herm}(2n,\C)\cap\mf{ham}(n,\C)$ & $I_{1,1}\otimes\mf d(n,\R)$ & C \\
\hline
Hermitian $*$-Hamiltonian EVD & $\operatorname{S}(\U(n)\times\U(n))$ & $\mf{herm}(2n,\C)\cap\mf{ham}^*(n,\C)$ & $I_{1,1}\otimes\mf d(n,\R)$ & AIII \\
\hline
Autonne-Takagi factorization & $\U(n)$ & $\mf{sym}(n,\C)$ & $\mf d(n,\R)$ & CI \\
\hline
Hua factorization & $\U(2n+1)$ & $\mf{asym}(2n+1,\C)$ & $\mf d(n,\R)\otimes J_1\oplus 0$ & DIII \\
\hline
Hua factorization & $\U(2n)$ & $\mf{asym}(2n,\C)$ & $\mf d(n,\R)\otimes J_1$ & DIII \\
\hline
\hline
\end{tabular}
\end{adjustbox}
\caption{\textbf{Various Matrix Diagonalizations.} We list several matrix diagonalizations and the corresponding symmetric Lie algebras. Note that some of these are isomorphic in the sense that they correspond to the same symmetric Lie algebra, but the isomorphism is not necessarily obvious. For more detail on the explicit expressions we refer to~\cite{Kleinsteuber}. The eigenvalue decompositions of so-called perplectic matrices described in~\cite{Mackey05} also fit into this setting, see~\cite[Sec.~4.2]{Kleinsteuber}. For explanations of the notation see Remarks~\ref{rmk:notation} and~\ref{rmk:notation2}.}
\label{tab:diags}
\end{table}
\end{landscape}

\begin{remark} \label{rmk:notation2}
The Hamiltonian matrices are defined as $\mf{ham}(n,\K)=\{X\in\K^{2n,2n}:J_n A=-A^\top J_n\}$. 
Note that $\mf u(n)=\iu\,\mf{herm}(n,\C)$, and $\mf{su}(n)=\iu\,\mf{herm}_0(n,\C)$, and $\mf{so}(n,\K)=\mf{asym}(n,\K)$, and for $\K=\R,\C$, we have $\mf{sp}(n,\K)=\mf{ham}(n,\K)$.
Moreover let us define the $*$-Hamiltonian matrices $\mf{ham}^*(n,\K) =\{X\in\K^{2n,2n} : J_nA=-A^*J_n\}$.
\end{remark}

\bigskip
\textbf{Acknowledgments.}
This research is part of the Bavarian excellence network \textsc{enb}
via the International PhD Programme of Excellence
\textit{Exploring Quantum Matter} (\textsc{exqm}), as well as the \textit{Munich Quantum Valley} of the Bavarian
State Government with funds from Hightech Agenda \textit{Bayern Plus}.

\appendix

\section{Symmetric Lie Algebras} \label{app:symmetric}

In this appendix we give a rigorous introduction to symmetric Lie algebras and prove a number of auxiliary results which are used repeatedly in the main text. Our definitions follow the standard reference of Helgason~\cite{Helgason78}.

\medskip
\subsection{Basic Definitions}

We start by considering symmetric Lie algebras, as defined in~\cite[p.~229]{Helgason78}.

\begin{Definition}
[Effective orthogonal symmetric Lie algebra]
Let $\mf g$ be a real finite-dimensional Lie algebra and $s$ an involutive\footnote{A map $f$ is involutive if $f\circ f$ is the identity.} Lie algebra automorphism of $\mf g$. Then the pair $(\mf g, s)$ is called a \emph{symmetric Lie algebra}. Now if $\mf k \subseteq \mf g$ denotes the fixed point set of $s$, we define the following:
\begin{enumerate}[(i)]
\item If $\mf k$ is a compactly embedded\footnote{This means that the analytic subgroup of $\GL(\mf g)$ with Lie algebra $\ad_{\mf k}$---denoted by $\Int_{\mf k}(\mf g)$---is compact, see~\cite[p.~130]{Helgason78}. More details on this group will be given later in this section.}
subalgebra of $\mf g$, then $(\mf g, s)$ is \emph{orthogonal}.
\item If $\mf k\cap\mf z=\{0\}$, where $\mf z$ denotes the center of $\mf g$, then $(\mf g, s)$ is \emph{effective}.
\item If $\mf k$ does not contain a non-trivial ideal of $\mf g$, then $(\mf g, s)$ is \emph{strongly effective}.
\end{enumerate}
\end{Definition}

\begin{remark} \label{rmk:semisimple-effective}
Note that if $\mf g$ is semisimple, then the center of $\mf g$ is trivial, i.e., $\mf z=\{0\}$, and hence $(\mf g, s)$ is automatically effective.
\end{remark}

A key feature of symmetric Lie algebras is that they admit a \emph{Cartan-like decomposition} $\mf g=\mf k\oplus\mf p$ with special commutator relations given in~\eqref{eq:cartan-commutators} below. In fact such a decomposition automatically yields the structure of a symmetric Lie algebra:

\begin{lemma}
Let $\mf g$ be a real Lie algebra. The following statements hold:
\begin{enumerate}[(i)]
\item \label{it:commutator} If $s$ is an involutive Lie algebra automorphism of $\mf g$ (i.e.~$(\mf g,s)$ is a symmetric Lie algebra) then $\mf g=\mf k\oplus\mf p$
(as a direct sum of subspaces) where $\mf k$ and $\mf p$ are the $+1$ and $-1$ eigenspaces of $s$, respectively. It holds that
\begin{align} \label{eq:cartan-commutators}
[\mf k,\mf k]\subseteq\mf k,\quad
[\mf k,\mf p]\subseteq\mf p,\quad
[\mf p,\mf p]\subseteq\mf k.
\end{align}
\item \label{it:commutator-converse} If $\mf g$ admits a vector space decomposition $\mf g = \mf k\oplus\mf p$ satisfying~\eqref{eq:cartan-commutators}, then the linear map $s$ defined as $+\id$ on $\mf k$ and $-\id$ on $\mf p$ is a Lie algebra automorphism, and hence $(\mf g, s)$ is a symmetric Lie algebra.
\end{enumerate}
\end{lemma}

\begin{proof}
\eqref{it:commutator}: Since $s$ is a linear involution on $\mf g$ it satisfies $s^2-\id=0$. The minimal polynomial of $s$ divides $(x+1)(x-1)$ and hence it splits into distinct linear factors, showing that $s$ is diagonalizable. Thus $\mf g = \mf k\oplus\mf p$ where $\mf k,\mf p$ are the corresponding $\pm1$ eigenspaces. Since $s$ is a Lie algebra automorphism it satisfies the relations~\eqref{eq:cartan-commutators}.

\eqref{it:commutator-converse}: By definition $s$ is a linear involution. Now given $k,l\in\mf k$ and $x,y\in\mf p$
\begin{align*}
s([k+x,l+y])
&= s([k,l])+s([k,y])+s([x,l])+s([x,y])\\
&= [k,l]-[k,y]-[x,l]+[x,y]\\
&= [k,l]+[k,-y]+[-x,l]+[-x,-y]\\
&=[k-x,l-y]=[s(k+x),s(l+y)]\,,
\end{align*}
showing that $s$ is a Lie algebra endomorphism. Since it is bijective, it is in fact an automorphism.
\end{proof}

Due to this result, we will often specify a symmetric Lie algebra by its Cartan-like decomposition, instead of by the corresponding automorphism. 

Next let us clarify the different notions of effectivity:

\begin{lemma}
Given a symmetric Lie algebra $(\mf g,s)$ the following statements hold.
\begin{enumerate}[(i)]
\item \label{it:effective} $(\mf g,s)$ is effective if and only if the adjoint representation of $\mf k$ on $\mf g$ is faithful.
\item \label{it:strongly-effective} $(\mf g,s)$ is strongly effective if and only if the adjoint representation of $\mf k$ on $\mf p$ is faithful.
\end{enumerate}
Furthermore if $(\mf g,s)$ is strongly effective then it is effective.
\end{lemma}

\begin{proof}
Let $\mf z$ denote the center of $\mf g$. 
\eqref{it:effective}: This is clear since $\mf k\cap\mf z$ is exactly the kernel of the adjoint representation of $\mf k$ on $\mf g$.
\eqref{it:strongly-effective}: If $\mf k$ contains a non-trivial ideal of $\mf g$, then this ideal lies in the kernel of the adjoint representation of $\mf k$ on $\mf p$, as can be seen from~\eqref{eq:cartan-commutators}. Conversely, the kernel of the adjoint representation of $\mf k$ on $\mf p$ is an ideal of $\mf g$.
\end{proof}

\begin{remark} \label{rem:effective-general}
There are slightly different definitions used in the literature. For instance, Kobayashi \& Nomizu~\cite[Ch.~XI Sec.~2]{KobNom96} call effective what we call strongly effective. Our terminology is consistent with Helgason~\cite{Helgason78} who, however, does not give a name to what we call strongly effective.
\end{remark}

Next let us have a look at orthogonality. This condition guarantees the existence of an adapted inner product on $\mf g$.

\begin{lemma} \label{lemma:inner-prod-orth}
Let $(\mf g,s)$ be an orthogonal, symmetric Lie algebra. Then there exists an $s$-invariant inner product for which all of $\ad_{\mf k}$ is skew-symmetric.
\end{lemma}

\begin{proof}
Start with any inner product $(\cdot,\cdot)$ on $\mf g$. Let $K=\Int_{\mf k}(\mf g)$, which is compact by assumption and hence admits a Haar measure denoted $d\mu$. Now define the average
\begin{equation} \label{eq:inner_prod_s_inv}
\braket{x,y} = \sum_{m=0}^1\int_K (s^m(U\cdot x),s^m(U\cdot y)) \,d\mu(U)\,.
\end{equation}
Since $K$ preserves the decomposition $\mf g=\mf k\oplus\mf p$ due to the commutator relations~\eqref{eq:cartan-commutators}, $s$ commutes with the action of $U\in K$. Thus~\eqref{eq:inner_prod_s_inv} is invariant under $s$ and under $K$. This shows that $K$ is a subgroup of the orthogonal group of $(\mf g, \braket{\cdot,\cdot})$, and hence its Lie algebra $\ad_{\mf k}$ is a subalgebra of the orthogonal Lie algebra, so each $\ad_k$ for $k\in\mf k$ is skew-symmetric with respect to $\braket{\cdot,\cdot}$.
\end{proof}

We have the following converse result in the semisimple case.
First recall that a Lie algebra $\mf g$ over a field of characteristic $0$ is semisimple if and only if the Killing form $B(x,y)=\operatorname{tr}(\ad_x\circ\ad_y)$ is non-degenerate, or, if and only if $\mf g$ contains no proper non-trivial Abelian ideals.

\begin{lemma} \label{lemma:orth-semisimple-converse}
Let $(\mf g,s)$ be a semisimple, symmetric Lie algebra and let $\braket{\cdot,\cdot}$ be an inner product on $\mf g$ for which all of $\ad_{\mf k}$ is skew-symmetric. Then $(\mf g,s)$ is orthogonal.
\end{lemma}

\begin{proof}
Let $G$ be ``the'' simply connected Lie group with Lie algebra $\mf g$, see~\cite[Thm.~20.21]{Lee13}, and let $K$ be the connected Lie subgroup with Lie algebra $\mf k$. By~\cite[Thm.~20.19]{Lee13} there is a unique Lie group automorphism $\sigma:G\to G$ with $D\sigma(e)=s$. As $\mf k$ is the fixed point set of $s$, $K$ equals the identity component of the fixed point set of $\sigma$. In particular, $K$ is closed in $G$, and hence an embedded Lie subgroup. Since $\mf g$ is semisimple, $\ad:\mf g\to\mf{gl}(\mf g)$ is faithful, and hence the adjoint representation $\Ad:G\to\Int(\mf g)$ is a covering homomorphism by~\cite[Thm.~21.31]{Lee13}. The image of $K$ under $\Ad$ is exactly $\Int_{\mf k}(\mf g)$. Now let $U\in K$ be arbitrary. There is a neighborhood $W$ of $\Ad_U$ in $\Int(\mf g)$ diffeomorphic to some neighborhood $W'$ of $U$ in $G$. Since $K$ is embedded, and by shrinking the neighborhoods, we can assume that $W'$ is a slice chart for $K$, and hence the same is true for $W$ and $\Int_{\mf k}(\mf g)$. 
By~\cite[Thm.~5.8]{Lee13} this shows that $\Int_{\mf k}(\mf g)$ is embedded in $\Int(\mf g)$ and hence closed. 
By~\cite[Ch.~II Coro.~6.5]{Helgason78} $\Int(\mf g)$ is closed in $\GL(\mf g)$. By assumption, $\Int(\mf g)\subseteq\SO(\mf g)$ with respect to the given inner product. Since $\SO(\mf g)$ is closed in $\GL(\mf g)$, $\Int(\mf g)$ is closed in $\SO(\mf g)$, and hence compact. This concludes the proof.
\end{proof}

One calls a Lie algebra $\mf g$ compact if its Killing form is negative definite\footnote{This definition excludes for instance Abelian Lie algebras which have trivial Killing form.}. 
These notions lead to different classes of symmetric Lie algebras:

\begin{Definition}
Let $(\mf g, s)$ be a symmetric Lie algebra.
\begin{enumerate}[(i)]
\item If $[\mf p,\mf p]=0$, then $(\mf g, s)$ is called \emph{of Euclidean type}.
\item If $\mf g$ is semisimple, it is called \emph{of compact type} if $\mf g$ is compact. Otherwise it is called of \emph{of non-compact type}.
\item If the adjoint representation of $\mf k$ on $\mf p$ is irreducible, then $(\mf g, s)$ is called \emph{irreducible}.
\end{enumerate}
\end{Definition}

There exists an important duality between symmetric Lie algebras of compact and of non-compact type. For a symmetric Lie algebra $(\mf g,s)$ with Cartan-like decomposition $\mf g=\mf k\oplus\mf p$ we can define its dual by $\mf g=\mf k\oplus \iu\mf p$ as a subspace of the complexification of $\mf g$. We denote the dual by $(\mf g^*,s^*)$. For the following result see~\cite[Ch.~V, Prop.~2.1]{Helgason78}.

\begin{lemma} \label{lemma:cpt-non-cpt-duality}
If $(\mf g,s)$ is a semisimple, orthogonal, symmetric Lie algebra of compact type, then $(\mf g^*,s^*)$ is of non-compact type, and vice-versa.
\end{lemma}

With these definitions in place one can understand the structure of effective, orthogonal, symmetric Lie algebras. Most importantly one gets a decomposition of the Lie algebra into a Euclidean, a compact, and a non-compact subalgebra:

\begin{lemma}
\label{lemma:decomp-of-eos-lie-alg}
Let $(\mf g, s)$ be an effective, orthogonal, symmetric Lie algebra and let $\braket{\cdot,\cdot}$ be an $s$-invariant inner product on $\mf g$ such that all $\ad_k$ for $k\in\mf k$ are skew-symmetric. Then there exists a decomposition of $\mf g$ into $s$-invariant ideals
\begin{align*}
\mf g=\mf g_0\oplus\mf g_-\oplus\mf g_+,
\end{align*}
which are orthogonal with respect to the Killing form. Denoting the restrictions of $s$ by $s_0$,$s_-$, and $s_+$, the pairs $(\mf g_0, s_0)$, $(\mf g_-, s_-)$, and $(\mf g_+, s_+)$ are effective, orthogonal, symmetric Lie algebras of Euclidean, compact, and non-compact type, respectively.
Moreover there exists a vector space decomposition
\begin{align*}
\mf p_-\oplus\mf p_+ = \bigoplus_{i=1}^{m} \mf p_i,
\end{align*}
such that $\braket{\cdot,\cdot}=\lambda_i B$ on $\mf p_i$ with $\lambda_i\neq0$. It holds that $\mf p_i\subseteq\mf p_-$ if $\lambda_i<0$, and $\mf p_i\subseteq\mf p_+$ if $\lambda_i>0$. Moreover the $\mf p_i$ are orthogonal to each other with respect to the inner product and the Killing form, and $[\mf p_i,\mf p_j]=0$.
\end{lemma}

For a proof we refer to~\cite[Ch.~V, Thm.~1.1]{Helgason78} and its proof.

\begin{remark}
Using strong effectivity, one even obtains a decomposition into irreducible ideals, see~\cite[Ch.~VIII, Prop.~5.2]{Helgason78}.
\end{remark}

\begin{corollary}
\label{coro:semisimple-p0}
If in addition to the assumptions from Lemma~\ref{lemma:decomp-of-eos-lie-alg} $\mf g$ is semisimple, then $\mf p_0 = \{0\}$ where $\mf g_0 = \mf k_0\oplus\mf p_0$.
\end{corollary}

\begin{proof}
Since $[\mf p_0,\mf p_0]=0$, it holds that $\mf p_0$ is an Abelian subalgebra of $\mf g$. Furthermore we have that $[\mf k_0,\mf p_0]\subseteq\mf p_0$ and hence $\mf p_0$ is an ideal in $\mf g$. By semisimplicity of $\mf g$ we get that $\mf p_0=\{0\}$.
\end{proof}

Combining some of the results above we find the following structure for semisimple, orthogonal, symmetric Lie algebras:

\begin{corollary} \label{coro:decomp-of-semisimple}
Let $(\mf g,s)$ be a semisimple, orthogonal, symmetric Lie algebra. Then we have the following decomposition into ideals
\begin{align*}
\mf g = \mf k_0 \oplus \mf g_- \oplus \mf g_+,
\end{align*}
where $\mf g_-$ and $\mf g_+$ are of compact and of non-compact type respectively.
\end{corollary}

\begin{proof}
This follows immediately from Remark~\ref{rmk:semisimple-effective}, Lemma~\ref{lemma:decomp-of-eos-lie-alg}, and Corollary~\ref{coro:semisimple-p0}.
\end{proof}

Before using this decomposition of effective, orthogonal, symmetric Lie algebras to connect the $s$-invariant inner product to the Killing form in the semisimple case, we need the following simple lemma.

\begin{lemma} \label{lemma:B-neg-def-on-k}
Let $(\mf g,s)$ be an effective, orthogonal, symmetric Lie algebra and let $B$ denote the Killing form on $\mf g$. Then $B$ is negative definite on $\mf k$.
\end{lemma}

\begin{proof}
By Lemma~\ref{lemma:inner-prod-orth} there exists an inner product on $\mf g$ such that all $\ad_k$ for $k\in\mf k$ are skew-symmetric.
Note that the trace and hence the Killing form are independent of the inner product on $\mf g$.
Hence $B(k,k)=-\tr(\ad_k^\top\circ\ad_k)\leq0$ with equality only if $\ad_k=0$, that is $k\in\mf z$, and hence by effectivity $k=0$.
\end{proof}

Now for the result in question:

\begin{lemma}
\label{lemma:nice-inner-prod}
Let a semisimple, orthogonal, symmetric Lie algebra $(\mf g,s)$ be given. Define $\braket{\cdot,\cdot}:\mf g\times\mf g\to\mathbb R$ as follows: Given any $x,y\in\mf g$ there exist by the previous results unique $x_k,y_k\in\mf k$, $x_+,y_+\in\mf p_+$, and $x_-,y_-\in\mf p_-$ such that
$x=x_k+x_++x_-$, $y=y_k+y_++y_-$. Then
$$
\braket{x,y}:=-B(x_k,y_k)+B(x_+,y_+)-B(x_-,y_-)
$$
is an $s$-invariant inner product on $\mf g$ such that $\ad_k$ is skew-symmetric for all $k\in\mf k$.
\end{lemma}

\begin{proof}
First it is easy to see that the definition yields a bilinear, symmetric form on $\mf g$. Then consider the group $K=\Int_{\mf k}(\mf g)$; since $s$ and all $U\in K$ define automorphisms of $\mf g$ they leave the Killing form $B$ invariant, and they clearly respect the decomposition $\mf g=\mf k\oplus\mf p_-\oplus\mf p_+$. Hence they leave $\braket{\cdot,\cdot}$ invariant and all $\ad_k$ for $k\in\mf k$ are skew-symmetric.
It remains to show that $\braket{\cdot,\cdot}$ is positive definite. For this consider another inner product $(\cdot,\cdot)$ as guaranteed by Lemma~\ref{lemma:inner-prod-orth}. Now Lemma~\ref{lemma:decomp-of-eos-lie-alg} and Corollary~\ref{coro:semisimple-p0} guarantee the existence of a decomposition $\mf p=\mf p_+\oplus\mf p_-=\bigoplus_{i=1}^{m} \mf p_i$, such that $(\cdot,\cdot)=\lambda_i B$ on $\mf p_i$ with $\lambda_i\neq0$; moreover $\mf p_i\subseteq\mf p_-$ if  $\lambda_i<0$, and $\mf p_i\subseteq\mf p_+$ if $\lambda_i>0$. Let us denote the corresponding index sets by $I_-$ and $I_+$, respectively. Now given arbitrary $k\in\mf k$ and $x_i\in\mf p_i$, using that the $\mf p_i$ are orthogonal to each other with respect to the Killing form we compute
\begin{align*}
\braket{k+\sum_{i=1}^m x_i,k+\sum_{j=1}^m x_j} 
&= -B(k,k)  +B\Big(\sum_{i\in I_+} x_i, \sum_{j\in I_+} x_j\Big)-B\Big(\sum_{i\in I_-} x_i, \sum_{j\in I_-} x_j\Big)  \\
&= -B(k,k)  +\sum_{i\in I_+} B(x_i, x_i)-\sum_{i\in I_-} B(x_i, x_i)  \\
&= -B(k,k) + \sum_{i=1}^m \frac{1}{|\lambda_i|}(x_i,x_i),
\end{align*}
so Lemma~\ref{lemma:B-neg-def-on-k} shows that $\braket{\cdot,\cdot}$ is positive definite, as claimed.
\end{proof}

Unless otherwise noted, we will always assume that we are using this inner product. The following result is the reason why we will mostly focus on semisimple Lie algebras, and it will be used several times in the main text.


\begin{lemma} \label{lemma:inner-prod-centralizer}
Let $(\mf g, s)$ be a semisimple, orthogonal, symmetric Lie algebra and let $x,y\in\mf p$, then the following are equivalent:
\begin{align*}
[x,y]=0
\iff 
\braket{[k,x],y}=0 \quad \text{ for all } k\in\mf k,
\end{align*}
where the inner product on $\mf g$ is as in Lemma~\ref{lemma:inner-prod-orth}. Differently put, for every $x\in\mf p$ we have an orthogonal vector space decomposition
$$
\mf p=(\mf p\cap\ker\ad_x)\oplus(\mf p\cap\img\ad_x).
$$
\end{lemma}

\begin{proof}
Let $k\in\mf k$ and $x,y\in\mf p$. Using Lemma~\ref{lemma:decomp-of-eos-lie-alg} we can write $\mf p = \oplus_{i=1}^m \mf p_i$. Note that due to semisimplicity and Corollary~\ref{coro:semisimple-p0} it holds that $\mf p_0=\{0\}$ and hence is omitted from the decomposition. Furthermore we know that the $\mf p_i$ are mutually orthogonal and invariant under $\ad_k$ for all $k\in\mf k$. Writing $x=\sum_{i=1}^m x_i$ and $y=\sum_{i=1}^m y_i$ we compute
\begin{align*}
\braket{[k,x],y}
=
\sum_{i=1}^m \braket{[k,x_i],y_i}
=
\sum_{i=1}^m \lambda_i B([k,x_i],y_i)
=
\sum_{i=1}^m \lambda_i B(k,[x_i,y_i])
\end{align*}
If $[x,y]=0$ then $[x_i,y_i]=0$ for $i=1,\ldots,m$ and hence $\braket{[k,x],y}=0$ for all $k\in \mf k$. Conversely we can choose $k_j=[x_j,y_j]/\lambda_j\in\mf k$ for $j=1,\ldots,m$. Then 
\begin{align*}
0
=\braket{[k_j,x],y}
=\sum_{i=1}^m  B([x_j,y_j],[x_i,y_i])
= B([x_j,y_j],[x_j,y_j])
\end{align*}
implies that $[x_j,y_j]=0$ by Lemma~\ref{lemma:B-neg-def-on-k}. Hence $[x,y]=0$. 
\end{proof}

The following example shows that the presence of a Euclidean part causes problems in the preceding lemma.

\begin{example}
This is~\cite[Example~(c) p.~230]{Helgason78}. Let $\mf p\neq\{0\}$ be any real vector space and let $\mf k$ be the Lie algebra of any compact subgroup of $\GL(\mf p)$. For $x,y\in\mf p$ and $k\in\mf k$ we set $[x, y]=0$ and $[k,p]=-[p,k]=k\cdot p$. Then $\mf k\oplus\mf p$ defines an effective, orthogonal, symmetric Lie algebra of Euclidean type. This violates Lemma~\ref{lemma:inner-prod-centralizer} since there is always some $x\in\mf p$ and $k\in\mf k$ such that $[k,x]\neq0$.
\end{example}

It turns out that maximal Abelian subspaces $\mf a \subseteq \mf p$ play an important role in understanding the structure of symmetric Lie algebras. A useful notation for $x\in\mf p$ is
$$\mf p_x = \{ y\in\mf p : [x,y]=0 \},$$ 
which denotes the centralizer (or commutant) of $x$ in $\mf p$. Clearly for $x\in\mf a$ it holds that $\mf a\subseteq\mf p_x$. If equality holds, then we call $x$ \emph{regular}. Hence a regular element is contained in a unique maximal Abelian subspace.

\begin{lemma}
\label{lemma:regular-element}
Let $(\mf g, s)$ be an effective, orthogonal, symmetric Lie algebra and let $\mf a \subseteq \mf p$ be a maximal Abelian subspace. Then there exists $x\in\mf a$ such that $\mf p_x = \mf a$.
\end{lemma}

\begin{proof}
First we show that we can reduce the problem to $(\mf g, s)$ being of compact type. Let $\mf g = \mf g_0\oplus\mf g_-\oplus\mf g_+$ be the decomposition of Result~\ref{lemma:decomp-of-eos-lie-alg}. Then we can write $\mf a = \mf a_0\oplus\mf a_-\oplus\mf a_+$ where $\mf a_0=\mf p_0$ and $\mf a_\pm$ is a maximal Abelian subspace of $\mf p_\pm$. If $x,y\in\mf p$ then $[x,y]=[x_0,y_0]+[x_-,y_-]+[x_+,y_+]$ and hence $x$ and $y$ commute if and only if $[x_i,y_i]=0$ for all $i\in\{0,-,+\}$. 
Assume that we have found $x_i\in\mf p_i$ satisfying $(\mf p_i)_{x_i} = \mf a_i$ for all $i$. Then setting $x=\sum_ix_i$ we get that $\mf p_x=\mf a$.
Hence it suffices to consider the Euclidean, compact and non-compact types separately. However the Euclidean case is trivial, since $\mf a_0=\mf p_0$ so any element of $\mf p_0$ does the job. 
Using the duality described in Lemma~\ref{lemma:cpt-non-cpt-duality} it suffices to consider the compact case, since $x,y\in\mf p$ commute if and only if $\iu x$ and $\iu y$ commute in the dual. Hence for the remainder of the proof we assume that $(\mf g, s)$ is of compact type.

Let $G$ be a compact group with Lie algebra $\mf g$ and consider the torus $A=\exp(\mf a)$. Let $x\in\mf a$ be the generator of a dense winding of $A$. If $y\in\mf p$ satisfies $[x,y]=0$, then $\exp(t y)$ commutes with all elements of $A$ for all $t\in\R$. This implies that $[y,\mf a]=0$, and since $\mf a$ is maximal Abelian, $y\in\mf a$. Thus $\mf p_x = \mf a$, as desired.
\end{proof} 

We have previously used the Lie group action of $\Int_{\mf k}(\mf g)$ on $\mf g$. Let us make this a bit more precise.\footnote{For more details we refer to~\cite[Ch.~II~\S5]{Helgason78}.}
By definition $\Int(\mf g)$ is the connected Lie subgroup of $\GL(\mf g)$ with Lie algebra $\ad_{\mf g}$. In fact this is the group of inner automorphisms of $\mf g$. Furthermore $\Int_{\mf k}(\mf g)$ is the connected Lie subgroup of $\Int(\mf g)$ with Lie algebra $\ad_{\mf k}$. We slightly generalize this idea by defining associated pairs following~\cite[Def.~p.~213]{Helgason78}.


\begin{Definition}
Consider a symmetric Lie algebra $\mf g=\mf k\oplus\mf p$. Let $G$ be a connected Lie group with a Lie algebra isomorphism\footnote{Here $\mathrm{Lie}(G)$ denotes the Lie algebra of $G$.} $\phi:\mf g\to\mathrm{Lie}(G)$ and let $K\subseteq G$ be a Lie subgroup with $\phi(\mf k)=\mathrm{Lie}(K)$. Then we say that $(G,K)$ is a \emph{pair associated to $\mf g$}.
\end{Definition}

\begin{lemma}
\label{lemma:existence-of-pair}
For any semisimple, symmetric Lie algebra $(\mf g,s)$ there exists an associated pair $(G,K)$ with $G$ and $K$ connected.
In fact we can choose $G=\Int(\mf g)$ and $K=\Int_{\mf k}(\mf g)$, in which case the Lie algebra isomorphism is $\phi=\ad$.
\end{lemma}

\begin{proof}
This follows from the definitions, and the fact that $\mf g$ is semisimple, and thus $\ad:\mf g\to\mf{gl}(\mf g)$ is a Lie algebra isomorphism.
\end{proof}

We call the pair described in~\ref{lemma:existence-of-pair} is called the \emph{canonical pair associated to $(\mf g,s)$}. First we give a simple lemma about Lie subgroups:

\begin{lemma} \label{lemma:restriction-Lie-hom}
Let $\phi:G\to H$ be a Lie group homomorphism, and let $\tilde G\subseteq G$ and $\tilde H\subseteq H$ be Lie subgroups. Assume that $\phi(\tilde G)\subseteq\tilde H$. Then the restriction $\phi\circ\iota:\tilde G\to\tilde H$ is also a Lie group homomorphism, where $\iota:\tilde G\to G$ is the inclusion.
\end{lemma}

\begin{proof}
Since $\iota$ is a smooth immersion, $\phi\circ\iota$ is smooth as a map from $\tilde G$ to $H$. The fact that it is still smooth as a map from $\tilde G$ to $\tilde H$ holds because all Lie subgroups are weakly embedded, see~\cite[Thm.~19.25]{Lee13}.
\end{proof}

\begin{lemma} \label{lemma:pairs-factor-canonical}
Consider any symmetric Lie algebra $(\mf g,s)$ and any associated pair $(G,K)$ with $G$ and $K$ connected and corresponding Lie algebra isomorphism $\phi:\mf g\to\mathrm{Lie}(G)$. Then the maps $\psi:G\to\Int(\mf g),g\mapsto\phi^{-1}\circ\Ad(g)\circ\phi$ and $\psi|_K:K\to\Int_{\mf k}(\mf g)$ are surjective Lie group homomorphisms. 
\end{lemma}

\begin{proof}
The maps are well-defined and surjective due to the connectedness of $G$ and $K$.
By~\cite[p.~127]{Helgason78} the maps $G\to\Int(\Lie(G))$ and $\Int(\Lie(G))\to\Int(\mf g)$ are Lie group homomorphisms, and hence also $\psi$. By~\ref{lemma:restriction-Lie-hom} then also $\psi|_K:K\to\Int_{\mf k}(\mf g)$ is a Lie group homomorphism.
\end{proof}

In the following we suppress the Lie algebra isomorphism $\phi$ from the notation. Then $\psi$ becomes $\Ad$, the adjoint representation. Note that for the canonical associated pair, the adjoint representation is essentially the identity map.

\begin{lemma} \label{lemma:semisimple-factor-canonical}
Let $(\mf g,s)$ be a semisimple, symmetric Lie algebra and let $(G,K)$ be an associated pair with $G$ and $K$ connected. Then the adjoint representation of $K$ on $\mf g$, written $\Ad:K\to\Int_{\mf k}(\mf g)$ is a covering homomorphism.
\end{lemma}

\begin{proof}
By Lemma~\ref{lemma:pairs-factor-canonical} $\Ad:K\to\Int_{\mf k}(\mf g)$ is a surjective Lie group homomorphism, and by semisimplicity it induces a Lie algebra isomorphism. Hence it is a covering map.
\end{proof}

In the next lemma we collect some equivariance properties of the action of $K$, but first we introduce some notation. By Lemma~\ref{lemma:inner-prod-centralizer} there is, for each $x\in\mf p$, an orthogonal decomposition of $\mf p$ into kernel and image of $\ad_x$. By $\Pi_x:\mf p\to\mf p$ we denote the orthogonal projection onto $\mf p_x:=\mf p\cap\ker\ad_x$, and by $\Pi_x^\perp=\id-\Pi_x$ the complementary projection. Moreover we can define an inverse of $\ad_x$ by restricting the domain and codomain. Then we get a well-defined inverse $\ad_x^{-1}:\mf p\cap\img\ad_x\to\mf k_x^\perp$.

\begin{lemma}[Equivariance properties] \label{lemma:equivariance}
Let $(\mf g,s)$ be a semisimple, orthogonal, symmetric Lie algebra with associated pair $(G,K)$ with $K$ connected. For $U\in K$ and $x\in\mf p$, it holds that:
\begin{enumerate}[(i)]
\item \label{it:equiv-ad} $\Ad_U\circ\ad_x=\ad_{\Ad_U(x)}\circ\Ad_U$ on $\mf p$;
\item \label{it:equiv-adinv} $\Ad_U\circ\ad_x^{-1}=\ad_{\Ad_U(x)}^{-1}\circ\Ad_U$ on $\mf p\cap\img\ad_x$;
\item \label{it:equiv-proj} $\Ad_U\circ\,\Pi_x = \Pi_{\Ad_U(x)}\circ\Ad_U$ on $\mf p$;
\item \label{it:equiv-projperp} $\Ad_U\circ\,\Pi_x^\perp = \Pi_{\Ad_U(x)}^\perp\circ\Ad_U$ on $\mf p$.
\end{enumerate}
\end{lemma}

\begin{proof}
\eqref{it:equiv-ad}: Clear since $\Ad_U([g,h])=[\Ad_U(g),\Ad_U(h)]$ for all $g,h\in\mf g$. Moreover this implies $\Ad_U(\mf p_x)=\mf p_{\Ad_U(x)}$ and using orthogonality $\Ad_U(\mf k_x^\perp)=\mf k_{\Ad_U(x)}^\perp$.
\eqref{it:equiv-adinv}: Consider $y=\ad_x(k)$ for some $k\in\mf k_x^\perp$, then clearly $\Ad_U(\ad_x^{-1}(y))=\Ad_U(k)$. By the previous point $\Ad_U(y)=\ad_{\Ad_U(x)}(\Ad_U(k))$ and $\Ad_U(k)\in\mf k_{\Ad_U(x)}^\perp$ and hence $\ad_{\Ad_U(x)}^{-1}(\Ad_U(y))=\Ad_U(k)$. 
\eqref{it:equiv-proj}: Let $z\in\mf p$ and consider the orthogonal decomposition $z=[k,x]+y$ where $k\in\mf k_x^\perp$ and $y\in\mf p_x$. Then, for any $U\in K$, it holds that $\Ad_U\circ\,\Pi_x(z)=\Ad_U(y)$. On the other hand, using the first point again, $\Pi_{\Ad_U(x)}\circ\Ad_U(z)=\Pi_{\Ad_U(x)}([\Ad_U(k),\Ad_U(x)]+\Ad_U(y))=\Ad_U(y)$.
\eqref{it:equiv-projperp}: Analogous to~\eqref{it:equiv-proj}.
\end{proof}

\begin{lemma}
\label{lemma:orbit-tangent-centralizer}
Let $\mf g=\mf k\oplus\mf p$ be a semisimple, orthogonal, symmetric Lie algebra with associated pair $(G,K)$. Let $x\in\mf p$, and let $O=\Ad_K(x)$ denote the $K$-orbit containing $x$. Then, using the canonical identification $T_x\mf p\cong\mf p$ it holds that $\ad_{\mf k}(x)$ is the tangent space $T_xO$ and $\mf p_x=\mf p\cap\ker\ad_x$ is its orthogonal complement.
\end{lemma}

\begin{proof}
This follows immediately from Lemma~\ref{lemma:inner-prod-centralizer}.
\end{proof}

\begin{lemma}
\label{lemma:max-abelian-conj}
Let $\mf g=\mf k\oplus\mf p$ be a semisimple, orthogonal, symmetric Lie algebra with associated pair $(G,K)$. Let $\mf a\subseteq\mf p$ be some maximal Abelian subspace. Then every $K$-orbit in $\mf p$ intersects $\mf a$ and all maximal Abelian subspaces of $\mf p$ are conjugate by some element in $K$. 
\end{lemma}

\begin{proof}
Without loss of generality we assume that we are using the canonical associated pair $(G,K)=(\Int(\mf g),\Int_{\mf k}(\mf g))$. Then Lemma~\ref{lemma:pairs-factor-canonical} guarantees the existence of the desired group element in any pair.

Let $x\in\mf a$ be regular, which exists by Lemma~\ref{lemma:regular-element}, and let $y\in\mf p$. Consider the smooth function
\begin{align*}
f:K\to\R,\quad U\mapsto B(U\cdot y,x).
\end{align*}
Let $U\in K$ be a critical point of $f$, which exists since $K$ is compact. Then for any $k\in\mf k$ it holds that
\begin{align*}
0 = \frac{d}{dt}\Big|_{t=0}(f(e^{tk}U)) = B([k,U\cdot y],x) = B(k,[U\cdot y,x]).
\end{align*}
Since this holds for all $k\in\mf k$, Lemma~\ref{lemma:B-neg-def-on-k} shows that $[U\cdot y,x]=0$ and since $x$ is regular, $U\cdot y\in\mf a$. This shows that every $K$-orbit in $\mf p$ intersects $\mf a$. 

For any $U\in K$ it holds that $U\cdot\mf a$ is maximal Abelian with regular element $U\cdot x$, by Lemma~\ref{lemma:equivariance}~\eqref{it:equiv-ad}. If $\mf a'$ is any other maximal Abelian subspace, we can choose $U$ such that $U\cdot x\in\mf a'$, and thus $\mf a' = \mf p_{U\cdot x} = U\cdot \mf p_x = U\cdot\mf a$. Hence all maximally Abelian subspaces are conjugate by some element in $K$.
\end{proof}

\begin{Definition}
Let $\mf g=\mf k\oplus\mf p$ be a semisimple, orthogonal, symmetric Lie algebra with associated pair $(G,K)$, and let $\mf a\subset\mf p$ be a maximal Abelian subspace. Then we define the \emph{normalizer of $\mf a$ in $K$}
\begin{align*}
N_K(\mf a)=\{U\in K : \Ad_U(\mf a)=\mf a\},
\end{align*}
and the \emph{centralizer of $\mf a$ in $K$}
\begin{align*}
Z_K(\mf a)=\{U\in K : \Ad_U(x)=x \text{ for all } x\in\mf a\}.
\end{align*}
Since $Z_K(\mf a)$ is clearly normal in $N_K(\mf a)$ we can define the \emph{Weyl group} $W=W_{\mf a}=N_K(\mf a)/ Z_K(\mf a)$ which acts on $\mf a$.
\end{Definition}

Note that since all maximal Abelian subspaces are conjugate by Lemma~\ref{lemma:max-abelian-conj}, different choices of $\mf a$ lead to isomorphic Weyl groups. We will see later that the Weyl group also does not depend on the choice of associated pair $(G,K)$.

\medskip
\subsection{Root Space Decomposition}

In order to understand the Weyl group action on $\mf a$ defined above we need to understand the root space decomposition of a semisimple, orthogonal, symmetric Lie algebra. This is the goal of the present section. By $\mf g_\C$ we denote the complexification of $\mf g$. 

\begin{Definition}
Let $(\mf g, s)$ be a symmetric Lie algebra and $\mf a\subseteq\mf p$ be given. For any linear functional $\alpha\in\operatorname{Hom}_\R(\mf a,\C)$ we define
\begin{align*}
\mf g_\C^\alpha = \{ x\in\mf g_\C : \ad_y(x) = \alpha(y) x \text{ for all } y\in\mf a\}.
\end{align*}
If $\mf g_\C^\alpha$ is non-empty we call $\alpha$ a \emph{root} and $\mf g_\C^\alpha$ the corresponding \emph{root space}. The non-zero elements of a root space are called \emph{root vectors}. We denote by $\Delta$ the set of all non-zero roots and by $\Delta_0$ the set of all roots including zero.\footnote{Some authors don't consider $0$ a root at all.}
\end{Definition}

\begin{lemma}
\label{lemma:semisimple-normal-adx}
Let $(\mf g,s)$ be a semisimple, orthogonal, symmetric Lie algebra. Then
\begin{align*}
\mf g_\C= \bigoplus_{\alpha\in\Delta_0}\mf g_\C^\alpha\,,
\end{align*}
and every non-zero root is either imaginary or real and supported on the compact or non-compact part of $\mf a$ respectively. We write $\Delta_-$ for the compact roots and $\Delta_+$ for the non-compact ones.
\end{lemma}

\begin{proof}
We use the inner product of Lemma~\ref{lemma:nice-inner-prod}. By Corollary~\ref{coro:decomp-of-semisimple} we have that $\mf g = \mf k_0 \oplus \mf g_- \oplus \mf g_+$, and each $\ad_x$ for $x=x_-+x_+\in\mf p$ preserves this decomposition. It acts trivially on $\mf k_0$ and as $\ad_{x_\pm}$ on $\mf g_\pm$. Furthermore $\ad_{x_-}$ is
skew-symmetric and $\ad_{x_+}$ is symmetric\footnote{These (skew-)symmetric operators on $\mf g$ are of course still real and (skew-)symmetric on $\mf g_\C$, and hence (skew-)Hermitian. In particular the eigenvalues are imaginary in the compact case and real in the non-compact case.} since
\begin{align*}
\braket{\ad_x(y),k}
=-B([x,y],k)=B(y,[x,k])=\pm \braket{y,\ad_x(k)},
\end{align*}
for $x,y\in\mf p_\pm$ and $k\in\mf k_\pm$.
If $x,y\in\mf p$ commute, then also $\ad_x$ and $\ad_y$ since, by the Jacobi identity,
\begin{align*}
\ad_x\circ\ad_y(z) 
= [x,[y,z]]
= [[x,y],z]+[y,[x,z]] 
= \ad_y\circ\ad_x(z).
\end{align*}
Hence all $\ad_x$ for $x\in\mf a$ can be simultaneously (unitarily) diagonalized and we obtain a complete root space decomposition of $\mf g_\C$. Now every root vector of $\mf g_+$ relative to $\mf a_+$ is a root vector of $\mf g$ and similarly for $\mf g_-$. Conversely, let $\alpha$ be a non-zero root and $x\in\mf g_\C^\alpha$. Then for $y\in\mf a$ we have that
\begin{align*}
\ad_y(x)
=[y_-,x_-]+[y_+,x_+]
=\alpha(y) x.
\end{align*}
If $x_-\neq0$ and $x_+\neq0$ then $[y_-,x_-]=\alpha(y) x_-$ and $[y_+,x_+]=\alpha(y) x_+$ and so $\alpha(y)$ would have to be purely imaginary and real, leading to a contradiction. This shows that either $\alpha$ only takes imaginary values, and its root space lies in the compact part, or it only takes real values, and its root space lies in the non-compact part
\end{proof}

\begin{corollary}
\label{coro:closed-under-transpose}
Let $(\mf g,s)$ be a semisimple, orthogonal, symmetric Lie algebra. Then, in some basis, the Lie algebra $\ad_{\mf g}=\{\ad_z:z\in\mf g\}$ is a Lie algebra of real matrices closed under transposition on the vector space $\mf g$ with inner product as in Lemma~\ref{lemma:nice-inner-prod}.
\end{corollary}

\begin{proof}
By construction, all $\ad_k$ for $k\in\mf k$ are skew-symmetric in the given inner product. The proof of Lemma~\ref{lemma:semisimple-normal-adx} shows that $\ad_x$ is skew-symmetric or symmetric for $x\in\mf p_-$ or $x\in\mf p_+$ respectively. Hence for any $z\in\mf g$, the transposed of $\ad_z$ with respect to the given inner product is also in $\ad_{\mf g}$. Now any orthonormal basis on $\mf g$ will do.
\end{proof}

We now have two involutions on $\mf g_\C$, namely (the complexification of) $s$ and complex conjugation. The next two lemmas show how they act on the root spaces.

\begin{lemma} \label{lemma:sym-roots}
Let $(\mf g,s)$ be a semisimple, orthogonal, symmetric Lie algebra. Then $s(\mf g^\alpha_\C)=\mf g^{-\alpha}_\C$.
\end{lemma}

\begin{proof}
Let $x\in\mf g^\alpha_\C$ then for all $y\in\mf a$ it holds that $[y,s(x)]=s([s(y),x])=s([-y,x])=-s([y,x])=-\alpha(y)s(x)$. Hence $s(x)\in\mf g^{-\alpha}_\C$. This shows that $s(\mf g^\alpha_\C)\subseteq\mf g^{-\alpha}_\C$. But then $\mf g^{-\alpha}_\C=s(s(\mf g^{-\alpha}_\C)) \subseteq s(\mf g^{\alpha}_\C)$, as desired.
\end{proof}

This shows that $\mf g^\alpha = (\mf g^\alpha_\C \oplus \mf g^{-\alpha}_\C) \cap \mf g$ is invariant under $s$ and hence decomposes as $\mf g^\alpha=\mf k^\alpha\oplus\mf p^\alpha$ where $\mf k^\alpha=\mf g^\alpha\cap \mf k$ and $\mf p^\alpha=\mf g^\alpha\cap \mf p$.

\begin{lemma} \label{lemma:conj-roots}
Let $(\mf g,s)$ be a semisimple, orthogonal, symmetric Lie algebra and let $\alpha\in\Delta_\pm$ be a non-zero root. Then $\overline{\mf g^\alpha_\C}=\mf g^{\pm\alpha}_\C$.
\end{lemma}

\begin{proof}
Let $x\in\mf g^\alpha_\C$ then for all $y\in\mf a$ it holds that $[y,\overline x]=\overline{[\overline y,x]}=\overline{[y,x]}=\overline{\alpha(y)x}=\pm\alpha(y) \overline x$.
\end{proof}

Using Lemmas~\ref{lemma:sym-roots} and~\ref{lemma:conj-roots} we can find for each root a corresponding root vector which is composed of an element of $\mf k^\alpha$ and an element of $\mf p^\alpha$.

\begin{lemma} \label{lemma:related-vectors}
Let $\alpha\in\Delta_\pm$ be a non-zero root. Then there exist $x\in\mf k^\alpha$ and $y\in\mf p^\alpha$ such that $x+\sqrt{\pm 1}y\in\mf g^{\alpha}_\C$.\footnote{Here we use $\sqrt{\pm 1}$ as a shorthand for $1$ or $\iu$ depending on the sign.} We call such $x$ and $y$ related. For any $a\in\mf a$ they satisfy
\begin{align*}
[a,x]&=\sqrt{\pm 1}\alpha(a)y \\ [a,y]&=\pm\sqrt{\pm 1}\alpha(a)x.
\end{align*}
\end{lemma}

\begin{proof}
First consider $\alpha\in\Delta_+$. Since $\mf g^\alpha_\C$ is invariant under complex conjugation, it contains a real root vector. More explicitly, if $z\in\mf g^\alpha_\C$, then $z+\overline z \in \mf g^\alpha$. Hence there are $x\in\mf k^\alpha$ and $y\in\mf p^\alpha$ such that $z+\overline z=x+y$. Now consider $\alpha\in\Delta_-$ and let $u\in\mf g^\alpha_\C$. We know that $s(\overline u)\in\mf g^\alpha_\C$. Let $v,w\in\mf g$ such that $u=v+\iu w$, then $u+s(\overline u)=v + \iu w + s(v) - \iu s(w)$. Let $z=(u+s(\overline u))/2$ then $z\in\mf g^\alpha_\C$ and $z=x+\iu y$ where $x=(v+s(v))/2\in\mf k$ and $y=w-s(w)\in\mf p$. Moreover $x=(z+\overline z)/2\in\mf g^\alpha$ and $y=\iu(-z+\overline z)/2\in\mf g^\alpha$. 
In both cases, the claimed equations follow from the fact that $[a,x+\sqrt{\pm 1}y]=\alpha(a)(x+\sqrt{\pm 1}y)$ by equating the $\mf k_\C$ and $\mf p_\C$ parts.
This concludes the proof.
\end{proof}

\begin{lemma} \label{lemma:related-commutator}
Let $\alpha\in\Delta_\pm$ be a non-zero root and let $x\in\mf k^\alpha$ and $y\in\mf p^\alpha$ be related. Then $[x,y]\in\mf a$ and if $\|y\|=1$, then $\braket{[x,y],a}=\sqrt{\pm1}\alpha(a)$, and hence $[x,y]$ only depends on $\alpha$. In particular $\|[x,y]\|^2=\sqrt{\pm1}\alpha([x,y])$.
\end{lemma}

\begin{proof}
First it is clear that $[x,y]\in\mf p$. Let $a\in\mf a$ be arbitrary, then $[a,[x,y]]=[[a,x],y]+[x,[a,y]]=0$ by Lemma~\ref{lemma:related-vectors}. In particular, if $a$ is regular this shows that $[x,y]\in\mf a$. Now we find
\begin{align*}
\braket{[x,y],a} = \pm B([x,y],a)=\pm B(y,[a,x])=\pm \sqrt{\pm 1}\alpha(a) B(y,y) = \sqrt{\pm 1}\alpha(a),
\end{align*}
as desired.
\end{proof}

\begin{lemma} \label{lemma:ad-related}
Let $(\mf g,s)$ be a semisimple, orthogonal, symmetric Lie algebra. Let $\alpha\in\Delta_\pm$ be a non-zero root. Let $x\in\mf k^\alpha$ and $y\in\mf p^\alpha$ be related unit vectors. Then, if $z=[x,y]$, it holds that
\begin{align*}
\ad_x^{2n-1}(z)&=(-\|z\|^2)^n y \\
\ad_x^{2n}(z)  &=(-\|z\|^2)^n z.
\end{align*}
\end{lemma}

\begin{proof}
The case $n=1$ follows immediately from Lemma~\ref{lemma:related-vectors} and Lemma~\ref{lemma:related-commutator}. Then induction yields the result.
\end{proof}

Note that the resulting expressions are the same for the compact and non-compact case. 

So far we have made no reference to an associated pair. In the next lemma we use an associated pair to relate the Weyl group action with the roots, but we note that the choice is arbitrary.

\begin{lemma} \label{lemma:refl-in-weyl}
Let $\mf g=\mf k\oplus\mf p$ be a semisimple, orthogonal, symmetric Lie algebra, and let $(G,K)$ be an associated pair.
Let $\alpha\in\Delta$ be a non-zero root. Then there exists $U\in N_K(\mf a)\cap K_0$ such that $U$ acts on $\mf a$ as orthogonal reflection with respect to the kernel of $\alpha$.
\end{lemma}

\begin{proof}
Let $x\in\mf k^\alpha$ and $y\in\mf p^\alpha$ be related unit vectors and let $z=[x,y]\in\mf a$. Then using Lemma~\ref{lemma:ad-related} we compute
\begin{align*}
\Ad_{e^{tx}}(z)=e^{\ad_{tx}}(z)
&=\sum_{n\geq0}\frac{\ad_{tx}^{2n}}{(2n)!}z + \sum_{n\geq1}\frac{\ad_{tx}^{2n-1}}{(2n-1)!}y \\
&= \cos(t \|z\|) z - \|z\| \sin (t\|z\|) y
\end{align*}
Setting $t=\pi/\|z\|$ we get $\Ad_{e^{tx}}z=-z$ and for $w\in\mf a$ satisfying $\braket{z,w}=0$ it holds that $[x,w]=-\sqrt{\pm1}\alpha(w)y=-\braket{z,w}y=0$. Hence $\Ad_{e^{tx}}w=w$. This concludes the proof.
\end{proof}

In fact these reflections generate the entire Weyl group. First we need an alternative characterization of regular elements.

\begin{lemma} \label{lemma:regular-aliter}
The commutant of an element $x\in\mf a$ in $\mf g_\C$ is given by
$$
(\mf g_\C)_x = \sum_{\substack{\alpha\in\Delta_0\\ \alpha(x)=0}} \mf g_\C^\alpha.
$$
In particular, $x$ is regular if and only if $\alpha(x)\neq0$ for all $\alpha\in\Delta$.
\end{lemma}

\begin{proof}
First note that $0$ is a root and $\mf p^0=\mf a$. Let $y\in\mf g_\C$. Then due to the rootspace decomposition $y=\sum_{\alpha\in\Delta_0}y_\alpha$ it holds that
$$
[x,y]=\sum_{\alpha\in\Delta} \alpha(x) y_\alpha.
$$ 
This proves the form of the commutant $(\mf g_\C)_x$. In particular 
$$\mf p_x=\bigoplus_{\substack{\alpha\in\Delta_0'\\ \alpha(x)=0}} \mf p^\alpha,$$ 
where $\Delta'_0$ contains $0$ and exactly one root of each pair $\alpha$ and $-\alpha$. Hence $\mf p_x=\mf a =\mf p^0$ if and only if $\alpha(x)\neq 0$ for all $\alpha\in\Delta$.
\end{proof}

The kernels of the non-zero roots define root hyperplanes in $\mf a$. The complement of the union of all root hyperplanes is a disjoint union of open connected components, called \emph{(open) Weyl chambers}. Lemma~\ref{lemma:regular-aliter} then shows that the regular elements in $\mf a$ are exactly the ones that lie in an open Weyl chamber.

\begin{lemma} \label{lemma:weyl-group-generators}
Let $\mf g=\mf k\oplus\mf p$ be a semisimple, orthogonal, symmetric Lie algebra, and let $(G,K)$ be an associated pair.
The Weyl group is finite and generated by the orthogonal reflections about the root hyperplanes. Furthermore it acts simply transitively on the Weyl chambers. 
\end{lemma}

\begin{proof}
Let $\mf n$ be the Lie algebra of $N_K(\mf a)$. Then for any $k\in\mf n$ and $x\in\mf a$ it holds that $[k,x]\in\mf a$. 
By Lemma~\ref{lemma:inner-prod-centralizer} it holds that $\ad_{\mf k}(x)\subset\mf a^\perp$, and thus $[k,x]=0$.
Hence $N_K(\mf a)$ and $Z_K(\mf a)$ have the same Lie algebra and the Weyl group must be discrete, and by compactness it must be finite. 

We will only sketch the remainder of the proof, see~\cite[Ch.~VII, Thm.~2.12]{Helgason78} for details. The plan is to show that the subgroup of the Weyl group generated by the root reflections (recall Lemma~\ref{lemma:refl-in-weyl}) acts transitively on the Weyl chambers, and moreover the Weyl group acts simply transitively on the Weyl chambers. This then shows that the Weyl group is generated by the reflections. For transitivity, let $W'$ denote the subgroup generated by the reflections, then for $x,y\in\mf a$ we can find $w\in W'$ such that $|x-w\cdot y|$ is minimal. Then $x$ and $w\cdot y$ lie in the same Weyl chamber, since otherwise there is a reflection which reduces the distance. For simple transitivity, assume that $w\in W$ maps some Weyl chamber $\mf w$ into itself. Then by averaging one finds a regular fixed point $x\in\mf w$, i.e., $w\cdot x=x$. Using the duality of Lemma~\ref{lemma:cpt-non-cpt-duality} we may assume that $(\mf g,s)$ is of compact type and $G$ is compact. Hence the closure of the one-parameter group generated by $x$ is a torus $T$. If $\Ad_U$ represents $w$, then $U$ commutes with $T$ and hence one can show that $U=e^k$ for some $k\in\mf k$ and $[k,x]=0$. Since $x$ is regular, $[k,\mf a]=0$ and hence $w=\id$.
\end{proof}

\begin{corollary} \label{coro:intersect-weyl-chamber}
Each $W$-orbit intersects the closed Weyl chamber $\mf w$ exactly once.
\end{corollary}

\begin{proof}
By Lemma~\ref{lemma:weyl-group-generators} each Weyl group orbit intersects $\mf w$, and for $x\in\mf a$ regular, this intersection is unique. Now consider $y\in\mf a$ singular and let $x\in\mf w$ be a regular point in some slice $U$ about $y$ (i.e., $U$ has the property that, for all $w\in W$, if $w\in W_y$, then $wU=U$, otherwise $(wU)\cap U=\emptyset$, see Def.~\ref{def:slice} for the general definition). Now assume that $y,wy\in\mf w$ for some $w\in W$ and consider $wx$ which lies in $wU$. One can show that there is an element in $w'\in W_{wy}$ such that $w'wx\in\mf w$. By uniqueness, $w'wx=x$ and $w'=w^{-1}$ and hence $w\in W_y$. Thus $wy=y$ as desired.
\end{proof}

\medskip
\subsection{Subalgebras and Quotients} \label{app:sym-quotients}

In this section we will look at some further properties of symmetric Lie algebras. In particular we are interested in subalgebras that appear as commutants and related quotient spaces. We start with a useful property of direct sums of symmetric Lie algebras.

\begin{lemma} \label{lemma:algebra-sum-orth}
Let $\mf g_i=\mf k_i\oplus\mf p_i$ be symmetric Lie algebras for $i=1,2$. Then $\mf g_1\oplus\mf g_2=(\mf k_1\oplus\mf k_2) \oplus (\mf p_1\oplus\mf p_2)$ is a symmetric Lie algebra. Moreover, $\mf g_1\oplus\mf g_2$ is orthogonal if and only if both $\mf g_1$ and $\mf g_2$ are orthogonal.
\end{lemma}

\begin{proof}
See~\cite[Ch.~V, Lem.~1.6]{Helgason78}.
\end{proof}

Next we consider symmetric subalgebras of a symmetric Lie algebra $(\mf g, s)$, and we will see how certain properties are inherited. In general we say that a subset of $\mf g$ is \emph{symmetric}, if it is left invariant by $s$.

\begin{Definition}
Let $(\mf g, s)$ be a symmetric Lie algebra and let $\mf h\subseteq\mf g$ be a Lie subalgebra invariant under $s$. Then we say that $\mf h$, or more precisely $(\mf h, s|_{\mf h})$, is a \emph{symmetric Lie subalgebra of $(\mf g, s)$}.
\end{Definition}

The Cartan-like decomposition of the symmetric Lie subalgebra $\mf h\subseteq\mf g$ is given by $\mf h = (\mf h\cap\mf k)\oplus(\mf h\cap\mf p)$. The following results show that orthogonality is inherited, but semisimplicity has to be replaced by the slightly weaker condition of reductivity\footnote{Recall that a Lie algebra (over a field of characteristic $0$) is reductive if and only if it can be written as a direct sum of a semisimple and an Abelian Lie algebra.}.
Let us first recall a basic fact about reductive Lie algebras.

\begin{lemma} \label{lemma:reductive-decomp}
Let $\mf g$ be a reductive Lie algebra, then $\mf g = [\mf g,\mf g]\oplus\mf z$ where $\mf z$ is the center of $\mf g$ and $[\mf g,\mf g]$ is semisimple, and this decomposition of $\mf g$ into a direct sum of a semisimple and an Abelian subalgebra is unique.
\end{lemma}

\begin{proof}
Let $\mf g = \mf s\oplus\mf a$ where $\mf s$ and $\mf a$ are semisimple and Abelian ideals in $\mf g$.
First we note that $\mf a$ must equal the center of $\mf g$. 
Indeed, it is clear that $\mf a\subseteq\mf z$.
Conversely, let $s+a\in\mf z$ where $s\in\mf s$ and $a\in\mf a$, then $0=[s+a,\mf s]=[s,\mf s]$ and so $s=0$.
Finally $[\mf g,\mf g]=[\mf s,\mf s]=\mf s$ since $\mf s$ is semisimple and hence perfect.
\end{proof}

\begin{lemma} \label{lemma:reductive-lie-sub-alg}
Let $(\mf g, s)$ be a semisimple, orthogonal, symmetric Lie algebra and let $\mf h\subseteq\mf g$ be a symmetric Lie subalgebra.
Moreover assume that $\mf h\cap\mf p=(\mf h\cap\mf p_-)\oplus(\mf h\cap\mf p_+)$, where $\mf p_\pm$ denote the non-compact and compact parts.
Then $(\mf h, s|_{\mf h})$ is a reductive, orthogonal, symmetric Lie algebra. Moreover setting $\mf h'=[\mf h,\mf h]$, it holds that $(\mf h', s|_{\mf h'})$ is a semisimple, orthogonal, symmetric Lie algebra.
\end{lemma}

\begin{proof}
This result generalizes~\cite[Cor.~6.29]{Knapp02}. We will work with the inner product on $\mf g$ defined in Lemma~\ref{lemma:nice-inner-prod}. 
Consider the adjoint representation of $\mf h$ on $\mf g$. 
By semisimplicity this is a faithful representation, and the assumption on $\mf h\cap\mf p$ implies that $\ad_h|_{\mf g}$ is closed under transposition, c.f. Corollary~\ref{coro:closed-under-transpose}. 
By~\cite[Prop.~1.56]{Knapp02} a Lie algebra of real matrices closed under transposition is reductive, and so $\mf h$ is reductive. 
By Lemma~\ref{lemma:reductive-decomp} $\mf h'$ is semisimple. Clearly $\mf h'$ is also invariant under $s$, and so by Lemma~\ref{lemma:orth-semisimple-converse} it holds that $(\mf h', s|_{\mf h'})$ is orthogonal. 
Let $\mf z$ denote the center of $\mf h$. For $h\in\mf h$ and $z\in\mf z$ we see by $[s(z),h]=s([z,s(h)])$ that $\mf z$ is a symmetric subalgebra, and it is trivially orthogonal. Hence also $\mf h=\mf h'\oplus\mf z$ is orthogonal by Lemma~\ref{lemma:algebra-sum-orth}.
\end{proof}

Centralizers (commutants) are a common source of symmetric subalgebras, so we briefly give a more general definition and notation. Let $\mf g$ be a Lie algebra and let $A,B\subseteq\mf g$ be arbitrary subsets. Then the centralizer of $A$ in $B$ is defined as $B_A:=\{b\in B : [b,a] = 0 \text{ for all } a\in A\}$. Of course these subsets have related stabilizer subgroups $G_A$ and $K_A$. 

\begin{lemma}
\label{lemma:semisimple-subsystem}
Let $(\mf g, s)$ be a semisimple, orthogonal, symmetric Lie algebra with associated pair $(G,K)$.
Let $A\subset\mf p$ be any subset and let $\mf g_A$ be its commutant. 
Then $\mf g_A$ is a reductive, orthogonal, symmetric Lie subalgebra with Cartan-like decomposition $\mf g_A=\mf k_A\oplus\mf p_A$, and $\mf g_A$ is invariant under the action of the stabilizer $G_A$. 
Moreover $\mf p_A=(\mf p_A\cap\mf p_-) \oplus (\mf p_A\cap\mf p_+)$ and for every $x\in\mf p_A$, we get the orthogonal decomposition $\mf p_A=\mf p_{A,x}\oplus[\mf k_A,x]$.
\end{lemma}

\begin{proof}
That $\mf g_A$ is a Lie subalgebra follows from the Jacobi identity. Let $x\in\mf g_A$, since $[s(x),z]=-s([x,z])=0$ for $z\in A$, we see that $\mf g_A$ is invariant under $s$. 
Similarly, if $U\in G_A$ then $[\Ad_U(x),z]=\Ad_U([x,\Ad_U^{-1}(z)])=0$. Finally consider $x=x_-+x_+\in\mf p_A$. Hence for any $z\in A$ it holds that $[x_-,z]+[x_+,z]=0$, and since $\mf k_-$ and $\mf k_+$ have zero intersection, this means that $[x_-,z]=[x_+,z]=0$. By Lemma~\ref{lemma:reductive-lie-sub-alg} this shows that $(\mf g_A,s)$ is reductive and orthogonal.
Let $x=x_-+x_+\in\mf p_A$ and $y=y_-+y_+\in\mf p_A$ and $k\in\mf k_A$. Then using the inner product and Killing form on $\mf g$ we compute $\braket{[k,x],y}=B(k,[x_+,y_+]-[x_-,y_-])$. Hence $\mf p_{A,x}\subseteq\mf p_A\cap[\mf k_A,x]^\perp$. Since $[x_+,y_+]-[x_-,y_-]\in\mf k_A$, the converse is true by Lemma~\ref{lemma:B-neg-def-on-k}.
\end{proof}

To better understand the action of $K_A$ on $\mf g$ we give the following powerful generalization of Lemma~\ref{lemma:max-abelian-conj}.

\begin{lemma} \label{lemma:max-abel-pA}
Let $\mf g=\mf k\oplus\mf p$ be a semisimple, orthogonal, symmetric Lie algebra, and let $(G,K)$ be an associated pair. Let $\mf a\subseteq\mf p$ be some maximal Abelian subspace. Let $A\subset\mf a$ be any subset. For any $y\in\mf p_A$, there is some $U\in K_A$ such that $\Ad_U(y)\in\mf a$.
\end{lemma}

\begin{proof}
As in the proof of Lemma~\ref{lemma:max-abelian-conj} we may assume that we are working with the canonical pair $(G,K)=(\Int(\mf g),\Int_{\mf k}(\mf g))$, and in particular $K$ is compact.
Now let $x\in\mf a$ be regular, and let $y\in\mf p_A$. Consider the smooth function 
\begin{align*}
f:K_A\to\R,\quad f(U) = B(U\cdot y,x)
\end{align*}
If $U$ is a critical point for $f$, which exists since $K_A$ is compact, then for every $k\in\mf k_A$ it holds that
\begin{align*}
0=B([k,U\cdot y],x)=B(k,[U\cdot y,x])
\end{align*}
and since $[U\cdot y,x]\in\mf k_A$ this means that $[U\cdot y,x]=0$ by Lemma~\ref{lemma:B-neg-def-on-k}. Since $x$ is regular, $U\cdot y\in\mf a$.
\end{proof}

\begin{lemma}
\label{lemma:weyl-group-equivalent}
Let $\mf g=\mf k\oplus\mf p$ be a semisimple, orthogonal, symmetric Lie algebra, and let $(G,K)$ be an associated pair. Let $\mf a\subseteq\mf p$ be some maximal Abelian subspace. Let $A\subset\mf a$ be any subset and let $U\in K$ be such that $\Ad_U(A)\subseteq\mf a$. Then there is some $w\in W$ such that $w\cdot x = \Ad_U(x)$ for all $x\in A$.
\end{lemma}

\begin{proof}
Let $U\in K$ be as in the statement, and let $x\in\mf a$ be regular. Then $\Ad_U^{-1}(x)\in\mf p_A$ and by Lemma~\ref{lemma:max-abel-pA} there exists $V\in K_A$ such that $\Ad_{VU^{-1}}(x)\in\mf a$. Hence $UV^{-1}\in N_K(\mf a)$ (note we took the inverse) and $\Ad_{UV^{-1}}=\Ad_{U}$ on $A$ and hence the Weyl group element corresponding to $UV^{-1}$ does the job.
\end{proof}

We immediately get some useful corollaries.

\begin{corollary}
\label{coro:centralizer-stabilizer}
Let $A\subseteq\mf a$ be any subset. Let $\mf p_A$ be the centralizer of $A$ in $\mf p$ and $K_A$ the stabilizer of $A$ in $K$. Then $\mf p_A=\Ad_{K_A}(\mf a)$.
\end{corollary}

\begin{proof}
First let $U\in K_A$, and $x\in A$, and $y\in\mf a$. Then $[x,\Ad_U(y)]=\Ad_U([x,y])=0$. Hence $\Ad_{K_A}(\mf a)\subseteq\mf p_A$. For the reverse inclusion let $x\in\mf p_A$. Then by Lemma~\ref{lemma:max-abel-pA} there is $U\in K_A$ such that $\Ad_U(x)\in\mf a$. Hence $x\in \Ad_U^{-1}(\mf a)\subseteq\Ad_{K_A}(\mf a)$.
\end{proof}

\begin{corollary} \label{coro:wx-kx}
For any $x\in\mf a$ it holds that $Wx=Kx\cap\mf a$. In particular the choice of $K$, even if $K$ is disconnected, does not affect the $K$-orbits in $\mf p$.
\end{corollary}

\begin{corollary} \label{coro:simul-diag}
Let $A\subset\mf a$ and $B\subset\mf p$ such that $A\cup B$ is Abelian. Then there is some $U\in K_A$ such that $\Ad_U(B)\subseteq\mf a$.
\end{corollary}

We have encountered several Lie group actions, such as the action of $K$ on $\mf p$ or that of $W$ on $\mf a$. In the following we want to understand the structure of the corresponding quotient spaces $\mf p/K$ and $\mf a/W$. For this the concept of a slice is crucial. We recall the definition here, see~\cite[Def.~3.47]{Alexandrino15}.

\begin{Definition} \label{def:slice}
Let $G$ be a Lie group, $M$ a smooth manifold, and $\mu:G\times M\to M$ a smooth action. A \emph{slice at $x\in M$ for the action $\mu$ of $G$} is an embedded submanifold $S_x$ of $M$ containing $x$ and satisfying the following properties:
\begin{enumerate}[(i)]
\item $T_xM=D\mu_x(\mf g)\oplus T_xS_x$ and $T_yM=D\mu_y(\mf g)+T_yS_x$ for $y\in S_x$ and where $\mu_x(g)=g\cdot x$.
\item $S_x$ is invariant under the stabilizer $G_x=\{g\in G:g\cdot x=x\}$.
\item if $y\in S_x$ and $g\in G$ such that $g\cdot y\in S_x$, then $g\in G_x$.
\end{enumerate}
\end{Definition} 

By~\cite[Thm.~3.49]{Alexandrino15}, proper Lie group actions on manifolds admit slices at every point of the manifold. The next result is an immediate consequence of Lemma~\ref{lemma:orbit-tangent-centralizer} and links slices at $x$ to the commutant $\mf p_x$.

\begin{lemma} \label{lemma:slice}
Let $\mf g=\mf k\oplus\mf p$ be a semisimple, orthogonal, symmetric Lie algebra, and let $(G,K)$ be an associated pair.
Let $x\in\mf p$ and let $S_x = \mf p_x \cap B_\epsilon(x)$ where $B_\epsilon(x)$ is an $\epsilon$-ball around $x$ in $\mf g$. Then for $\epsilon>0$ small enough, $S_x$ is a slice at $x$ for the adjoint action of $K$ on $\mf p$.
\end{lemma}

\begin{proof}
This follows from the proof of the slice theorem, see~\cite[Thm.~3.49]{Alexandrino15}. In that proof, for an isometric action of $K$ on $M$, a slice about $x\in M$ is constructed by taking the exponential of an $\epsilon$-ball around the origin in the subspace of $T_xM$ orthogonal to the tangent space of the orbit $T_x(Kx)$. By Lemma~\ref{lemma:orbit-tangent-centralizer} we have the orthogonal decomposition $T_x\mf p=T_x(Kx)\oplus\mf p_x$. Since $\mf p$ is a vector space we can canonically identify the tangent space $T_x\mf p$ with $\mf p$ and the exponential function is simply $\exp_x(v) = x+v$. This concludes the proof.
\end{proof}

\begin{corollary} \label{coro:containment-of-slices}
Let $x\in\mf p$ and let $S_x$ be a slice, then for all $y\in S_x$ it holds that $\mf p_y \subseteq \mf p_x$.
\end{corollary}

\begin{proof}
Without loss of generality $x,y\in\mf a$. By Corollary~\ref{coro:centralizer-stabilizer} it holds that $\mf p_x = \Ad_{K_x}(\mf a)$ and analogously for $y$. But if $y\in S_x$ then $K_y\subseteq K_x$ by the definition of a slice.
\end{proof}

The next results relate the quotients $\mf a/ W$ and $\mf p/ K$, showing that they are isometric. 
But first let us clarify what metric each space is endowed with. As usual all subspaces of $\mf g$ are given the $K$ invariant inner product of Lemma~\ref{lemma:nice-inner-prod}, which induces a norm and a metric. Hence the actions of $K$ on $\mf p$ and of $W$ on $\mf a$ are isometric. The closed Weyl chamber $\mf w\subset\mf a$ inherits the metric on $\mf a$. We will consider quotient spaces such as $\mf p/K$, and $\mf p_x/K_x$, and $\mf a/W$. Then the following lemma describes the relevant metric properties.

\begin{lemma} \label{lemma:quotient-metric}
Let $M$ be a complete Riemannian manifold and $G$ a compact Lie group acting isometrically on $M$. Then the quotient $M/ G$ with the distance $d(Gx,Gy):=d(x,Gy)$ becomes a metric space and the quotient map $\pi:M\to M/ G$ is non-expansive.
\end{lemma}

\begin{proof}
Since $G$ acts by isometries, the distance is well defined. The axioms of a metric are easily verified. It is clear that $d(Gx,Gy)\leq d(x,y)$, showing that the quotient map $\pi$ is non-expansive.
\end{proof}

\begin{lemma} \label{lemma:triple-bijection}
The maps $\psi:\mf w\to\mf a/ W,\, x\mapsto Wx$ and $\phi:\mf a/ W\to\mf p/ K,\, Wx\mapsto Kx$ are bijections.
\end{lemma}

\begin{proof}
The map $\psi$ is bijective since by Corollary~\ref{coro:intersect-weyl-chamber} every $W$ orbit in $\mf a$ intersects $\mf w$ in exactly one point. 
The map $\phi:Wx\mapsto Kx$ is well defined since for any $y\in Wx$ there exists by definition of the Weyl group some element $U\in N_K(\mf a)$ such that $y=\Ad_U(x)$. Injectivity of $\phi$ follows from Corollary~\ref{coro:wx-kx} and surjectivity follows from Lemma~\ref{lemma:max-abelian-conj}.
\end{proof}

\begin{lemma} \label{lemma:triple-isometry}
The maps $\psi:\mf w\to\mf a/ W$ and $\phi:\mf a/ W\to\mf p/ K$ of the previous lemma are isometries with respect to any $K$-invariant inner product on $\mf p$ and its restriction to $\mf a$ and $\mf w$.
\end{lemma}

\begin{proof}
We start by showing that $\psi$ is an isometry. Let $\|\cdot\|$ denote any norm induced by a $K$-invariant inner product on $\mf p$. This norm also induces the metric on $\mf w$ and $\mf a$. 
Since the action of $W$ on $\mf a$ is isometric, it holds that $d(Wx,Wy)=\min_{w\in W}\|x-wy\|$ where $x,y$ can always be chosen in $\mf w$. 
However the minimum must be achieved by the identity in $W$, since otherwise the segment connecting $x$ and $wy$ lies in more than one Weyl chamber. Reflecting this segment into $\mf w$ yields a continuous, piecewise linear path in $\mf w$ connecting $x$ to $y$ which must be longer than the line segment connecting $x$ to $y$. Hence $d(Wx,Wy)=\|x-y\|$.

Next we show that $\phi$ is an isometry. Let $x,y\in\mf a$ with $x$ regular. First note that $d(\Ad_K(x),\Ad_K(y))\leq d(Wx,Wy)$. Since $K$ acts isometrically on $\mf p$ it holds that $d(\Ad_K(x),\Ad_K(y))=d(\Ad_K(x),y)$. But any geodesic in $\mf p$ realizing the distance $d(\Ad_K(x),y)$ must be a line segment starting at $x$ and orthogonal to $\Ad_K(x)$, see~\cite[Ch.~9, Example~1]{doCarmo92}. Hence it is contained in $\mf a$ and so $d(Kx,Ky)\geq d(Wx,Wy)$ and hence they are equal. The proof for $x$ non-regular is similar: again we may assume that $x\in\mf a$, and we know that the segment realizing the distance $d(\Ad_K(x),y)$ is orthogonal to $\Ad_K(x)$, and hence contained in $\mf p_x$, and so is $y$. By Lemma~\ref{lemma:max-abel-pA}, there is some $U\in K_x$ such that $\Ad_U(y)\in\mf a$. This concludes the proof.
\end{proof}

\begin{corollary}
\label{coro:geodesic-segment-in-w}
Let $x,y\in\mf w$ be distinct. Then the line segment connecting $x$ to $y$ is a geodesic segment in $\mf p$ realizing the distance between the orbits $\Ad_K(x)$ and $\Ad_K(y)$. 
\end{corollary}

\begin{proof}
By Lemma~\ref{lemma:triple-isometry} the distance in $\mf w$ between $x$ and $y$ is the same as the distance in $\mf p$ between the orbits $\Ad_K(x)$ and $\Ad_K(y)$. Since the straight line segment in $\mf w$ realizes the distance between $x$ and $y$ in $\mf w$, the same line segment considered in $\mf p$ realizes the distance between the orbits.
\end{proof}



Similarly we can relate the quotients $\mf a/ W_x$ and $\mf p_x/ K_x$ for $x\in\mf p$. In this case we just need a homeomorphism.

\begin{corollary} \label{coro:homeo-a-p}
Let $x\in\mf a$. The inclusion $\iota : \mf a\hookrightarrow\mf p_x$ descends to a homeomorphism $\phi_x:\mf a/ W_x\to\mf p_x/ K_x$, $W_xy\mapsto K_xy$.
\end{corollary}

\begin{proof}
First we show that $\phi_x$ is a well defined bijection. Let $y,z\in\mf a$ and let $w\in W_x$. 
By definition of $W_x$ there is some $U\in K_x$ with $\Ad_U(y)=w\cdot y$. Hence $\phi_x$ is well-defined. 
If $K_xy=K_xz$, that is $y=\Ad_U(z)$ for some  $U\in K_x$, then by Lemma~\ref{lemma:weyl-group-equivalent} there is some $w\in W_x$ with $y=w\cdot z$, so $\phi_x$ is injective.
Surjectivity follows from Lemma~\ref{lemma:max-abel-pA}.
By the definition of the quotient topology, $\phi_x$ is continuous, and it remains to show that the inverse is too.

We show that $\phi_x$ is open. Let $y\in\mf a$. By Lemma~\ref{lemma:semisimple-subsystem}, the centralizer of $y$ in $\mf p_x$ is the orthogonal complement of the tangent space at $y$ of $\Ad_{K_x}(y)$ in $\mf p_x$. Setting $A=\{x,y\}$, this centralizer it denoted by $\mf p_A$ and for $\epsilon$ small enough $S_y = \mf p_A \cap B_\epsilon(y)$ is a slice at $y$ for the action of $K_x$ on $\mf p_x$, similarly to Lemma~\ref{lemma:slice}. Now let $O$ be an open neighborhood of $y$ in $\mf a$. Then, if $\epsilon$ is chosen small enough, $S_y \subseteq \Ad_{K_A}(O)$ and by the tubular neighborhood theorem, see~\cite[Thm.~3.57]{Alexandrino15}, $\Ad_{K_x}(S_y)$ contains $y$ in its interior. Hence the image of $O$ in $\mf p_x/ K_x$ contains the image of $y$ in its interior.
\end{proof}

Moreover we can show that whenever $y=\Ad_U(x)$, the quotients $\mf p_x/ K_x$ and $\mf p_y/ K_y$ are isomorphic in a unique way.

\begin{lemma}
\label{lemma:homeo-slices}
Let $x\in\mf p$ and let $y=\Ad_U(x)$ for some $U\in K$. Then $\Ad_U : \mf p_x \to \mf p_y$ is a linear $K_x$-$K_y$-equivariant isomorphism. Hence it descends to a homeomorphsim $\phi_{x,y}:\mf p_x/ K_x\to\mf p_y/ K_y$ which does not depend on the choice of $U$.
Furthermore if $z\in\mf p$ belongs to the same $K$-orbit as $x$ and $y$, then $\phi_{yz}\circ\phi_{xy}=\phi_{xz}$, or equivalently $\phi_{xx}$ is the identity.
\end{lemma}

\begin{proof}
First we show that $\Ad_U : \mf p_x \to \mf p_y$ is a linear isomorphism. Linearity and invertibility are clear, we just need to show that $\Ad_U(z)\in\mf p_y$ for $z\in\mf p_x$. But this follows from $[y,\Ad_U(z)] = \Ad_U([\Ad_U^{-1}(y),z]) = \Ad_U([x,z]) = 0$. 
Let $V\in K_x$, then clearly $\Ad_U\circ\Ad_V=\Ad_{UVU^{-1}}\circ\Ad_U$, that is, $\Ad_U$ is equivariant. In particular, it maps orbits to orbits bijectively, and hence $\Ad_U$ induces a well-defined bijection $\phi_{x,y}:\mf p_x/ K_x\to\mf p_y/ K_y$. Note that $\phi_{x,y}$ does not depend on the choice of $U$, since any other choice differs from $U$ by multiplication with a stabilizer element, which leaves the orbits unchanged. Continuity of $\phi_{x,y}$ follows from the definition of the quotient topology, and continuity of the inverse follows analogously.
\end{proof}

\begin{lemma} \label{lemma:Weyl-stabilizer}
Let $x\in\mf a$. Then the stabilizer subgroup $W_x$ is generated by the reflections $s_\alpha$ corresponding to the roots $\Delta_x=\{\alpha\in\Delta:\alpha(x)=0\}$. Hence $\Delta_x$ is a (possibly non-reduced) root system on its span and $W_x$ is its Weyl group. In particular $W_x$ acts simply transitively on its Weyl chambers.
\end{lemma}

\begin{proof}
By~\cite[Ch.~VII, Thm.~2.16]{Helgason78} $\Delta$ is a (generally non-reduced) root system and $W$ is the corresponding Weyl group. By~\cite[Ch.~X~Lem.~3.2]{Helgason78} $\Delta$ contains a reduced root system $\Delta'\subseteq\Delta$ with the same Weyl group.
By~\cite[Sec.~10.3~Lemma~B]{Humphreys72} it holds that $W_x$ is generated by reflections with respect to elements $\alpha\in\Delta'_x\subseteq\Delta_x$. Since for $\alpha,\beta\in\Delta_x$ it holds that $s_{s_\alpha(\beta)}(x)=s_\alpha\circ s_\beta\circ s_\alpha^{-1}$ we know that $\Delta_x$ is invariant under its own reflections. Since $\Delta_x\subseteq\Delta$, the integrality condition is inherited and hence $\Delta_x$ is a possibly non-reduced root system on its span.
\end{proof}

\begin{remark}
Corollary~\ref{coro:homeo-a-p} and Lemma~\ref{lemma:homeo-slices} show that for $x\in\mf a$ we can identify $\mf a/ W_x$ and $\mf p_x/ K_x$ and $\mf p_y/ K_y$ for all $y\in\Ad_K(x)$. By Lemma~\ref{lemma:Weyl-stabilizer}, $W_x$ is itself a Weyl group. Furthermore one can identify $\mf a/ W_x$ with the orbifold tangent space $T_{\pi(x)}(\mf a/ W)$, see Definition~\ref{def:orbifold-tangent-bundle}. 
\end{remark}

\section{Orbifolds}
\label{app:orbifolds}

In this appendix we give the necessary background on orbifolds and prove some technical results which are used in the main text. For a general introduction to orbifolds see~\cite{Adem07}. We only consider the local theory, that is, we work in a single linear orbifold structure chart. This section does not presuppose any knowledge of orbifolds. We note that the orbifolds encountered in the main text are of a special kind due to the Weyl group structure, but we will not make use of this assumption in this appendix.

To kick things off we recall a basic topological concept in the context of group actions: given a finite dimensional real (or complex) vector space $V$, and a finite group $\Gamma$ acting linearly (and thus continuously) on $V$, we denote by $V/\Gamma$ the usual quotient space endowed with its quotient topology. Moreover,  $\pi:V\to V/\Gamma$, $x\mapsto{[x]}$ denotes the quotient map which is continuous by definition of the quotient topology and one can easily show that it is open\footnote{Let $G$ be any group acting on a topological space $X$ by homeomorphisms, then $\pi:X\to X/G$ is open. Indeed let $U\subseteq X$ be any open subset. By definition of the quotient topology it holds that $\pi(U)$ is open if and only if $\pi^{-1}(\pi(U))$ is open. But $\pi^{-1}(\pi(U))=\bigcup_{g\in G} gU$ is clearly open.}. With this we can define tangent spaces of points in $V/\Gamma$ by ``pulling over'' the well-known concept of tangent spaces of manifolds:

\begin{Definition}
\label{def:orbifold-tangent-bundle}
Let $V$ be a finite dimensional real vector space, and $\Gamma$ a finite group acting linearly on $V$.
We define the tangent bundle of $V/\Gamma$ to be $T(V/\Gamma) := (TV)/ \Gamma$ where---when identifying $T_xV$ with $V$ as usual---the action of $\Gamma$ on $TV$ is given by $ g\cdot(x,v)=( g\cdot x, g\cdot v)$. One can illustrate this with the following commutative diagram:\bigskip
\begin{align*}
\centering
\adjincludegraphics[valign=c,width=0.75\textwidth]{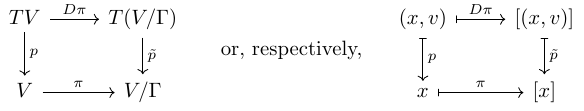}
\end{align*}
for all $x\in V$, $v\in T_xV$, where $\pi$ and $D\pi$ are the respective quotient maps and $p,\tilde p$ are the respective footpoint maps.
\end{Definition}

One readily verifies that $\tilde p$ is well-defined. Moreover, one can show that for all $x\in V$ the tangent space (also called tangent cone) of ${[x]}$---that is, $\tilde p^{-1}({[x]})$---is homeomorphic\footnote{See for instance~\cite[p.~11]{Adem07}. As usual, $\Gamma_x:=\{ g\in\Gamma: g x=x\}$ denotes the stabilizer of $x$ in $\Gamma$.} to $(T_x V)/ \Gamma_x$ by means of the map
$\tilde p^{-1}({[x]})\to (T_x V)/ \Gamma_x$, $ [(x,v)]\mapsto\{g\cdot v:g\in\Gamma_x\}$. 
We denote the tangent space $\tilde p^{-1}({[x]})\simeq (T_x V)/ \Gamma_x$ by $T_{\pi(x)}(V/\Gamma)$ or $T_{[x]}(V/\Gamma)$.

\begin{example} \label{ex:R_Z2}
Consider the manifold $\R$ with the group action of $\Z_2$ whose non identity element acts by $x\mapsto-x$. The quotient map is the absolute value: $\pi(x)=|x|$. Then the action on the tangent bundle $T\R$ is given by $(x,v)\mapsto(-x,-v)$ (which is not a reflection). Note that the only fixed point of this action is $(0,0)$. Then for $x\neq0$ it holds that $T_{|x|}(\R/ \Z_2)\cong\R$ and $T_{|0|}(\R/ \Z_2)\cong\R/\Z_2$.
\end{example}

After establishing the basic setting we can make sense of differentiating within the quotient:

\begin{Definition}
\label{def:orbifold-derivatives}
Let $V$ be a finite dimensional real vector space, and $\Gamma$ a finite group acting linearly on $V$. Let $I$ be an open interval and let $\xi:I\to V/\Gamma$ be a continuous function. Given $t\in I$ we say that $\xi$ is differentiable at $t$ if there exists $\lambda:I\to V$ such that $\xi\equiv\pi\circ\lambda$ (such $\lambda$ is called a ``lift'') and that $\lambda$ is differentiable at $t$. We then call\footnote{Another way to write this would be $D\xi(t)=D\pi((\lambda(t),\lambda'(t)))$.}
$$
D\xi(t):=\{g\cdot(\lambda(t),\lambda'(t)):g\in\Gamma\} \in T(V/\Gamma)
$$
the derivative of $\xi$ at $t$.
We say that $\xi$ is differentiable if it is differentiable at every $t\in I$, and we denote the derivative by
$
D\xi : I \to T(V/\Gamma).
$
If additionally the derivative $D\xi$ is continuous, then we say that $\xi$ is $C^1$.
\end{Definition}

Using the previously discussed homeomorphism one could equivalently define $D\xi(t)$ as the collection $\{g\lambda'(t):g\in\Gamma_{\lambda(t)}\} \in (T_{\lambda(t)}V)/ \Gamma_{\lambda(t)}$. Our definition however has the advantage that all derivatives live in the same space $T(V/\Gamma)$.

\begin{remark} 
Properly defining derivatives of maps between orbifolds is notoriously difficult. The situation is easier for us, since we only deal with paths. Our definition is quite general, as it is ``pointwise''. In Proposition~\ref{prop:orbifold-C1-lift} we show that it is equivalent to a seemingly stronger definition, which in turn is similar to the original definition given in~\cite{Satake56}.
\end{remark}

Of course we first have to make sure that the derivative of $\xi$ is well-defined in the first place, that is, it does not dependent on the chosen lift $\lambda$:

\begin{lemma} \label{lemma:deriv-well-def}
Let a finite dimensional real vector space $V$ as well as a finite group $\Gamma$ acting linearly on $V$ be given. Moreover, let $I$ be an open interval and let $\xi:I\to V/\Gamma$ be continuous. Now for arbitrary $t\in I$ the following statements hold:
\begin{enumerate}[(i)]
\item \label{it:deriv-well-def} If $\xi$ is differentiable at $t$, then the derivative $D\xi(t)$ is a well-defined element of $T(V/\Gamma)$.
\item \label{it:cont-lifts} Given any two lifts $\lambda,\mu$ of $\xi$ which are continuous at $t$ there exists a neighborhood $I'$ of $t$ in $I$ such that the following hold: if $\mu(s)=g\lambda(s)$ for some $s\in I'$ and some $g\in\Gamma$, then $\mu(t)=g\lambda(t)$, and if $\mu(t)=h\lambda(t)$, then for each $s\in I'$ there is some $g\in\Gamma_{\mu(t)}$ such that $g\mu(s)=h\lambda(s)$.
\item \label{it:left-right-derivs} Assume that $\xi$ is differentiable at $t$. If $\lambda$ is any lift of $\xi$ which is continuous on $(t-\varepsilon,t+\varepsilon)$ for some $\varepsilon>0$, then $\lambda$ admits a left and a right derivative at $t$ which are both elements of $D\xi(t)$.
\end{enumerate}
\end{lemma}

\begin{proof}
\eqref{it:deriv-well-def}: Let $\lambda_1$ and $\lambda_2$ be lifts of $\xi$ that are differentiable at $t$.
What we have to show now is that
$$
\{g\cdot(\lambda_1(t),\lambda_1'(t)):g\in\Gamma\}=\{\tilde g\cdot(\lambda_2(t),\lambda_2'(t)):\tilde g\in\Gamma\}\,.
$$
For this let $g\in\Gamma$ as well as any sequence $(t_n)_{n\in\mathbb N}$ in $I$ which converges to $t$ be given. By the lift property of $\lambda_1,\lambda_2$ there exists for all $n\in\mathbb N$ some $ g_n\in\Gamma$ such that $g\lambda_1(t_n)= g_n\lambda_2(t_n)$.

Now because $\Gamma$ is finite there exists a subsequence $(g_{n_k})_{k\in\mathbb N}$ of $ (g_n)_{n\in\mathbb N}$ which is constant, that is, equal to some $ \tilde g\in\Gamma$. Using continuity of $\lambda_1,\lambda_2$, and the group action we compute
\begin{align*}
g\lambda_1(t)=\lim_{k\to\infty}g\lambda_1(t_{n_k})=\lim_{k\to\infty}g_{n_k}\lambda_2(t_{n_k})=\tilde g\lambda_2(t)\,.
\end{align*}
For all $k\in\mathbb N$ this yields
\begin{align*}
g\cdot\frac{\lambda_1(t_{n_k})-\lambda_1(t)}{t_{n_k}-t}
&=\frac{g\lambda_1(t_{n_k})-g\lambda_1(t)}{t_{n_k}-t}
=\frac{g_{n_k}\lambda_2(t_{n_k})-\tilde g\lambda_2(t)}{t_{n_k}-t}
\\&=\frac{\tilde g\lambda_2(t_{n_k})-\tilde g\lambda_2(t)}{t_{n_k}-t}
=\tilde g\cdot\frac{\lambda_2(t_{n_k})-\lambda_2(t)}{t_{n_k}-t}
\end{align*}
so taking the limit $k\to\infty$ shows $g\lambda_1'(t)=\tilde g\lambda_2'(t)$ meaning $\tilde g$ is the group element we were looking for.\medskip

\eqref{it:cont-lifts}: First of all, the lift property guarantees that one finds $h\in\Gamma$ which satisfies $\mu(t)=h\lambda(t)$. Now because $\Gamma$ is finite there exists an open set $U$ containing $\mu(t)$ (called ``slice'') with the property that, for all $g\in\Gamma$, if $g\mu(t)=\mu(t)$, then $gU=U$, and if $g\mu(t)\neq\mu(t)$ then $(gU)\cap U=\emptyset$\footnote{This is indeed a special case of Definition~\ref{def:slice}.}. Combining this with the lift property we get $\lambda(t)=h^{-1}\mu(t)\in h^{-1}U$. Therefore continuity yields $I'\subseteq I$ such that on all of $I'$, $\mu$ lies in $U$ and $\lambda$ lies in $h^{-1}U$.

For the first part we have to show is that, given any $s\in I'$ and any $g\in\Gamma$ such that $\mu(s)=g\lambda(s)$, one also has $\mu(t)=g\lambda(t)$. But by our previous continuity argument we know that $\mu(s)\in U$ as well as $\mu(s)=g\lambda(s)\in g(h^{-1}U)=gh^{-1}U$. Therefore $U\cap (gh^{-1}U)\neq\emptyset$ which by the slice property implies $\mu(t)=gh^{-1}\mu(t)=g(h^{-1}\mu(t))=g\lambda(t)$, as desired.

For the second part note that since $h\lambda(s)\in U$ for all $s\in I'$, any $g\in\Gamma$ satisfying $g\mu(s)=h\lambda(s)$ must lie in $\Gamma_{\mu(t)}$.

\medskip

\eqref{it:left-right-derivs}: Now let $\lambda,\mu$ be lifts of $\xi$ where $\lambda$ is arbitrary but continuous on $(t-\varepsilon,t+\varepsilon)$ for some $\varepsilon>0$, and $\mu$ is differentiable at $t$ (such $\mu$ exists because $\xi$ is assumed to be differentiable). Again we start by using the lift property, that is, for all $s\in(t-\varepsilon,t+\varepsilon)$ there exists $g_s\in\Gamma$ such that $\mu(s)=g_s\lambda(s)$. On the other hand $\lambda,\mu$ are both continuous at $t$, so \eqref{it:cont-lifts} yields $\delta>0$ such that for all $s\in (t-\min\{\delta,\varepsilon\},t+\min\{\delta,\varepsilon\})$ one has $\mu(t)=g_s\lambda(t)$. With this in mind let us look at the map
\begin{align*}
\Lambda:(t,t+\min\{\delta,\varepsilon\})&\to V\times V\\
s&\mapsto\Big(\lambda(s),\frac{\lambda(s)-\lambda(t)}{s-t}\Big)\,.
\end{align*}
Given $s$ arbitrary from the domain of $\Lambda$ we compute
\begin{equation}\label{eq:lambda_bounded}
\Lambda(s)=\Big(\lambda(s),\frac{\lambda(s)-\lambda(t)}{s-t}\Big)
=g_s^{-1}\underbrace{\Big(\mu(s),\frac{\mu(s)-\mu(t)}{s-t}\Big)}_{\to (\mu(t),\mu'(t))
\text{ as }s\to t^+}\,.
\end{equation}
As $\Gamma$ is finite the set of possible accumulation points
$
\{g\cdot(\mu(t),\mu'(t)):g\in\Gamma\}=D\xi(t)
$
of $\Lambda(s)$ (as $s\to t^+$) is finite. But because $\Lambda$ is bounded by \eqref{eq:lambda_bounded} and continuous, its cluster set at any (locally connected) boundary point of its domain is either a point, or a continuum~\cite[p.~2]{Collingwood66}, meaning it has exactly one accumulation point, denoted by $g_+\cdot(\mu(t),\mu'(t))$. Similarly one obtains some $g_-$, and this concludes the proof.
\end{proof}

\begin{remark}
Note that:
\begin{enumerate}[(i)]
\item In general $g_+\neq g_-$ so the left and right derivative of $\lambda$ at $t$ need not coincide.
\item In Lemma~\ref{lemma:deriv-well-def}~\eqref{it:left-right-derivs} it does not suffice for $\lambda$ to be continuous only at the point of interest $t$. For this consider $\mathbb R/\mathbb Z_2$ (cf.~Example~\ref{ex:R_Z2}) and the function $\xi:(-1,1)\to\mathbb R/\mathbb Z_2$, $t\mapsto[t]=|t|$, so $\xi$ effectively describes a reflection of a 1D motion at the origin. Obviously, $\mu:(-1,1)\to\mathbb R$, $t\mapsto t$ is a lift of $\xi$ which is differentiable (at the origin). Now an example of a lift of $\xi$ which is continuous at $0$ but does not admit a left- or right-derivative is given by
\begin{align*}
\lambda:(-1,1)&\to\mathbb R/\mathbb Z_2\\
t&\mapsto\begin{cases}
t&t\in\mathbb Q\\
-t&t\in\mathbb R\setminus\mathbb Q
\end{cases}\,.
\end{align*}
Continuity in $t=0$ is as evident as the fact that neither $\lim_{t\to 0^+}\frac{\lambda(t)-\lambda(0)}{t}$ nor $\lim_{t\to 0^-}\frac{\lambda(t)-\lambda(0)}{t}$ exist because they both have the accumulation points $1$ and $-1$.
\end{enumerate}
\end{remark}


The following result shows that if a path $\xi:I\to V/\Gamma$ admits local $C^1$-lifts, then these lifts can be stitched together to form a global $C^1$-lift. Note that this lift need not be unique, not even up to global group action. (e.g. $\xi:(-1,1)\to\R/\Z_2,\,x\mapsto [x^2]$ has four $C^1$-lifts).

\begin{lemma}
\label{lemma:local-global-C1}
Let $V$ be a finite dimensional real vector space, and $\Gamma$ a finite group acting linearly on $V$. Let $I$ be an open interval and $(I_j)_{j\in J}$ be an arbitrary family of open intervals whose union is $I$. Given $\xi:I\to V/\Gamma$ the following statements hold:
\begin{enumerate}[(i)]
\item \label{it:glue-cont} If $\xi$ admits a continuous lift on each interval $I_j$, then $\xi$ has a continuous lift $\lambda:I\to V$.
\item \label{it:glue-diff} If $\xi$ admits a differentiable lift on each interval $I_j$, then $\xi$ has a differentiable lift $\lambda:I\to V$.
\item \label{it:glue-C1} If $\xi$ admits a $C^1$-lift on each interval $I_j$, then $\xi$ has a $C^1$-lift $\lambda:I\to V$.
\end{enumerate}
\end{lemma}

\begin{proof}
We will only prove~\eqref{it:glue-C1}, as the simpler cases~\eqref{it:glue-cont} and~\eqref{it:glue-diff} can be shown analogously. 
First we will show that for every two open intervals with non-empty intersection we can ``glue'' the corresponding lifts together to a lift on their union. 
Let $j_1,j_2\in J$ such that $ I_{j_1}\cap I_{j_2}\neq\emptyset$ and neither interval fully contains the other.
Then (w.l.o.g.) $I_{j_1}\cup I_{j_2}=\{t\in I_{j_1}:t\leq t_0\} \cup \{t\in I_{j_2}:t\geq t_0\}$ for an arbitrary but fixed $t_0\in I_{j_1}\cap I_{j_2}$. Let $\lambda_1$ and $\lambda_2$ be $C^1$-lifts of $\xi$ on $I_{j_1},I_{j_2}$, respectively. Since their projections coincide on the open neighborhood $I_{j_1}\cap I_{j_2}$ of $t_0$, by well-definedness of the derivative we have
$$
D\pi(\lambda_1(t_0),\lambda_1'(t_0)) = D\xi(t_0) = D\pi (\lambda_2(t_0),\lambda_2'(t_0)) \in T(V/\Gamma)\,.
$$
Thus one finds $ g\in\Gamma$ such that $\lambda_1(t_0) = ( g\cdot\lambda_2)(t_0)$ and $\lambda_1'(t_0) = (g\cdot\lambda_2)'(t_0)= g\cdot\lambda_2'(t_0)$. Then the new path
\begin{align*}
\lambda_0:I_{j_1}\cup I_{j_2}\to V,\quad
t\mapsto \begin{cases}
\lambda_1(t) & t\leq t_0,\\
 g\cdot\lambda_2(t) & t> t_0,
\end{cases}
\end{align*}
is $C^1$ and hence $\lambda_0$ is a $C^1$-lift of $\xi|_{I_{j_1}\cup I_{j_2}}$.

Now consider a finite subset $J'\subseteq J$ such that $I':=\bigcup_{j\in J'} I_j$ is connected. Then we can construct a $C^1$ lift on $I'$ as follows. Choose two elements $j_1,j_2$ of $J'$ such that $I_{j_1}\cap I_{j_2}\neq\emptyset$, which exist since $I'$ is connected. If one interval contains the other, discard the smaller one, otherwise replace both intervals with their union, on which we can construct a $C^1$ list by gluing as above. We continue doing this until only one interval is left, namely $I'$. This also implies that we can construct a $C^1$ lift on any compact set in $I$.

To construct a $C^1$ lift on the entire open interval $I$, consider first two non-empty closed intervals $K_1,K_2$ such that $K_1$ lies in the interior of $K_2$. Given a $C^1$ lift $\lambda_1$ on $K_1$, we want to find a $C^1$ lift $\lambda_2$ on $K_2$ which is an extension of the former. For this let $D_1,D_2$ be the connected components of $K_2\setminus\operatorname{int}(K_1)$. As above we can find $C^1$ lifts on the closed intervals $D_1$ and $D_2$ and glue them as above to the given lift $\lambda_1$ on $K_1$, without modifying $\lambda_1$. Finally we extend this idea to a compact exhaustion of $I$ by closed intervals $K_i$ for $i\in\N$ using induction. This yields $C^1$ lifts $\lambda_i$ on $K_i$ satisfying $\lambda_{i+1}|_{K_i}=\lambda_i$. Hence we may define a lift $\lambda:I\to V$ by $\lambda|_{K_i}:=\lambda_i$. Then $\lambda$ is clearly $C^1$, as desired.
\end{proof}

The previous result can be further strengthened by only assuming that $\xi$ has the corresponding property in each point---this will be the main result of this section. For this we, given $\xi\in V/\Gamma$, have to define the \textit{degeneracy} of $\xi$ as the size of the stabilizer of any lift of $\xi$, that is, as the number $\degen_\xi:=|\Gamma|/|\xi|$ where $|\xi|$ is the cardinality of $\xi$ when taken as a subset of $V$.

\begin{proposition}
\label{prop:orbifold-C1-lift}
Let $V$ be a finite dimensional real vector space, $\Gamma$ a finite group acting linearly on $V$, and $I$ an open interval. Given $\xi:I\to V/\Gamma$ the following statements hold:
\begin{enumerate}[(i)]
\item \label{it:glue-cont2} If $\xi$ is continuous, then $\xi$ has a continuous lift $\lambda:I\to V$.
\item \label{it:glue-diff2} If $\xi$ is differentiable, then $\xi$ has a differentiable lift $\lambda:I\to V$.
\item \label{it:glue-C1-2} If $\xi$ is $C^1$, then $\xi$ has a $C^1$-lift $\lambda:I\to V$.
\end{enumerate}
\end{proposition}

\begin{proof}
By Lemma~\ref{lemma:local-global-C1} it suffices to show that for every point $t_0\in I$ there exists an open interval $I'\subseteq I$ containing $t_0$ on which a lift with the corresponding property exists.  \smallskip

\eqref{it:glue-cont2} \& \eqref{it:glue-diff2}: We proceed by induction on the degeneracy $\degen_{\xi(t_0)}$ of $\xi(t_0)$. First assume that $\degen_{\xi(t_0)}=1$. Let $\tilde\lambda$ be any lift of $\xi(t_0)$ and consider a slice $U$ about $\tilde\lambda$. Then the restriction $\pi|_U$ of the quotient map is a homeomorphism and hence, for an open interval $I'$ containing $t_0$ and such that $\xi(t)\in\pi(U)$ for $t\in I'$, the function $\lambda:=(\pi|_U)^{-1}\circ\xi|_{I'}$ is a continuous lift of $\xi|_{I'}$. For~\eqref{it:glue-diff2} we additionally need to show that at every $t\in I'$ the lift is differentiable. For this consider any other lift $\mu:I'\to V$ with $\mu(t)\in U$ which is differentiable at $t$. In a small enough neighborhood of $t$, $\mu$ takes values in $U$ and hence coincides with $\lambda$, which is therefore differentiable at $t$.

Now assume that $\degen_{\xi(t_0)}>1$ and that for any $t$ with $\degen_{\xi(t)}<\degen_{\xi(t_0)}$ there exists a continuous resp. differentiable lift in a neighborhood of $t$. We will now show that there is a continuous resp. differentiable lift in a neighborhood of $t_0$. Again let $\tilde\lambda$ be any lift of $\xi(t_0)$ and let $U$ be a slice about $\tilde\lambda$. Let $I'\subseteq I$ be an open interval containing $t_0$ such that $\xi(t)\in\pi(U)$ for all $t\in I'$. Let $J\subset I'$ be the open subset on which $\degen_{\xi(t)}<\degen_{\xi(t_0)}$ for all $t\in J$. Hence $J$ is an at most countable union of open intervals $J_k$, and by the induction hypothesis and using Lemma~\ref{lemma:local-global-C1}~\eqref{it:glue-cont} resp.~\eqref{it:glue-diff} there exists a continuous resp. differentiable lift on each $J_k$ taking values in $U$, which we denote by $\lambda_k:J_k\to U$. Clearly $g_k\lambda_k$ is also a differentiable lift for all $g_k\in\Gamma_\lambda$ and still lies in $U$. Note that for all $t\in I'\setminus J$ there exists a unique lift $\mu_t$ in $U$. Hence we can define the function
\begin{align*}
\lambda:I'\to U,\quad\lambda(t)=\begin{cases}g_k\lambda_k(t) &t\in J_k\\ \mu_t &t\in I'\setminus J,\end{cases}
\end{align*}
where the $g_k$ will be determined later. First we show continuity. For $t\in J$ this is clear by construction, hence we consider $t\in I'\setminus J$. Let $t_n\to t$ be a sequence in $I'$ and let $U'\subseteq U$ be any slice about $t$. Note that if $gU'\subseteq U$, the $gU'=U'$. Since the quotient map $\pi$ is open, $\pi(U')$ is an open neighborhood of $\xi(t)$. By continuity of $\xi$ this means that for $n$ large enough, $\xi(t_n)\in \pi(U')$ and hence $\lambda(t_n)\in U'$. Note that this argument only uses that $\lambda$ is a lift contained in $U$.

Next we show differentiability. By Lemma~\ref{lemma:deriv-well-def}~\eqref{it:left-right-derivs} $\lambda$ admits left and right derivatives at each point and it remains to show that they agree. For $t\in J$ this is clear, hence consider $t\in I'\setminus J$. If $t$ is an isolated point of $I'\setminus J$ connecting two intervals $J_{k_1}$ and $J_{k_2}$, then one can choose $g_{k_1}$ and $g_{k_2}$ such that $\lambda$ will be differentiable at $t$. Let $P$ be the set of all isolated points of $I'\setminus J$ and let consider the open set $J\cup P$. This is again an at most countable union of open intervals $J_k'$ and each $J_k'$ contains at most countably many intervals $J_k$ joint at their boundary points which lie in $P$. On each $J_k'$ we can then define $\lambda$ such that it is differentiable. If $t$ is not an isolated point of $I'\setminus J$ then there exists a sequence $t_n$ in $I'\setminus J$ converging to $t$ such that $t_n\leq t$ for all $n$ or $t_n\geq t$ for all $n$. Without loss of generality, assume that $t_n\geq t$ for all $n$. Hence the right derivative at $t$ is invariant under the stabilizer $\Gamma_{\lambda(t)}$ and thus it is uniquely determined. This implies that it coincides with the left derivative and $\lambda$ is differentiable at $t$.

\eqref{it:glue-C1-2}: By~\eqref{it:glue-diff2} there exists a differentiable lift $\lambda:I\to V$ of $\xi$. We claim that $\lambda$ must be $C^1$. Let $\phi_i$ be a basis of the dual space $V^*$. We show that each $\phi_i\circ\lambda$ is $C^1$ on $I$. Certainly $\phi_i\circ\lambda$ is differentiable on $I$. Assume that $\phi_i\circ\lambda'$ is discontinuous at $t_0\in I$. Thus there is some $\epsilon>0$ such that for every neighborhood $I'$ of $t_0$ there are $t_1,t_2\in I'$ such that $|\phi_i\circ\lambda'(t_1)-\phi_i\circ\lambda'(t_2)|>\epsilon$ and one can show that all values between $\phi_i\circ\lambda'(t_1)$ and $\phi_i\circ\lambda'(t_2)$ are taken, see~\cite[p.~114]{Kato80}. Now let $D_\delta$ be a disk in $V$ of radius $\delta>0$ centred at $\lambda'(t_0)$. Then, since $\xi'$ is continuous, there is a neighborhood $J_\delta$ of $t_0$ such that $\lambda'(t) \in \Gamma D_\delta$ for all $t\in J_\delta$. Choosing $\delta$ small enough we can ensure that $\phi_i(\Gamma D_\delta)$ does not contain an interval of length $\epsilon$, which yields the desired contradiction.
\end{proof}

\begin{remark}
Due to this result one could equivalently define a path $\xi:I\to V/\Gamma$ to be $C^1$ if it admits local $C^1$ lifts, which is more in line with the original definition of smooth functions on orbifolds.
\end{remark}

Finally we give two simple results that will be used in the measurable and analytic diagonalizations respectively.

\begin{lemma} \label{lemma:orb-meas}
Let $\Omega$ be a measurable space, and let $\lambda,\mu:\Omega\to V$ be measurable (where $V$ is endowed with its Borel $\sigma$-algebra). Assume that $\pi\circ\lambda=\pi\circ\mu$. Then there is some measurable function $\gamma:\Omega\to\Gamma$ such that $\lambda=\gamma\cdot\mu$.
\end{lemma}

\begin{proof}
Let $\gamma_i$ for $i=1,\ldots,m$ be an enumeration of $\Gamma$. Define the sets $\Omega_1=\{\omega\in\Omega:\lambda(\omega)=\gamma_1\cdot\mu(\omega)\}$ and iteratively $\Omega_i=\{\omega\in\Omega:\lambda(\omega)=\gamma_i\cdot\mu(\omega)\}\setminus\bigcup_{j=1}^{i-1}\Omega_j$ for $i>1$. These sets are measurable and form a partition of $\Omega$. Now define $\gamma(\omega)=\gamma_i$ for $\omega\in\Omega_i$. This $\gamma$ is measurable and satisfies the desired condition by construction.
\end{proof}

\begin{lemma} \label{lemma:analytic-unique}
Let $\lambda:I\to V$ be real analytic. Then any real analytic path $\mu$ satisfying $\pi\circ\mu=\pi\circ\lambda$ satisfies $\mu=g\cdot\lambda$ for some $g\in\Gamma$.
\end{lemma}

\begin{proof}
Consider the size of the stabilizer $|\Gamma_{\mu(t)}|$ as a function of $t$. Let $t_0\in I$ be a point where this value is minimal. Then there is an open interval $J\subseteq I$ containing $t_0$ on which the stabilizer is constant. Hence if $\mu(t_0)=g\cdot\lambda(t_0)$, then in a neighborhood $J'$ of $t_0$ it holds by Lemma~\ref{lemma:deriv-well-def}~\eqref{it:cont-lifts} that $h(t)\cdot\mu(t)=g\cdot\lambda(t)$ for some $h(t)\in\Gamma_{\mu(t_0)}$. Hence on $J\cap J'$ it holds that $\mu(t)=h(t)\mu(t)=g\cdot\lambda(t)$. Since both paths are real analytic, they coincide by the identity theorem.
\end{proof}

\section{Resolvent Computations}
\label{app:computations}

\begin{lemma} \label{lem:resolvents}
Let $V$ be a finite dimensional complex vector space and $A\in\mathfrak{gl}(V)$ be an operator with semisimple eigenvalues (or, equivalently, with vanishing eigen-nilpotents). 
Moreover, let $\mathbb{D}_{\varepsilon}$ be a small disc around the eigenvalue $\lambda_k$ of $A$ which does not contain any other eigenvalue $\lambda_l$, $l \neq k$. Then the identity
\begin{equation} 
\frac{1}{2\pi \mathrm{i} }\int_{\partial\mathbb{D}_{\varepsilon}} (\lambda-\mu)^{-1} R\big(\lambda, A\big)\mathrm{d} \lambda
=
\begin{cases}
(\lambda_k-\mu)^{-1}P_k & \text{for $\mu \not\in \overline{\mathbb{D}}_{\varepsilon}$}\,,\\[2mm]
\sum_{\substack{l=1\\l \neq k}}^N(\mu-\lambda_l)^{-1}P_l& \text{for $\mu \in \mathbb{D}_{\varepsilon}$}\,,
\end{cases}
\end{equation}
is fulfilled, where $R\big(\lambda, A \big) = \big(\lambda - A \big)^{-1}$ denotes the resolvent of $A$ and $P_k,P_l$ stand for the corresponding eigenprojections.
\end{lemma}

\begin{proof}
The equality can easily be obtained via the representation $A = \sum_{l=1}^N\lambda_l P_l$. It follows
\begin{equation*}
\begin{split}
    \frac{1}{2\pi \mathrm{i} } & \int_{\partial\mathbb{D}_{\varepsilon}} (\lambda-\mu)^{-1} R\big(\lambda, A\big)\mathrm{d} \lambda
    = \sum_{l=1}^N \left(\frac{1}{2\pi \mathrm{i} }\int_{\partial\mathbb{D}_{\varepsilon}} (\lambda-\mu)^{-1}(\lambda-\lambda_l)^{-1}\mathrm{d} \lambda\right) P_l\\
    & = \begin{cases}
      (\lambda_k-\mu)^{-1}P_k & \text{for $\mu \not\in \overline{\mathbb{D}}_{\varepsilon}$}\,,\\[2mm]
      \sum_{l=1 \atop l \neq k}^N(\mu - \lambda_l)^{-1}P_l& \text{for $\mu \in \mathbb{D}_{\varepsilon}$}\,,
      \end{cases}
    \end{split}
  \end{equation*}
  as desired.
\end{proof}

\begin{proof}[Proof of Lemma~\ref{lemma:analytic-diag-V}]
Let us consider the right-hand side of~\eqref{eq:holomorphic-vf}. 
By Kato~\cite[Ch.~II, Sec.~§1.4 \& Thm.~1.10]{Kato80} we know that the corresponding eigenprojections $P_k(z)$ can be chosen analytically on $\dot{\mathbb{D}}_r$ with analytic extension to $t_0$. 
Moreover, $P_k(z)$ has the well-known integral representation\footnote{See for instance~\cite[p.~67]{Kato80}, but be aware that Kato defines the resolvent to be $R(\lambda,A)=(A-\lambda)^{-1}$ which explains the additional minus sign in his integral representation of the eigenprojection.}
\begin{equation*}
P_k(z) = \frac{1}{2\pi \mathrm{i} }\int_{\partial\mathbb{D}_{\varepsilon_k}} R\Big(\lambda, A_c(z)\Big)\mathrm{d} \lambda\,,
\end{equation*}
where $\varepsilon_k > 0$ has to be chosen such that $\mathbb{D}_{\varepsilon_k}$ contains only the $k$-th\footnote{Note that the numbering of the eigenvalues does not matter due to the summation on the right-hand side of~\eqref{eq:holomorphic-vf}.}
  eigenvalue of $A_c(z)$. Hence we obtain
  \begin{equation*}
    \begin{split}
      & \big[\dot{P}_k(z),P_k(z)\big] \\
      & = \Bigg[\frac{1}{2\pi \mathrm{i} }\int_{\partial\mathbb{D}_{\varepsilon_k}} R\Big(\lambda,A_c(z)\Big)\dot{A}_c(z)R\Big(\lambda,A_c(z)\Big)\mathrm{d} \lambda,
      \frac{1}{2\pi \mathrm{i} }\int_{\partial\mathbb{D}_{\varepsilon_k}} R\Big(\lambda, A_c(z)\Big)\mathrm{d} \lambda\Bigg]\\
      & = \frac{1}{(2\pi \mathrm{i})^2} \int_{\partial\mathbb{D}_{\varepsilon'_k}}\int_{\partial\mathbb{D}_{\varepsilon_k}}
      \bigg[R\Big(\lambda,A_c(z)\Big)\dot{A}_c(z)R\Big(\lambda,A_c(z)\Big), R\Big(\mu, A_c(z)\Big)\bigg] \mathrm{d} \lambda \mathrm{d} \mu\,.
    \end{split}
  \end{equation*}
  with $\varepsilon'_k > 0$ slightly larger than $\varepsilon_k > 0$. One could also choose $\varepsilon'_k > 0$ slightly
  smaller than $\varepsilon_k > 0$. Now the standard resolvent identity
  \begin{equation*}
    R(\lambda,A)R(\mu,A) = (\mu - \lambda)^{-1}\big(R(\lambda,A) - R(\mu,A)\big)
  \end{equation*}
  yields
  \begin{equation*}
    \begin{split}
      & \big[\dot{P}_k(z),P_k(z)\big] \\
      = & \frac{1}{(2\pi \mathrm{i})^2} \int_{\partial\mathbb{D}_{\varepsilon'_k}}\int_{\partial\mathbb{D}_{\varepsilon_k}}
      \bigg[R\Big(\lambda,A_c(z)\Big)\dot{A}_c(z)R\Big(\lambda,A_c(z)\Big), R\Big(\mu, A_c(z)\Big)\bigg] \mathrm{d} \lambda \mathrm{d} \mu\\
      = & \frac{1}{(2\pi \mathrm{i})^2} \int_{\partial\mathbb{D}_{\varepsilon'_k}}\int_{\partial\mathbb{D}_{\varepsilon_k}}
      R\Big(\lambda,A_c(z)\Big)\dot{A}_c(z)R\Big(\lambda,A_c(z)\Big) R\Big(\mu, A_c(z)\Big) \mathrm{d} \lambda \mathrm{d} \mu\\
      & - \frac{1}{(2\pi \mathrm{i})^2} \int_{\partial\mathbb{D}_{\varepsilon'_k}}\int_{\partial\mathbb{D}_{\varepsilon_k}}
      R\Big(\mu, A_c(z)\Big)R\Big(\lambda,A_c(z)\Big)\dot{A}_c(z)R\Big(\lambda,A_c(z)\Big) \mathrm{d} \lambda \mathrm{d} \mu\\
      = & \frac{1}{(2\pi \mathrm{i})^2} \int_{\partial\mathbb{D}_{\varepsilon'_k}}\int_{\partial\mathbb{D}_{\varepsilon_k}}
      R\Big(\lambda,A_c(z)\Big)\dot{A}_c(z)(\mu - \lambda)^{-1}\Big(R\Big(\lambda,A_c(z)\Big) - R\Big(\mu, A_c(z)\Big)\Big) \mathrm{d} \lambda \mathrm{d} \mu\\
      & - \frac{1}{(2\pi \mathrm{i})^2} \int_{\partial\mathbb{D}_{\varepsilon'_k}}\int_{\partial\mathbb{D}_{\varepsilon_k}}
      (\lambda - \mu)^{-1}\Big(R\Big(\mu,A_c(z)\Big) - R\Big(\lambda, A_c(z)\Big)\Big)\dot{A}_c(z)R\Big(\lambda,A_c(z)\Big) \mathrm{d} \lambda \mathrm{d} \mu\\
      = & \frac{1}{(2\pi \mathrm{i})^2} \int_{\partial\mathbb{D}_{\varepsilon'_k}}\int_{\partial\mathbb{D}_{\varepsilon_k}}
      (\lambda - \mu)^{-1}R\Big(\lambda,A_c(z)\Big)\dot{A}_c(z)R\Big(\mu, A_c(z)\Big)\mathrm{d} \lambda \mathrm{d} \mu\\
      & - \frac{1}{(2\pi \mathrm{i})^2} \int_{\partial\mathbb{D}_{\varepsilon'_k}}\int_{\partial\mathbb{D}_{\varepsilon_k}}
      (\lambda - \mu)^{-1}R\Big(\mu,A_c(z)\Big)\dot{A}_c(z)R\Big(\lambda,A_c(z)\Big) \mathrm{d} \lambda \mathrm{d} \mu\,.
    \end{split}
  \end{equation*}
  First we investigate the second last term of the above identity
  \begin{equation*}
    \begin{split}
      & \frac{1}{(2\pi \mathrm{i})^2} \int_{\partial\mathbb{D}_{\varepsilon'_k}}\int_{\partial\mathbb{D}_{\varepsilon_k}}
      (\lambda - \mu)^{-1}R\Big(\lambda,A_c(z)\Big)\dot{A}_c(z)R\Big(\mu, A_c(z)\Big)\mathrm{d} \lambda \mathrm{d} \mu\\
      & = \frac{1}{2\pi \mathrm{i}} \int_{\partial\mathbb{D}_{\varepsilon'_k}}\bigg(\frac{1}{2\pi \mathrm{i}}\int_{\partial\mathbb{D}_{\varepsilon_k}}
      (\lambda - \mu)^{-1}R\Big(\lambda,A_c(z)\Big) \mathrm{d} \lambda \bigg) \dot{A}_c(z)R\Big(\mu, A_c(z)\Big) \mathrm{d} \mu\\
    \end{split}
  \end{equation*}
  Due to Lemma~\ref{lem:resolvents} we conclude
  \begin{equation*}
    \begin{split}
      \frac{1}{2\pi \mathrm{i}}\int_{\partial\mathbb{D}_{\varepsilon_k}} (\lambda - \mu)^{-1}R\Big(\lambda,A_c(z)\Big) \mathrm{d} \lambda = (\lambda_k - \mu)^{-1} P_k(z)\\
    \end{split}
  \end{equation*}
  and thus
  \begin{equation*}
    \begin{split}
      & \frac{1}{(2\pi \mathrm{i})^2} \int_{\partial\mathbb{D}_{\varepsilon'_k}}\int_{\partial\mathbb{D}_{\varepsilon_k}}
      (\lambda - \mu)^{-1}R\Big(\lambda,A_c(z)\Big)\dot{A}_c(z)R\Big(\mu, A_c(z)\Big)\mathrm{d} \lambda \mathrm{d} \mu\\
      & = \frac{1}{2\pi \mathrm{i}} \int_{\partial\mathbb{D}_{\varepsilon'_k}} (\lambda_k - \mu)^{-1} P_k(z)\dot{A}_c(z)R\Big(\mu, A_c(z)\Big) \mathrm{d} \mu\\
      & = \sum_{l=1 \atop l\neq k}^N(\lambda_l-\lambda_k)^{-1} P_k(z)\dot{A}_c(z)P_l(z)\,.
    \end{split}
  \end{equation*}
  Similarly, we conclude for the last term
  \begin{equation*}
    \begin{split}
      & \frac{1}{(2\pi \mathrm{i})^2} \int_{\partial\mathbb{D}_{\varepsilon'_k}}\int_{\partial\mathbb{D}_{\varepsilon_k}}
      (\lambda - \mu)^{-1}R\Big(\mu,A_c(z)\Big)\dot{A}_c(z)R\Big(\lambda, A_c(z)\Big)\mathrm{d} \lambda \mathrm{d} \mu\\
      & = \sum_{l=1 \atop l\neq k}^N(\lambda_l-\lambda_k)^{-1} P_l(z)\dot{A}_c(z)P_k(z)\,.
    \end{split}
  \end{equation*}
  Summing up, we obtain
  \begin{align*}
      \frac{1}{2}\sum_{k = 1}^N & [\dot{P_k}(z),P_k(z)] = \\
      & = \frac{1}{2}\bigg(\sum_{k=1}^N \sum_{l=1 \atop l\neq k}^N(\lambda_l-\lambda_k)^{-1} P_k(z)\dot{A}_c(z)P_l(z)
      - \sum_{k=1}^N \sum_{l=1 \atop l\neq k}^N(\lambda_l-\lambda_k)^{-1} P_l(z)\dot{A}_c(z)P_k(z)\bigg)\\
      & = \frac{1}{2}\bigg(\sum_{k=1}^N \sum_{l=1 \atop l\neq k}^N(\lambda_l-\lambda_k)^{-1} P_k(z)\dot{A}_c(z)P_l(z)
      - \sum_{l=1}^N \sum_{k=1 \atop l\neq k}^N(\lambda_k-\lambda_l)^{-1} P_k(z)\dot{A}_c(z)P_l(z)\bigg)\\
      & = \sum_{k=1}^N \sum_{l=1 \atop l\neq k}^N(\lambda_l-\lambda_k)^{-1} P_k(z)\dot{A}_c(z)P_l(z).
  \end{align*}
Using~\eqref{eq:adinv} this concludes the proof.
\end{proof}


\begin{thebibliography}{10}
%

\bibitem{Adem07}
A.~Adem, J.~Leida, and Y.~Juan.
\newblock {\em {Orbifolds and Stringy Topology}}.
\newblock Cambridge University Press, Cambridge, 2007.

\bibitem{Alexandrino15}
M.~Alexandrino and R.~Bettiol.
\newblock {\em {Lie Groups and Geometric Aspects of Isometric Actions}}.
\newblock Springer, Cham, 2015.

\bibitem{Baumgaertel85}
H.~Baumg{\"a}rtel.
\newblock {\em {Analytic Perturbation Theory for Matrices and Operators}}.
\newblock Birkh{\"a}user, Basel, 1985.


\bibitem{BunseGerstner91}
A.~Bunse-Gerstner, R.~Byers, V.~Mehrmann, and N.~Nichols.
\newblock {Numerical Computation of an Analytic Singular Value Decomposition of
  a Matrix Valued Function}.
\newblock {\em Numer. Math.}, \textbf{60} (1991), 1--39.

\bibitem{Collingwood66}
E.~Collingwood and A.~Lohwater.
\newblock {\em {The Theory of Cluster Sets}}.
\newblock Cambridge University Press, Cambridge, 1966.

\bibitem{Dadok85}
J.~Dadok.
\newblock {Polar Coordinates Induced by Actions of Compact Lie Group}.
\newblock {\em Trans. Amer. Math. Soc.}, \textbf{288} (1985), 125--137.

\bibitem{doCarmo92}
M.~do Carmo.
\newblock {\em {Riemannian Geometry}}.
\newblock Birkh\"auser, Basel, 1992.


\bibitem{Helgason78}
S.~Helgason.
\newblock {\em {Differential Geometry, Lie Groups, and Symmetric Spaces}}.
\newblock Academic Press, New York, 1978.

\bibitem{Humphreys72}
J.~Humphreys.
\newblock {\em {Introduction to Lie Algebras and Representation Theory}}.
\newblock Springer, New York, 1972.

\bibitem{Kato80}
T.~Kato.
\newblock {\em {Perturbation Theory for Linear Operators}}.
\newblock Springer, Berlin, 1980.

\bibitem{Kleinsteuber}
M.~Kleinsteuber.
\newblock {Jacobi-type Methods on Semisimple Lie Algebras: a Lie Algebraic
  Approach to Numerical Linear Algebra}.
\newblock Dissertation, University of W\"urzburg, W\"urzburg, Germany. 2005.

\bibitem{Knapp02}
A.~Knapp.
\newblock {\em {Lie Groups Beyond an Introduction}}.
\newblock Birkh{\"a}user, Boston, 2nd edition, 2002.

\bibitem{KobNom96}
S.~Kobayashi and K.~Nomizu.
\newblock {\em {Foundations of Differential Geometry, Vols. I--II}}.
\newblock Wiley Interscience, New York, 1996.

\bibitem{Kriegl03}
A.~Kriegl and P.~Michor.
\newblock {Differentiable Perturbation of Unbounded Operators}.
\newblock {\em Math. Ann.}, \textbf{327} (2003), 191--201.

\bibitem{Lee13}
J.~Lee.
\newblock {\em {Introduction to Smooth Manifolds}}.
\newblock Springer, New York, 2nd edition, 2013.

\bibitem{Mackey05}
D.S.~Mackey, N.~Mackey, and D.M.~Dunlavy.
\newblock {Structure Preserving Algorithms for Perplectic Eigenproblems}. 
\newblock {\em Electronic J. Linear Algebra}, \textbf{13}, (2005), 10--39.

\bibitem{Mailybaev05}
A.~Mailybaev, O.~Kirillov, and A.~Seyranian.
\newblock {Coupling of Eigenvalues of Complex Matrices at Diabolic and Exceptional Points}.
\newblock {\em J. Phys. A Math. Theor.}, \textbf{38}, (2005), 1723--1740.


\bibitem{Quintana14}
Y.~Quintana and J.~Rodr{\'i}guez.
\newblock {Measurable Diagonalization of Positive Definite Matrices}.
\newblock {\em J. Approx. Theory}, \textbf{185} (2014), 91--97.

\bibitem{Rainer11}
A.~Rainer.
\newblock {Quasianalytic Multiparameter Perturbation of Polynomials and Normal
  Matrices}.
\newblock {\em Trans. Am. Math. Soc.}, \textbf{363} (2011), 4945--4977.

\bibitem{Rellich69}
F.~Rellich.
\newblock {\em {Pertubation Theory of Eigenvalue Problems}}.
\newblock Gordon and Breach, New York, 1969.

\bibitem{Satake56}
I.~Satake.
\newblock {On a Generalization of the Notion of Manifold}.
\newblock {\em Proc. Natl. Acad. Sci.}, \textbf{42}(6) (1956), 359--363.


\bibitem{Zalinescu02}
C.~Z\u{a}linescu.
\newblock {\em {Convex Analysis in General Vector Spaces}}.
\newblock World Scientific, Singapore, 2002.

\end{thebibliography}


\end{document}